\documentclass[reqno]{amsart}
\PassOptionsToPackage{linktocpage}{hyperref}
\usepackage{stackengine}
\stackMath
\newcommand\tsup[2][2]{%
 \def\useanchorwidth{T}%
  \ifnum#1>1%
    \stackon[-.5pt]{\tsup[\numexpr#1-1\relax]{#2}}{\scriptscriptstyle\sim}%
  \else%
    \stackon[.5pt]{#2}{\scriptscriptstyle\sim}%
  \fi%
}
 \theoremstyle{definition}
 
 \theoremstyle{remark}

 \numberwithin{equation}{section}

\usepackage{verbatim}
\usepackage{xcolor}
\usepackage{graphicx}
\usepackage{amssymb, amsmath, amsthm}
\usepackage{amsfonts}
\usepackage{amssymb}
\usepackage{float}
\usepackage{hyperref}
\usepackage{mathrsfs}
\usepackage{enumitem}

\newtheorem{theorem}{Theorem}

\newtheorem{corollary}[theorem]{Corollary}

\newtheorem{definition}[theorem]{Definition}

\newtheorem{lemma}[theorem]{Lemma}

\newtheorem{proposition}[theorem]{Proposition}
\newtheorem{remark}[theorem]{Remark}

\newcommand{\dis}{\displaystyle}
\newcommand{\thetalambda}{\theta^{(\lambda)}}
\newcommand{\thetakappa}{\theta^{(\kappa)}}

\newcommand{\bK}{\bar{K}_{\infty}}
\newcommand{\Dh}{\delta_{h}}
\newcommand{\dt}{\partial_t}

\newcommand{\dimf}{\text{\rm dim}_{f}}
\newcommand{\A}{\mathcal{G}^{0}}

\newcommand{\Gg}{\mathcal{G}^{\lambda}}
\newcommand{\Ggo}{\mathcal{G}^{\lambda_{0}}}

\newcommand{\Lh}{\mathcal{L}_{h}}
\newcommand{\K}{K_{\infty}}

\newcommand{\N}{\mathbb N}
\newcommand{\Tr}{\text{\rm Tr}}
\newcommand{\T}{\mathbb T}
\newcommand{\Z}{\mathbb Z}

\newcommand{\pin}{\pi^{\lambda}}

\newcommand{\pio}{\pi^{0}}

\newcommand{\U}{\mathbf{U}}
\newcommand{\Uc}{\mathbf{K}}

\newcommand{\R}{\mathbb R}

\newcommand{\dist}{\text{\rm dist}}

\newcommand{\intox}{\int_{\Omega}}

\numberwithin{equation}{section}
\numberwithin{theorem}{section}

\numberwithin{figure}{section}
\usepackage{scalerel,stackengine}
\stackMath
\newcommand\reallywidehat[1]{%
\savestack{\tmpbox}{\stretchto{%
  \scaleto{%
    \scalerel*[\widthof{\ensuremath{#1}}]{\kern.1pt\mathchar"0362\kern.1pt}%
    {\rule{0ex}{\textheight}}
  }{\textheight}%
}{2.4ex}}%
\stackon[-6.9pt]{#1}{\tmpbox}%
}
\begin{document}

%
%
%
%
%
%
%
%
%

\title[viscous forced active scalar]
 {Long time dynamics and anomalous dissipation of energy in viscous forced active scalar equations}

\author{Susan Friedlander}

\address{Department of Mathematics\\
University of Southern California}

\email{susanfri@usc.edu}

\author{Anthony Suen}

\address{Department of Mathematics and Information Technology\\
The Education University of Hong Kong}

\email{acksuen@eduhk.hk}

\date{}



\begin{abstract}
We study an abstract family of advection-diffusion equations within the framework of the fractional Laplacian. The system involves two independent diffusion parameters: one introduced via a damping operator acting on the scalar unknown and the other as the coefficient of the fractional Laplacian. We establish existence and convergence results in specific parameter regimes and limits. In particular, we demonstrate the absence of anomalous energy dissipation for long-time averaged solutions. Moreover, we investigate the long time dynamics and prove the existence of a unique global attractor. These results are then applied to two specific classes of active scalar equations in geophysical fluid dynamics, namely the surface quasigeostrophic equation and the magnetogeostrophic equation.
\end{abstract}

\maketitle
\tableofcontents
\section{Introduction}\label{introduction}

Active scalar equations are ubiquitous in nature, science and engineering. They are of great importance in practical applications ranging from fluid mechanics to atmospheric science, oceanography and geophysics. Various active scalar equations, such as drift-diffusion equations, surface quasi-geostrophic (SQG) equation and magnetogeostrophic (MG) equation, have been widely studied by different teams of researchers, yet due to the nonlinear and nonlocal nature of these equations, many fundamental issues in the rigorous mathematical analysis remain unsolved and highly challenging. In this present paper we study an abstract class of active scalar equations in $\Omega\times(0,\infty)$ with $\Omega=\mathbb{T}^d=[-\pi,\pi]^d$ and $d\in\{2,3\}$ of the following form
\begin{align}
\label{abstract active scalar eqn} \left\{ \begin{array}{l}
\partial_t\theta+u\cdot\nabla\theta+\lambda\mathcal{D}\theta+\kappa\Lambda^{\gamma}\theta=S, \\
u=u_j[\theta]=\partial_{x_i} T_{ij}[\theta],\theta(x,0)=\theta_0(x)
\end{array}\right.
\end{align}
where $\lambda\ge0$, $\kappa>0$, $\gamma\in(0,2]$ and $\Lambda:=\sqrt{-\Delta}$. Here $\theta_0$ is the initial datum and $S=S(x)$ is a given time-independent function that represents the forcing of the system. The damping operator $\mathcal{D}$ is given by
\begin{align}\label{def of damping}
\mathcal{D}=\Lambda+\mathbf{1}
\end{align}
with $\mathbf{1}$ being the identity map. We assume that
\begin{equation}\label{zero mean assumption on data and forcing}
\int_{\Omega}\theta_0(x)dx=\int_{\Omega}S(x)=0,
\end{equation}
and throughout this paper, we consider mean-zero (zero average) solutions. $T_{ij}$ is an operator which satisfies:
\begin{enumerate}
\item[A1.]  $\partial_{x_i}\partial_{x_j} T_{ij}f=0$ for any smooth functions $f$.
\item[A2.] $T_{ij}:L^\infty(\Omega)\rightarrow BMO(\Omega)$ is linear and bounded for all $i,j$.
\item[A3.] There exists a constant $C>0$ such that for all $1\le i,j\le 3$, $$|\widehat{T_{ij}}(k)|\le C|k|^{-1}, \qquad\forall k\neq0.$$
\item[A4.] For each $1\le i,j\le 3$, $\widehat{T_{ij}}(k)=0$ for $|k|=0$. 
\end{enumerate}
\begin{remark}
There are several remarks for the assumptions A1 to A4 as given above:
\begin{itemize}
\item A1 implies that $u=u[\theta]$ is divergence-free. Hence together with \eqref{zero mean assumption on data and forcing}, it immediately implies that $\theta$ obeys
\begin{equation}\label{zero mean assumption}
\int_{\Omega}\theta(x,t)dx=0,\qquad\forall t\ge0. 
\end{equation}
\item A2 implies that the drift velocity $u$ lies in the space $L^{\infty}_t BMO^{-1}_x$.
\item A3 implies that $u_j[\cdot]=\partial_{x_i} T_{ij}[\cdot]$ is a zero-order bounded operator in the sense that for any $s\ge0$ and $f\in L^p$ with $p>1$,
\begin{align}\label{zero order bounded for u when nu>0}
\|\Lambda^s u[f]\|_{L^p}\le C\|\Lambda^{s}f\|_{L^p},
\end{align}
where $C$ is a positive constant which depends on $p$ and $\Omega$ only.
\item A4 implies that $u$ has zero mean, which is consistent with $\theta$ also having zero mean.
\end{itemize}
\end{remark}

The motivation for addressing such a class of abstract family of active scalar equations mainly comes from several different physical systems, all of them take the form \eqref{abstract active scalar eqn} under particular parameter regimes. The critical surface quasi-geostrophic (SQG) equation is a two-dimensional example where the damping parameter $\lambda=0$, the thermal diffusion $\kappa>0$, the fractional power $\gamma=1$, and the relation between the velocity $u$ and the scalar field $\theta$ is given by the perpendicular Riesz transform \cite{CMT94}. It is worth noting that this is a singular integral operator of degree zero. Global well-posedness for the critical SQG equation was first proved in Kiselev, Nazarov and Volberg \cite{KFV07} and Caffarelli and Vasseur \cite{CV10}, and recently, there has been a considerable literature on the long time dynamics and regularity of the SQG equation including \cite{CIN25, CZ24, CLO25, JZ24, ZL25} and references therein.

Another model with related, but distinct, features is the three-dimensional magnetogeostrophic equation (MG) which was proposed by Moffatt and Loper as a model for magnetogeostrophic turbulence \cite{FV11a, FRV14, M08, ML94}. The MG equation is an example where the damping parameter $\lambda=0$, the thermal diffusion $\kappa\ge0$ and the fractional power $\gamma=2$. The mathematical properties of the MG equation have been determined in various settings of the diffusive parameters via an analysis of the Fourier multiplier symbol relating the velocity $u$ and the scalar field $\theta$. In a series of papers \cite{FS15, FS18, FS19}, the authors proved global existence of classical solutions to the forced MG equations and obtained convergences of solutions as various diffusive parameters vanish. Moreover, it was shown in \cite{FS21} that the MG equations possess global attractors with various interesting properties; also refer to the survey articles \cite{FRV14, FS23} for more details. 

The purpose of our current work is to investigate properties of the abstract system \eqref{abstract active scalar eqn} under assumptions A1--A4, with emphasis on the long time dynamics and convergence results in the context of the fractional Laplacian $\kappa\Lambda^{\gamma}$ and the damping parameter $\lambda$, and then we apply these results to physical models as mentioned above. Specifically, we examine the limiting behaviour of solutions as $\kappa$ or $\lambda$ vanishes, and two main studying directions are illustrated as follows:

\noindent{1.} With the presence of forcing, the long time behaviour of \eqref{abstract active scalar eqn} can be related to the study of {\it global attractors} or dynamical systems. There is a considerably rich literature on the topic of global attractors; see for example \cite{CCP12, CD14, CF85, HOR15, Ra02, R13, RP03} and the references therein. For the dissipative SQG equation, the global attractor has been addressed previously for the subcritical regime \cite{Ju05, WY13, Ze18} as well as the critical case \cite{CTV14, CZV16}. In the context of MG equation, the existence of compact global attractor in $L^2$ was proved in \cite{FS18}, the convergence and upper upper semicontinuity of the global attractors were later obtained in \cite{FS21}. With the presence of damping term $\lambda u$, recently the authors in \cite{GST25} showed the existence of a robust family of exponential attractors for the scalar-valued convective Cahn–Hilliard/Allen–Cahn (CH-AC) equation, and yet there is no known result on global attractors for active scalar equations with the damping operator $\lambda\mathcal{D}$. In this current study, we aim to fill this research gap by investigating the existence and convergence properties of global attractors for the abstract active scalar equations \eqref{abstract active scalar eqn} under the influence of $\lambda\mathcal{D}$.

\noindent{2.} We also propose to address the {\it average behaviour} of the energy dissipation rate given by
\begin{align*}
\varepsilon^{\kappa}=\kappa\lim_{t\to\infty}\frac{1}{t}\int_0^t\int|\nabla\theta^{(\kappa)}|^2dxds,
\end{align*}
where $\theta^{(\kappa)}$ are the solutions to the general active scalar equations \eqref{abstract active scalar eqn}. The number $\varepsilon^{\kappa}$ plays an important role in turbulence theory, in the case when $\lim_{\kappa\to0}\varepsilon^{\kappa}>0$, it suggests that there is still {\it anomalous dissipation} that refers to the dissipation of energy even though the limit inviscid equation conserves energy. Rigorous studies have been conducted on the phenomenon of anomalous dissipation in fluid mechanics \cite{Fr95, DF02, DN18, CCS23, EL24}. For example, in the context of Navier-Stokes equations in bounded domains with no-slip Dirichlet boundary conditions, the vanishing of energy dissipation rate \footnote{Here we replace $\theta^{(\kappa)}$ by the velocity $u^{(\kappa)}$ in the definition of $\varepsilon^{\kappa}$ for Navier-Stokes equations.} $\varepsilon^{\kappa}$ as $\kappa\to0$ is equivalent to the convergence of solutions of the Navier-Stokes equations to the corresponding solution of the Euler equations \cite{Ka83, Wa01}. When the forcing is absent, it was shown in \cite{CR07} that the limit $\lim_{\kappa\to0}\varepsilon^{\kappa}$ vanishes even for Leray weak solutions. On the other hand, it can also be shown that $\varepsilon^{\kappa}$ becomes infinite as $\kappa\to0$ when the forcing terms are chosen to be some spatially periodic eigenfunctions of the Stokes operator \cite{CTV14b}. Regarding scalar equations, when the damping operator $\lambda\mathcal{D}$ is enforced, it was proved in Constantin, Tarfulea and Vicol \cite{CTV14b} that the anomalous dissipation of energy for long time averaged solutions of the SQG with time independent deterministic forcing is indeed absent. In a very recent work, Armstrong and Vicol \cite{AV25} considered a class of {\it passive scalar equation} without forcing in which they constructed divergence-free velocity vector field $u\in C^0_t C^\alpha_x\cap C^\alpha_t C^0_x$ with $\alpha<\frac{1}{3}$ that is periodic in space and time, such that corresponding scalar
advection-diffusion equation exhibits anomalous dissipation of scalar variance for arbitrary $H^1$ initial data. These results demonstrate the substantial influence of forcing and initial data on the limiting behaviour of $\varepsilon^{\kappa}$. 

The rest of the paper is organised as follows. In Section~\ref{main results}, we state our main results on the forced active scalar equation \eqref{abstract active scalar eqn}, while in Section~\ref{preliminaries} we introduce some notations and useful formulae for later analysis. In Section~\ref{absence of anomalous dissipation section} the parameter $\lambda$ is taken to be positive and we investigate the anomalous dissipation for \eqref{abstract active scalar eqn}. Following the similar idea in the proof in \cite{CTV14b} for critical SQG, we show that there is no anomalous dissipation by introducing stationary statistical solution for the viscous $(\kappa>0)$ and inviscid $(\kappa=0)$ cases. We note that the results in Section~\ref{absence of anomalous dissipation section} are valid in the examples of the fractional SQG equation and MG equation. In contrast, in Section~\ref{Existence and convergence of Hs-solutions section}, when the viscosity $\kappa$ is kept to be positive but $\lambda$ is allowed to vary, we obtain global-in-time existence and convergence results as $\lambda\to0$ under an extra smoothing assumption on $T_{ij}$ that is consistent with the case of MG equations. In Section~\ref{long time behaviour and attractors section} we further study the long time behaviour of solutions to the abstract system when $\lambda\ge0$ and $\kappa>0$. We prove that the solution map associated with \eqref{abstract active scalar eqn} possesses a unique global attractor $\Gg$ in $H^1$. With the restriction that $\lambda>0$, we prove that $\Gg$ always has finite fractal dimension. In Section~\ref{Applications to magneto-geostrophic equations section} we apply the results proved in Section~\ref{long time behaviour and attractors section} to the specific example of the MG equation. We prove convergence of $\Gg$ as $\lambda$ goes to zero to the global attractor $\A$ in $H^1$ for the MG equation whose existence was demonstrated in \cite{FS21}.

\section{Main results}\label{main results}

The main results that we prove for the forced problem \eqref{abstract active scalar eqn} are stated in the following theorems, and they will be proved in Section~\ref{absence of anomalous dissipation section}, Section~\ref{Existence and convergence of Hs-solutions section} and Section~\ref{long time behaviour and attractors section}. These results will then be applied to the magnetogeostrophic active scalar equation which will be discussed in Section~\ref{Applications to magneto-geostrophic equations section}.

\begin{theorem}[Absence of anomalous dissipation of energy as $\kappa\to0$ when $\lambda>0$]\label{absence of anomalous dissipation thm}
Let $\theta_0,S\in L^\infty$, and assume that $\lambda>0$ and $\gamma\in(0,2]$ be fixed. Under the assumptions A1, A2, A3 and A4, for each $\kappa\ge0$, there exists a unique solution $\theta=\theta^{(\kappa)}(x,t)$ to \eqref{abstract active scalar eqn} satisfying
\begin{align}\label{absence of anomalous dissipation identity}
\lim_{\kappa\to0}\left(\limsup_{t\to\infty}\int_{0}^{t}\|\nabla\theta^{(\kappa)}(\cdot,s)\|_{L^2}^2ds\right)=0.
\end{align}
\end{theorem}

\begin{theorem}[$H^s$-convergence as $\lambda\rightarrow0$ when $\kappa>0$]\label{Hs convergence theorem}
Let $\kappa>0$ and $\gamma\in(0,2]$ be given in \eqref{abstract active scalar eqn}, and let $\theta_0,S\in C^\infty$ be the initial datum and forcing term respectively which satisfy \eqref{zero mean assumption on data and forcing}. Under the assumptions A1, A2, A3' and A4, if $\thetalambda$ and $\theta^{(0)}$ are smooth solutions to \eqref{abstract active scalar eqn} for $\lambda>0$ and $\lambda=0$ respectively, then 
\begin{align}\label{Hs convergence thm} 
\lim_{\lambda\rightarrow0}\|(\thetalambda-\theta^{(0)})(\cdot,t)\|_{H^s}=0,
\end{align}
for all $s\ge0$ and $t\ge0$.
\end{theorem}

\begin{theorem}[Existence of global attractors]\label{existence of global attractor theorem}
Let $S\in L^\infty\cap H^1$ be the forcing term. For $\lambda\ge0$, $\kappa>0$ and $\gamma\in(0,2]$, under the assumptions A1, A2, A3' and A4, let $\pin(t)$ be the solution operator for the initial value problem \eqref{abstract active scalar eqn} via
\begin{align*}
\pin(t): H^1\to H^1,\qquad \pin(t)\theta_0=\theta(\cdot,t),\qquad t\ge0.
\end{align*}
Then the solution map $\pin(t):H^1\to H^1$ associated to \eqref{abstract active scalar eqn} possesses a unique global attractor $\Gg$. In particular, if we assume that $\lambda$ and $\gamma$ satisfy either one of the following conditions:
\begin{align*}
\mbox{$\lambda\ge0$ and $\gamma\in[1,2]$};
\end{align*}
or
\begin{align*}
\mbox{$\lambda>0$ and $\gamma\in(0,2]$},
\end{align*}
then the global attractor $\Gg$ of $\pin(t)$ further enjoys the following properties:
\begin{itemize}
\item $\Gg$ is fully invariant, namely
\begin{align*}
\pin(t)\Gg=\Gg,\qquad \forall t\ge0.
\end{align*}
\item $\Gg$ is maximal in the class of $H^1$-bounded invariant sets.
\item $\Gg$ has finite fractal dimension.
\end{itemize}
\end{theorem}

\begin{remark}
There are several remarks for the main results as given above:
\begin{itemize} 
\item In Theorem~\ref{Hs convergence theorem} and Theorem~\ref{existence of global attractor theorem}, assumption A3 is replaced by A3' that will be given in Section~\ref{Existence and convergence of Hs-solutions section}.
\item Theorem~\ref{absence of anomalous dissipation thm} will be proved by constructing {\it stationary statistical solutions} to \eqref{abstract active scalar eqn} for both the cases when $\lambda>0$ and $\lambda=0$; refer to Section~\ref{absence of anomalous dissipation section} for the related definition and construction of stationary statistical solutions.
\item Regarding Theorem~\ref{Hs convergence theorem}, the existence of global-in-time $H^s$-solutions to \eqref{abstract active scalar eqn} when $\lambda\ge0$ will be given by Theorem~\ref{global-in-time wellposedness in Sobolev} in Subsection~\ref{existence of Hs solution subsec}, while the convergence of solutions as $\lambda\to0$ will be addressed in Subsection~\ref{convergence of Hs solution subsec}.
\item For the case when $\lambda\ge0$, $\kappa>0$ and $\gamma=2$, the regularity of the global attractor $\Gg$ as stated in Theorem~\ref{existence of global attractor theorem} can be further improved to $H^2$; refer to Remark~\ref{H2 global attractor rem} for a more detailed explanation. 
\end{itemize}
\end{remark}

\section{Preliminaries}\label{preliminaries}

We introduce the following notations and conventions:

\begin{itemize}
\item We write $u:=u[\theta]$, where $u$ is given by $u_j:=\partial_{x_i} T_{ij}[\theta]$. We also refer $u[\cdot]$ to an operator in the sense that $u_j[f]=\partial_{x_i} T_{ij}[f]$ for appropriate functions $f$.
\item To emphasise the dependence of solutions on $\kappa$ or $\lambda$, we sometimes write $\theta=\thetakappa$, $\theta=\theta^{(\lambda)}$, $u=u^{\kappa}:=u[\thetakappa]$ and $u=u^{(\lambda)}:=u[\theta^{(\lambda)}]$ for varying $\kappa,\lambda$.
\item As a consequence of the mean-zero setting, we may identify the identify the homogenous Sobolev spaces and the inhomogenous Sobolev spaces, and we denote these by $H^s$ for $s\in\R$. These spaces are the closure of mean-zero $C^\infty(\T^d)$ under the norm $\|\cdot\|_{H^s}=\|\Lambda^s(\cdot)\|_{L^2}$.
\item We define $\langle\cdot,\cdot\rangle$ to be the $L^2$-inner product on $\Omega$, that is
\begin{equation*}
\langle f,g\rangle:=\int_{\Omega}fgdx,
\end{equation*}
for any $f,g\in L^2$.
\item Regarding the constants used in this work, we have the following conventions:
\begin{itemize}
\item $C$ shall denote a positive and sufficiently large constant, whose value may change from line to line; 
\item $C$ is allowed to depend on the size of $\Omega$ when $\Omega=\mathbb{T}^d$ and other universal constants which are fixed throughout this work; 
\item In order to emphasise the dependence of $C$ on a certain quantity $Q$ we usually write $C_{Q}$ or $C(Q)$.
\end{itemize}
\item By abuse of notation, we denote by $\phi:\R^d\to\R$ the periodic extension to the whole space of a
$\T^d$ periodic function $\phi$ without further notice.
\end{itemize}

We recall the following Sobolev embedding inequalities from the literature (see for example Bahouri-Chemin-Danchin \cite{BCD11} and Ziemer \cite{Z89}): 
\begin{itemize}
\item Let $d\ge2$ be the dimension. There exists $C=C(d)>0$ such that
\begin{equation}
\|f\|_{L^\infty }\le C \|f\|_{W^{2,d} }.\label{L infty bound 1}
\end{equation}
\item For $q>d$, there exists $C=C(q)>0$ such that
\begin{equation}
\|f\|_{L^\infty }\le C \|f\|_{W^{1,q} }.\label{L infty bound 2}
\end{equation}
\item If $k>l$ and $k-\frac{d}{p}>l-\frac{d}{q}$, then there exists $C=C(k,l,d,p,q)>0$ such that
\begin{align}
\|f\|_{W^{l,q}}\le C\|f\|_{W^{k,p}}.\label{Kondrachov embedding theorem}
\end{align}
\item For $q>1$, $q'\in[q,\infty)$ and $\frac{1}{q'}=\frac{1}{q}-\frac{s}{d}$, if $\Lambda^s h\in L^q$, there exists $C=C(q,q',d,s)>0$ such that
\begin{align}\label{Sobolev inequality}
\|h\|_{L^{q'}}\le C\|\Lambda^s h\|_{L^{q}}.
\end{align}
\end{itemize}

We also recall the following product and commutator estimates: If $s>0$ and $p>1$, then for all $f,g\in H^s\cap L^\infty$, we have
\begin{align}\label{product estimate}
\|\Lambda^s(fg)\|_{L^p}\le C\Big(\|f\|_{L^{p_1}}\|\Lambda^s g\|_{L^{p_2}}+\|\Lambda^s f\|_{L^{p_3}}\|g\|_{L^{p_4}}\Big),
\end{align}
\begin{align}\label{commutator estimate}
\|\Lambda^s(fg)-f\Lambda^s(g)\|_{L^p}\le C\Big(\|\nabla f\|_{L^{p_1}}\|\Lambda^{s-1} g\|_{L^{p_2}}+\|\Lambda^s f\|_{L^{p_3}}\|g\|_{L^{p_4}}\Big),
\end{align}
where $\frac{1}{p}=\frac{1}{p_1}+\frac{1}{p_2}=\frac{1}{p_3}+\frac{1}{p_4}$, $p$, $p_2$, $p_3\in(1,\infty)$ and $p_1,p_4\in(1,\infty]$. 

\section{Stationary statistical solution and absence of anomalous dissipation when $\lambda>0$}\label{absence of anomalous dissipation section}

In this section, we address the anomalous dissipation for the forced active scalar equation \eqref{abstract active scalar eqn} and give the proof of Theorem~\ref{absence of anomalous dissipation thm} by showing that there is no dissipative anomaly, namely
\begin{align*}
\lim_{\kappa\to0}\left(\limsup_{t\to\infty}\int_{0}^{t}\|\nabla\theta(\cdot,s)\|_{L^2}^2ds\right)=0.
\end{align*}

Theorem~\ref{absence of anomalous dissipation thm} will be proved by a sequence of propositions and lemmas. To begin with, the following proposition gives a bound for $\|\theta\|_{L^p}$ for all $1\le p\le \infty$ in terms of $\lambda$.


\begin{proposition}\label{uniform bound on theta kappa>0 prop}
Let $\kappa>0$, $S\in L^\infty$, $\theta_0\in L^\infty$. There exists a unique solution $\theta=\theta(x,t)$ to \eqref{abstract active scalar eqn} satisfying
\begin{align}\label{energy equality with damping}
\frac{1}{2}\frac{d}{dt}\|\theta\|_{L^2}^2+\lambda\|\theta\|_{H^\frac{1}{2}}^2+\kappa\|\Lambda^\frac{\gamma}{2}\theta\|^2_{L^2}=\langle S,\theta\rangle,
\end{align}
and
\begin{align}\label{bound on theta by damping}
\|\theta(\cdot,t)\|_{L^p}\le e^{-\lambda t}\Big[\|\theta_0\|_{L^p}-\frac{1}{\lambda}\|S\|_{L^p}\Big]+\frac{1}{\lambda}\|S\|_{L^p}
\end{align}
for all $1\le p\le \infty$. Moreover the positive semi-orbit
\begin{align}\label{def of positive semi-orbit}
O_{+}(\theta_0)=\{\theta=\theta(\cdot,t): t\ge0\}\subset L^2
\end{align}
is uniformly integrable: for every $\varepsilon>0$, there exists $R>0$ such that
\begin{align}\label{uniform integrable}
\int_{|x|\ge R}|\theta(x,t)|^2dx\le \varepsilon
\end{align}
for all $t\ge0$.
\end{proposition}
\begin{proof}
The existence and uniqueness of solution $\theta$ to \eqref{abstract active scalar eqn} can be proved by the methods given in \cite{FS18} and \cite{FS21}, and the bound \eqref{bound on theta by damping} follows from the point-wise estimate given in \cite{CV12}. 

To prove the uniform integrability \eqref{uniform integrable}, as suggested in \cite{CTV14b}, one can consider the function
\begin{align*}
Y_{R}(t)=\int \chi(\frac{x}{R})\theta^2(x,t)dx,
\end{align*}
where $\chi(\cdot)$ is a nonnegative smooth function supported in the set $\{x\in\Omega:|x|\ge\frac{1}{2}\}$ and identically equal to 1 for $|x|\ge1$. Then we readily have
\begin{align*}
2\lambda\int_{\Omega}(\Lambda\theta(x))\chi(\frac{x}{R})\theta(x,t)dx\ge-\frac{C\lambda}{R}\|\theta(\cdot,t)\|_{L^2}^2,
\end{align*}
and using \eqref{zero order bounded for u when nu>0}, 
\begin{align*}
\int_{\Omega}\chi(\frac{x}{R})\theta u\cdot\nabla\theta(x,t)\le \frac{C}{R}\|\Big[\|u(\cdot,t)\|_{L^3}^3+\|\theta(\cdot,t)\|_{L^3}^3\Big]\le \frac{C}{R}\|\theta(\cdot,t)\|_{L^3}^3.
\end{align*}
The integral involving $S$ can be bound by
\begin{align*}
2\Big|\int_{\Omega} S\chi\theta dx\Big|\le C\|\theta(\cdot,t)\|_{L^2}\Big(\int_{|x|\ge\frac{R}{2}}|S(x)|^2dx\Big)^\frac{1}{2},
\end{align*}
and we obtain
\begin{align*}
\frac{d}{dt} Y_{R}(t)+2\lambda Y_{R}(t)\le \frac{C}{R}\|\Big[\|\theta(\cdot,t)\|_{L^3}^3+\lambda\|\theta(\cdot,t)\|_{L^2}^2\Big]+\frac{C\kappa}{R^2}\|\theta\|_{L^2}^2\\
+C\|\theta(\cdot,t)\|_{L^2}\Big(\int_{|x|\ge\frac{R}{2}}|S(x)|^2dx\Big)^\frac{1}{2}\le \frac{\lambda}{2},
\end{align*}
where the last inequality follows by choosing $R$ large enough. Upon integrating the above inequality over $[0,t]$, we conclude that \eqref{uniform integrable} holds.
\end{proof}
Next we give the definition of stationary statistical solution of \eqref{abstract active scalar eqn} when $\kappa>0$.
\begin{definition}
A stationary statistical solution of \eqref{abstract active scalar eqn} is a Borel probability measure $\mu^{(\kappa)}$ on $L^2(\Omega)$ such that
\begin{align*}
\int_{L^2(\Omega)}\|\theta\|_{H^1}^2d\mu^{(\kappa)}(\theta)<\infty,
\end{align*}
and the equation
\begin{align*}
\int_{L^2(\Omega)}\langle N^{(\kappa)}(\theta),\Psi'(\theta)\rangle d\mu^{(\kappa)}(\theta)=0,
\end{align*}
holds for all $\Psi\in\mathcal{T}$, and
\begin{align}\label{inequality on the integral on mu kappa>0}
\int_{E_1\le\|\theta\|_{H^\frac{1}{2}}\le E_2}(\lambda\|\theta\|_{H^\frac{1}{2}}^2+\kappa\|\nabla\theta\|_{L^2}^2-\langle S,\theta\rangle)d\mu^{(\kappa)}(\theta)\le 0
\end{align}
for all $E_1\le E_2$. Here $N^{(\kappa)}(\theta)$ is given by
\begin{align*}
N^{(\kappa)}(\theta)=u[\theta]\cdot\nabla\theta+\lambda\mathcal{D}\theta+\kappa\Lambda^{\gamma}\theta-S
\end{align*}
and the class of cylindrical test functionals $\mathcal{T}$ is defined as follows: $\Psi\in\mathcal{T}$ if there exists $N$, $w_1,..,w_N\in C^{\infty}_{0}(\Omega)$, $\varepsilon\ge0$ and $\psi:\R^N\to\R$, smooth function such that
\begin{align*}
\Psi(\theta)=\psi(\langle J_{\varepsilon}(\theta),w_1\rangle,...,\langle J_{\varepsilon}(\theta),w_N\rangle)
\end{align*}
with $J_{\varepsilon}$ being a standard mollifier.
\end{definition}
We aim at showing that the map $\theta\mapsto\langle N^{(\kappa)}(\theta),\Psi'(\theta)\rangle$ is locally bounded and weakly continuous. To this end, we further introduce $F_1$, $F_2$ and $F_3$ for computing $\langle N^{(\kappa)}(\theta),\Psi'(\theta)\rangle=\sum_{i=1}^3 F_i(\theta)$:
\begin{align*}
F_1(\theta)&=\lambda\langle\theta,\mathcal{D}\Psi'(\theta)\rangle - \langle S,\Psi'(\theta)\rangle,\\
F_2(\theta)&=\kappa\langle\theta,\Lambda^\gamma\Psi'(\theta)\rangle,\\
F_3(\theta)&=-\langle\theta u[\theta],\nabla\Psi'(\theta)\rangle.
\end{align*}
The following proposition gives the desired properties of the maps $F_1$, $F_2$ and $F_3$.
\begin{proposition}\label{weakly continuity of mu kappa prop}
For any fixed $\Psi\in\mathcal{T}$, the maps $\theta\mapsto F_i(\theta)$ with $i=1,2,3$, are locally bounded in $L^2$ and weakly continuous in $L^2$ on bounded sets of $L^p$, $1\le p<2$. In particular, the map
\begin{align}\label{nonlinear term map}
\theta\mapsto\langle N^{(\kappa)}(\theta),\Psi'(\theta)\rangle
\end{align}
is locally bounded in $L^2$ and weakly continuous in $L^2$ on bounded sets of $L^p$, $1\le p<2$.
\end{proposition}
\begin{proof}
The method is similar to the one given in \cite{CTV14b}, and we give the proof for $\Omega=\mathbb{T}^3$ as an example. Suppose $\theta_n$ is weakly converging to its weak limit $\theta$ in $L^2$ and $\theta_n$ satisfy the assumed bound
\begin{align}\label{bound on theta n weak convergence}
\sup_{n}\|\theta_n\|_{L^p}\le A_p
\end{align}
for some $A_p>0$. Then by integrating $\chi(\frac{x}{R})\theta^{p-1}$ against $\theta_n$, passing to the limit in $n$ and taking $R\to\infty$, we also have
\begin{align}\label{bound on theta weak convergence}
\|\theta\|_{L^p}\le A_p.
\end{align}

In view of the decomposition of the map \eqref{nonlinear term map}, it suffices to address the map $F_3$ since $F_1$ and $F_2$ satisfy the claimed properties trivially. The crucial convergence for consideration is the one that involves $u[\theta]$, which is given by
\begin{align}\label{convergence of F3 term}
\langle \theta_nu[\theta_n],\nabla J_{\varepsilon}w_k\rangle \to \langle \theta u[\theta],\nabla J_{\varepsilon}w_k\rangle
\end{align}
for $k=1,2,...,N$. For $\theta\in L^p$ and smooth compactly supported $\phi$, the key for proving \eqref{convergence of F3 term} is to rewrite $\langle \theta u[\theta],\nabla \phi\rangle$ as 
\begin{align*}
\langle \theta u[\theta],\nabla \phi\rangle =\frac{1}{2}\int(\Lambda^{-1}\theta)(x)[\Lambda,\nabla \phi]u[\theta](x)dx,
\end{align*}
where $[\Lambda,f]g=\Lambda(fg)-f\Lambda g$ is the commutator operator. Since we have
\begin{align*}
\Lambda(fg)(x)-f\Lambda g(x)=C\,P.V. \int_{\R^3}\frac{g(y)(f(x)-f(y))}{|x-y|^4}dy,
\end{align*}
if $f$ is compactly supported in a ball $B(0,R)$ of radius $R$, then for $|x|\ge 2R$,
\begin{align}\label{bound on commutator}
|\Lambda(fg)(x)-f\Lambda g(x)|\le C|x|^{-4}\|f\|_{L^2}\|g\|_{L^2}.
\end{align}
Take $\phi=J_{\varepsilon}w_k$ and choose $R>1$ such that $supp(\phi)\subset B(0,R)$, for $\theta\in L^p$, we define $C_{\phi}(\theta)(x)$ by
\begin{align*}
C_{\phi}(\theta)(x)=[\Lambda,\nabla\phi]u[\theta].
\end{align*}
Then by \eqref{zero order bounded for u when nu>0} and \eqref{bound on commutator}, we have
\begin{align*}
\Big|\int_{|x|\ge 2R}(\Lambda^{-1}\theta_n)(x)C_{\phi}(\theta_n)(x)dx\Big|\le CR^{-1}\|\theta_n\|_{L^2}\Big(\|\theta_n\|_{L^p}+\|\theta_n\|_{L^2}\Big),
\end{align*}
and hence using the bounds \eqref{bound on theta n weak convergence} and \eqref{bound on theta weak convergence},
\begin{align}\label{uniform bound on the theta n weak convergence}
\Big|\int_{|x|\ge 2R}(\Lambda^{-1}\theta_n)(x)C_{\phi}(\theta_n)(x)dx\Big|\le CR^{-1}A_2(A_p+A_2).
\end{align}
Now we consider the difference
\begin{align}\label{difference on weak convergence}
&\int(\Lambda^{-1}\theta_n)(x)C_{\phi}(\theta_n)(x)dx-\int(\Lambda^{-1}\theta)(x)C_{\phi}(\theta)(x)dx\\
&=\int_{B(0,2R)}(\Lambda^{-1}(\theta_n-\theta))(x)C_{\phi}(\theta_n)(x)dx+\int_{B(0,2R)}(\Lambda^{-1}(\theta))(x)C_{\phi}(\theta_n-\theta)(x)dx\notag\\
&\qquad+\int_{|x|\ge2R}(\Lambda^{-1}\theta_n)(x)C_{\phi}(\theta_n)(x)dx-\int_{|x|\ge2R}(\Lambda^{-1}\theta)(x)C_{\phi}(\theta)(x)dx.\notag
\end{align}
We let $\epsilon>0$ be given. Since $\Lambda^{-1}\theta,C_{\phi}(\theta)\in L^2$, together with the bound \eqref{uniform bound on the theta n weak convergence}, by choosing $R$ large enough, the third and the fourth term on the right side of \eqref{difference on weak convergence} can be bounded by $\frac{\epsilon}{2}$ each. On the other hand, the first term on the right side of \eqref{difference on weak convergence} can be bounded by
\begin{align*}
\Big|\int_{B(0,2R)}(\Lambda^{-1}(\theta_n-\theta))(x)C_{\phi}(\theta_n)(x)dx\Big|\le CA_2\|\Lambda^{-1}(\theta_n-\theta)\|_{L^2(B(0,2R)}.
\end{align*}
Since $\chi(\frac{x}{2R})\Lambda^{-1}(\theta_n-\theta)$ is uniformly bounded in $H^1$ and converges weakly to $0$ in $L^2$, $\|\Lambda^{-1}(\theta_n-\theta)\|_{L^2(B(0,2R)}$ converges to zero as $n\to\infty$. Finally, since $C_\phi(\cdot)$ is bounded linear on $L^2$, $C_{\phi}(\theta_n-\theta)$ converges weakly to 0 as $n\to0$ and hence as $n\to0$, we have
\begin{align*}
\int_{B(0,2R)}(\Lambda^{-1}(\theta))(x)C_{\phi}(\theta_n-\theta)(x)dx=\int\chi(\frac{x}{2R})(\Lambda^{-1}(\theta))(x)C_{\phi}(\theta_n-\theta)(x)dx\to0.
\end{align*}
The above shows that as $n\to\infty$, the difference on the left side of \eqref{difference on weak convergence} can be bounded in absolute value by $\epsilon$. Since $\epsilon$ is arbitrary, we conclude that $F_3$ is locally bounded in $L^2$ and weakly continuous in $L^2$ on bounded sets of $L^p$, $1\le p<2$.
\end{proof}
We give the definition of stationary statistical solution of \eqref{abstract active scalar eqn} when $\kappa=0$.
\begin{definition}
A stationary statistical solution of \eqref{abstract active scalar eqn} with $\kappa=0$, namely the inviscid active scalar equation 
\begin{align}\label{abstract active scalar eqn kappa=0}
\partial_t\theta+u\cdot\nabla\theta+\lambda\mathcal{D}\theta=S,
\end{align}
is a Borel probability measure $\mu$ on $L^2(\Omega)$ such that
\begin{align}\label{boundedness on mu on L2}
\int_{L^2(\Omega)}\|\theta\|_{H^1}^2d\mu(\theta)<\infty,
\end{align}
and the equation 
\begin{align}\label{integral identity on mu on L2}
\int_{L^2(\Omega)}\langle N(\theta),\Psi'(\theta)\rangle d\mu(\theta)=0,
\end{align}
holds for all $\Psi\in\mathcal{T}$, where
\begin{align*}
N(\theta)=u[\theta]\cdot\nabla\theta+\lambda\mathcal{D}\theta-S.
\end{align*}
We say that the stationary statistical solution satisfies the energy dissipation balance if
\begin{align}\label{energy dissipation balance}
\int_{L^2(\Omega)}(\lambda\|\theta\|_{H^\frac{1}{2}}^2-\langle S,\theta\rangle)d\mu(\theta)=0.
\end{align}
\end{definition}
The following theorem shows the convergence of stationary statistical solution as $\kappa$ tends to zero, which is a consequence of Proposition~\ref{weakly continuity of mu kappa prop}.
\begin{theorem}
Let $\mu^{(\kappa)}$ be a sequence of stationary statistical solutions of \eqref{abstract active scalar eqn} with $S\in L^2$. Assume that there exists $1\le p<2$ and $A_p$ such that the support of the measure $\mu^{(\kappa)}$ are included in the set
\begin{align}\label{set Bp}
B_{p}=\{\theta\in L^p:\|\theta\|_{L^p}\le A_p\},
\end{align}
where $A_p>0$ is given in \eqref{bound on theta n weak convergence}-\eqref{bound on theta weak convergence}. Then there exists a subsequence, denoted also $\mu^{(\kappa)}$ and a stationary statistical
solution $\mu$ of \eqref{abstract active scalar eqn kappa=0} such that
\begin{align}\label{weak limit integral identity for mu}
\lim_{\kappa\to0}\int_{L^2}\Phi(\theta)d\mu^{(\kappa)}(\theta)=\int_{L^2}\Phi(\theta)d\mu(\theta)
\end{align}
holds for all weakly continuous, locally bounded real valued functions $\Phi$.
\end{theorem}
\begin{proof}
Define a set $B$ by
\begin{align*}
B=\Big\{\theta\in H^\frac{1}{2}:\|\theta\|_{H^\frac{1}{2}}\le\frac{\|S\|_{L^2}}{\lambda}\Big\}.
\end{align*}
Then the intersection $B\cap B_p$ is weakly closed in $L^2$ and it is a separable metrizable compact space with the weak $L^2$ topology. By Prokhorov's theorem, the sequence $\mu^{(\kappa)}$ has a weakly convergent subsequence that converges to $\mu$. It is clear that $\mu$ is a Borel probability measure on $B\cap B_p$, and we can further extend it to $L^2$ by defining $\mu(X)=\mu(X\cap(B\cap B_p))$ for all measurable set $X\subset L^2$. Then we have
\begin{align*}
\int_{L^2(\Omega)}\|\theta\|_{H^1}^2d\mu(\theta)\le \int_{B}\|\theta\|_{H^1}^2d\mu(\theta)\le \int_{B}\frac{\|S\|_{L^2}^2}{\lambda^2}d\mu(\theta)<\infty,
\end{align*}
and hence \eqref{boundedness on mu on L2} holds. To prove \eqref{integral identity on mu on L2}, using Proposition~\ref{weakly continuity of mu kappa prop}, we have
\begin{align*}
\int_{L^2(\Omega)}\langle N(\theta),\Psi'(\theta)\rangle d\mu(\theta)=\lim_{\kappa\to0}\int_{L^2(\Omega)}\langle N^{(\kappa)}(\theta),\Psi'(\theta)\rangle d\mu^{(\kappa)}(\theta)=0.
\end{align*}
Therefore $\mu$ is a stationary statistical solution $\mu$ of \eqref{abstract active scalar eqn kappa=0} such that \eqref{weak limit integral identity for mu} holds for all weakly continuous, locally bounded real valued functions $\Phi$ as guaranteed by Proposition~\ref{weakly continuity of mu kappa prop}.
\end{proof}

In the next theorem, it shows that if a sequence of stationary statistical solutions of \eqref{abstract active scalar eqn} is supported in the set $A$ given by
\begin{align*}
A=\Big\{\theta:\|\theta\|_{L^p}\le A_p,\,\|\theta\|_{L^\infty}\le A_{\infty},\,\|\theta\|_{H^\frac{1}{2}}\le\frac{\|S\|_{L^2}}{\lambda}\Big\},
\end{align*}
then the weak limit of the sequence is a stationary statistical solutions of \eqref{abstract active scalar eqn kappa=0} that further satisfies the energy dissipation balance \eqref{energy dissipation balance}.

\begin{theorem}
Let $\mu^{(\kappa)}$ be a sequence of stationary statistical solutions of \eqref{abstract active scalar eqn} is supported in the set $A$ as given above. Let $\mu$ be any weak limit of $\mu^{(\kappa)}$ in $L^2$. Then $\mu$ is a stationary statistical solution of \eqref{abstract active scalar eqn kappa=0} that satisfies the energy dissipation balance \eqref{energy dissipation balance}.
\end{theorem}

\begin{proof}
The proof is essentially the same as the one given in \cite{CTV14b} for the 2D case, and we only sketch some of the key steps. To prove the energy dissipation balance \eqref{energy dissipation balance}, we observe that
\begin{align*}
\lim_{\varepsilon\to0}\int_{L^2(\Omega)}\langle J_{\varepsilon}(\theta),J_{\varepsilon}(\kappa\mathcal{D}-S)\rangle d\mu(\theta)=\int_{L^2(\Omega)}(\lambda\|\theta\|_{H^\frac{1}{2}}^2-\langle S,\theta\rangle)d\mu(\theta)
\end{align*}
where $J_{\varepsilon}$ is the standard mollifier as defined before. Assuming that for all $\varepsilon>0$, we have
\begin{align}\label{equality on integral with varepsilon}
\int_{L^2(\Omega)}\langle J_{\varepsilon}(\theta),J_{\varepsilon}(\kappa\mathcal{D}-S)\rangle d\mu(\theta)=-\int_{L^2(\Omega)}\langle J_{\varepsilon}(\theta),J_{\varepsilon}(u[\theta]\cdot\nabla\theta)\rangle d\mu(\theta),
\end{align}
then it suffices to show that 
\begin{align}\label{zero limit involving u mu} 
\lim_{\varepsilon\to0}\int_{L^2(\Omega)}\langle J_{\varepsilon}(\theta),J_{\varepsilon}(u[\theta]\cdot\nabla\theta)\rangle d\mu(\theta)=0.
\end{align}
Since $\nabla\cdot u[\theta]=0$, upon integration by parts, we have
\begin{align*}
\int_{L^2(\Omega)}\langle J_{\varepsilon}(\theta),J_{\varepsilon}(u[\theta]\cdot\nabla\theta)\rangle d\mu(\theta)=\int_{L^2(\Omega)}\langle \nabla J_{\varepsilon}(\theta),\rho_{\varepsilon}(u[\theta],\theta)\rangle d\mu(\theta),
\end{align*}
where 
\begin{align}\label{def of rho function with varepsilon}
\rho_{\varepsilon}(u[\theta],\theta)=J_{\varepsilon}(u[\theta]\theta)-(J_{\varepsilon}(u[\theta])(J_{\varepsilon}(\theta)).
\end{align} 
Notice that for any $\theta\in L^\frac{1}{2}\cap L^\infty$,
\begin{align*}
\lim_{\varepsilon\to0}\langle \nabla J_{\varepsilon}(\theta),\rho_{\varepsilon}(u[\theta],\theta)\rangle=0,
\end{align*}
and using the boundedness of $u[\cdot]$, we obtain the point-wise bound that
\begin{align*}
|\langle \nabla J_{\varepsilon}(\theta),\rho_{\varepsilon}(u[\theta],\theta)\rangle|\le C\|\theta\|_{L^\infty}\|\theta\|_{H^\frac{1}{2}}^2.
\end{align*}
Therefore, by the Lebegue dominated convergence theorem, we have
\begin{align*}
\lim_{\varepsilon\to0}\int_{L^2(\Omega)}\langle \nabla J_{\varepsilon}(\theta),\rho_{\varepsilon}(u[\theta],\theta)\rangle d\mu(\theta)=0,
\end{align*}
which implies \eqref{zero limit involving u mu} holds. It remains to show the identity \eqref{equality on integral with varepsilon}. Notice that
\begin{align*}
J_{\varepsilon}(N(\theta))=J_{\varepsilon}(\kappa\mathcal{D}-S)+J_{\varepsilon}(u[\theta]\cdot\nabla\theta),
\end{align*}
hence it is equivalent to show that
\begin{align*}
\int_{L^2(\Omega)}\langle J_{\varepsilon}(\theta),J_{\varepsilon}(N(\theta))\rangle d\mu(\theta)=0.
\end{align*}
We choose ${w_j}_{j=1}^{\infty}\subset C^{\infty}_0$ an orthonormal basis of $L^2$. For each $\varepsilon>0$, consider a sequence of cylindrical test functionals $\Psi_m(\cdot)\in \mathcal{T}$ such that
\begin{align*}
\Psi_m(\theta)=\frac{1}{2}\sum_{j=1}^m\langle J_{\varepsilon}(\theta),w_j\rangle^2,
\end{align*}
then we have
\begin{align*}
\langle N(\theta),\Psi_m'(\theta)\rangle = \sum_{j=1}^m\langle J_{\varepsilon}(\theta),w_j\rangle\langle J_{\varepsilon}(N(\theta)),w_j\rangle.
\end{align*}
Moreover, it can proved that $\langle N(\theta),\Psi_m'(\theta)\rangle\to0$ point-wisely as $m\to\infty$ and
\begin{align*}
|\langle N(\theta),\Psi_m'(\theta)\rangle|\le \|J_{\varepsilon}(\theta)\|_{L^2}\|J_{\varepsilon}(N(\theta))\|_{L^2}
\end{align*}
uniformly in $m$. Hence by applying Lebegue dominated convergence theorem and using \eqref{integral identity on mu on L2},
\begin{align*}
\int_{L^2(\Omega)}\langle J_{\varepsilon}(\theta),J_{\varepsilon}(N(\theta))\rangle d\mu(\theta)=\lim_{m\to\infty}\int_{L^2(\Omega)}\langle N(\theta),\Psi_m'(\theta)\rangle d\mu(\theta)=0,
\end{align*}
which implies \eqref{equality on integral with varepsilon}.
\end{proof}

We recall the concept of generalized (Banach) limit that will be used for addressing the long time average of solutions $\theta$ and the corresponding stationary statistical solutions.

\begin{definition}
A generalized limit (Banach limit) is a bounded linear functional
\begin{align*}
{\rm Lim}_{t\to\infty}:BC([0,\infty))\to\R
\end{align*}
such that
\begin{itemize}
\item ${\rm Lim}_{t\to\infty}(g)\ge0,\,\forall g\in BC([0,\infty)),\, g\ge0.$
\item ${\rm Lim}_{t\to\infty}(g)=\lim_{t\to\infty}g(t)$ whenever the usual limit exists.
\end{itemize}
\end{definition}
The following proposition gives the relation between the long time averages of \eqref{abstract active scalar eqn} with $\kappa>0$ and the stationary statistical solutions. Basically, it shows that the generalized limit (Banach limit) of the long time averages defines a stationary statistical solution $\mu^{(\kappa)}$ for each $\kappa>0$.
\begin{proposition}\label{long time average as ss kappa>0 prop}
Let $S\in L^\infty$ and $\theta_0\in L^\infty$, and denote the set of strongly continuous real-valued functionals on a set $X\subseteq L^2$ by $\mathcal{C}(X)$. If $\theta=\theta^{(\kappa)}(x,t)$ is the solution to \eqref{abstract active scalar eqn}, then for $\Phi\in\mathcal{C}(L^2)$, the map
\begin{align}\label{realisation map on solution kappa>0}
\Phi\mapsto{\rm Lim}_{t\to\infty}\frac{1}{t}\int_{0}^{t}\Phi(\theta^{(\kappa)}(x,s))ds
\end{align}
defines a stationary statistical solution $\mu^{(\kappa)}$ of \eqref{abstract active scalar eqn}:
\begin{align}\label{representation of theta as measure kappa>0}
\int_{L^2}\Phi(\theta)d\mu^{(\kappa)}(\theta)={\rm Lim}_{t\to\infty}\frac{1}{t}\int_{0}^{t}\Phi(\theta^{(\kappa)}(x,s))ds.
\end{align}
The measure is supported in the set
\begin{align*}
\bar{A}=\Big\{\theta:\|\theta\|_{H^\frac{1}{2}}\le\frac{\|S\|_{L^2}}{\lambda},\,\|\theta\|_{L^p}\le \bar{A}_p,\,1\le p\le \infty\,\Big\}
\end{align*}
with
\begin{align*}
\bar{A}_p=\|\theta_0\|_{L^p}+\frac{\|S\|_{L^p}}{\lambda},\,1\le p\le\infty.
\end{align*}
The inequality 
\begin{align}\label{inequality on nabla theta measure mu kappa>0}
\int_{L^2}\Big[\kappa\|\nabla\theta\|_{L^2}^2+\lambda\|\theta\|_{H^\frac{1}{2}}^2-\langle S,\theta\rangle\Big]d\mu^{(\kappa)}(\theta)\le0
\end{align}
holds.
\end{proposition}
\begin{proof}
Notice that the positive semi-orbit $O_{+}(\theta_0)$ given in \eqref{def of positive semi-orbit} is relatively compact in $L^2$ as shown by Proposition~\ref{uniform bound on theta kappa>0 prop}. Since $C(L^2)\subset C(\overline{O_{+}(\theta_0)})$, for any $\Phi\in C(L^2)$, the map \eqref{realisation map on solution kappa>0} is a positive functional on $\overline{O_{+}(\theta_0)}$, and by the Riesz representation theorem on compact spaces, there exists a Borel measure $\mu^{(\kappa)}$ such that \eqref{representation of theta as measure kappa>0} holds. The verification of $\mu^{(\kappa)}$ for being a stationary statistical solution of \eqref{abstract active scalar eqn} can be found in \cite{CTV14b}, and using the bound \eqref{bound on theta by damping}, it shows that $O_{+}(\theta_0)$ is included in the set $B_p$ as given in \eqref{set Bp}, hence the measure $\mu^{(\kappa)}$ is supported in the set $\bar{A}$ by \eqref{inequality on the integral on mu kappa>0}.

To show the inequality \eqref{inequality on nabla theta measure mu kappa>0}, by using $J_{\varepsilon}$, we mollify $\theta$ and $u[\theta]$ and obtain the integral identity:
\begin{align*}
&\frac{1}{t}\int_{0}^{t}\left(\lambda\|J_{\varepsilon}(\theta)\|_{H^\frac{1}{2}}^2-\langle J_{\varepsilon}(S),J_{\varepsilon}(\theta)\rangle+\kappa\|\nabla J_{\varepsilon}(\theta)\|_{L^2}^2\right)ds\\
&=\frac{1}{2t}\Big(\|J_{\varepsilon}(\theta_0)\|_{L^2}^2-\|J_{\varepsilon}(\theta)(\cdot,t)\|_{L^2}^2\Big)+\frac{1}{t}\int_{0}^{t}\langle\rho_\varepsilon(u[\theta],\theta),\nabla J_{\varepsilon}(\theta)\rangle ds,
\end{align*}
where $\rho_\varepsilon$ is defined in \eqref{def of rho function with varepsilon}. Taking $t\to\infty$,
\begin{align}\label{integral involving rho function and varepsilon}
&\int_{L^2}\left(\kappa\|\nabla J_{\varepsilon}(\theta)\|_{L^2}^2+\lambda\|J_{\varepsilon}(\theta)\|_{H^\frac{1}{2}}^2-\langle J_{\varepsilon}(S),J_{\varepsilon}(\theta)\rangle\right)d\mu^{(\kappa)}\notag\\
&={\rm Lim}_{t\to\infty}\frac{1}{t}\int_{0}^{t}\langle\rho_\varepsilon(u[\theta],\theta),\nabla J_{\varepsilon}(\theta)\rangle ds.
\end{align}
Using \eqref{zero order bounded for u when nu>0}, we can bound $\rho_\varepsilon(u[\theta],\theta)$ by
\begin{align*}
\|\rho_\varepsilon(u[\theta],\theta)\|_{L^2}\le C\sqrt{\varepsilon}\|\theta\|_{L^\infty}\|\theta\|_{L^2},
\end{align*}
and hence the right side of \eqref{integral involving rho function and varepsilon} tends to zero as $\varepsilon\to0$. By Fatou's Lemma, this implies 
\begin{align*}
&\int_{L^2}\Big[\kappa\|\nabla\theta^{(\kappa)}\|_{L^2}^2+\lambda\|\theta^{(\kappa)}\|_{H^\frac{1}{2}}^2-\langle S,\theta^{(\kappa)}\rangle\Big]d\mu^{(\kappa)}(\theta^{(\kappa)})\\
&\le\limsup_{\varepsilon\to0}\int_{L^2}\left(\kappa\|\nabla J_{\varepsilon}(\theta)\|_{L^2}^2+\lambda\|J_{\varepsilon}(\theta)\|_{H^\frac{1}{2}}^2-\langle J_{\varepsilon}(S),J_{\varepsilon}(\theta)\rangle\right)d\mu^{(\kappa)}\le 0,
\end{align*}
and the inequality \eqref{inequality on nabla theta measure mu kappa>0} holds.
\end{proof}
Finally, we conclude this section by ahowing the absence of anomalous dissipation of energy for long time averaged solutions of the abstract active scalar equation \eqref{abstract active scalar eqn}, thereby proving Theorem~\ref{absence of anomalous dissipation thm}. 
\begin{proof}[Proof of Theorem~\ref{absence of anomalous dissipation thm}]
It can be argued by contraction. Suppose there exists $\delta>0$, a sequence $\kappa_k\to0$ and a sequence of times $t_j\to\infty$ such that for all $t_j$,
\begin{align*}
\frac{\kappa_k}{t_j}\int_{0}^{t_j}\|\nabla\theta^{(\kappa_k)}(x,t)\|_{L^2}^2 ds\ge\delta.
\end{align*}
Integrating \eqref{energy equality with damping} from $0$ to $t_j$ and taking limit of supreme, 
\begin{align*}
\limsup_{t\to\infty}\frac{1}{t}\int_{0}^{t}\left(-\lambda\|\theta^{(\kappa_k)}(\cdot,s)\|_{H^\frac{1}{2}}^2+\langle S(\cdot),\theta^{(\kappa_k)}(\cdot,s)\rangle\right)ds\ge\delta.
\end{align*}
Hence by Proposition~\ref{long time average as ss kappa>0 prop}, for each $\kappa_k>0$, there exists a corresponding stationary statistical solution $\mu^{(\kappa_k)}$ satisfying
\begin{align}\label{lower bound of mu kappa k}
\int_{L^2}\left(-\lambda\|\theta(\cdot,s)\|_{H^\frac{1}{2}}^2+\langle S(\cdot),\theta(\cdot,s)\rangle\right)d\mu^{(\kappa_k)}(\theta)\ge\delta.
\end{align}
By taking a weakly convergent subsequence of $\mu^{(\kappa_k)}$, there exists a stationary statistical solution $\mu$ that satisfies \eqref{energy dissipation balance}. On the other hand, in view of \eqref{lower bound of mu kappa k}, the weak limit $\mu$ further satisfies
\begin{align*}
\int_{L^2}\left(\lambda\|\theta(\cdot,s)\|_{H^\frac{1}{2}}^2-\langle S(\cdot),\theta(\cdot,s)\rangle\right)d\mu(\theta)\le -\delta<0,
\end{align*}
which contradicts to \eqref{energy dissipation balance}. This completes the proof of \eqref{absence of anomalous dissipation identity}.
\end{proof}

\section{Existence and convergence of $H^s$-solutions when $\kappa>0$}\label{Existence and convergence of Hs-solutions section}

In this section, we first obtain the global-in-time existence of $H^s$-solutions to \eqref{abstract active scalar eqn} when $\lambda\ge0$, $\kappa>0$ and the assumption A3 is replacing by 
\begin{itemize}
\item[A3'.] There exists a constant $C>0$ such that for all $1\le i,j\le 3$, $$|\widehat{T_{ij}}(k)|\le C|k|^{-3}, \qquad\forall k\neq0.$$
\end{itemize}
Under assumption A3', it implies that $u_j[\cdot]=\partial_{x_i} T_{ij}[\cdot]$ is an operator of smoothing order 2, in the sense that for any $s\ge0$ and $f\in L^p$ with $p>1$,
\begin{align}\label{two order smoothing for u when nu>0}
\|\Lambda^s u[f]\|_{L^p}\le C\|\Lambda^{s'}f\|_{L^p},
\end{align}
where $s'=\max\{s-2,0\}$. Here $C$ is a positive constant which depends on $p$ and $\Omega$ only. After obtaining global-in-time existence, we then proceed to prove $H^s$-convergence of the solutions $\theta$ as $\lambda\rightarrow0$. More precisely, given $\kappa>0$ and $\gamma\in(0,2]$, we prove that if $\theta^{(\lambda)}$ and $\theta^{(0)}$ are smooth solutions to \eqref{abstract active scalar eqn} for $\lambda>0$ and $\lambda=0$ respectively, then $\|(\theta^{(\lambda)}-\theta^{(0)})(\cdot,t)\|_{H^s}\rightarrow0$ as $\lambda\rightarrow0$ for all $t>0$ and $s\ge0$. 

\begin{remark}
It is clear that assumption A3 implies A3', and the class of MG equations given by \eqref{MG active scalar} later in Section~\ref{Applications to magneto-geostrophic equations section} is an example of active scalar equations that satisfy assumption A3'.
\end{remark}

\subsection{Existence of $H^s$-solutions}\label{existence of Hs solution subsec}

In this subsection, we prove global-in-time wellposedness for \eqref{abstract active scalar eqn} in Sobolev space $H^s$ when $\lambda\ge0$ and $\kappa>0$. The results are given in Theorem~\ref{global-in-time wellposedness in Sobolev}.

\begin{theorem}[Global-in-time wellposedness in Sobolev space]\label{global-in-time wellposedness in Sobolev}
Fix $\kappa>0$, $\gamma\in(0,2]$ and $s\ge0$, and let $\theta_0\in H^s$ and $S\in H^s\cap L^\infty$ be given. 
\begin{itemize}
\item For any $\lambda>0$, there exists a global-in-time solution to \eqref{abstract active scalar eqn} such that 
\begin{align}\label{functional spaces for theta}
\thetalambda\in C([0,\infty);H^s)\cap L^2([0,\infty);H^{s+\frac{\bar{\gamma}}{2}}),
\end{align}
where $\bar\gamma=\max\{1,\gamma\}$.
\item For $\lambda=0$, there exists a global-in-time solution to \eqref{abstract active scalar eqn} such that 
\begin{align}\label{functional spaces for theta lambda=0}
\theta^{(0)}\in C([0,\infty);H^s)\cap L^2([0,\infty);H^{s+\frac{\gamma}{2}}).
\end{align}
\end{itemize}
\end{theorem}
\begin{remark}
When $\lambda>0$ and $\gamma\le1$, thanks to the damping term $\lambda\mathcal{D}\thetalambda$, Theorem~\ref{global-in-time wellposedness in Sobolev} ensures that $\thetalambda\in C([0,\infty);H^s)\cap L^2([0,\infty);H^{s+\frac{1}{2}})$.
\end{remark}
The most subtle part for proving Theorem~\ref{global-in-time wellposedness in Sobolev} is to estimate the $L^\infty$-norm of $\theta^{(\lambda)}(\cdot,t)$ when the initial datum $\theta_0$ is {\it not} necessarily in $L^\infty$. In achieving our goal, we apply De Giorgi iteration method which will be illustrated in Lemma~\ref{De Giorgi iteration lemma}. 

We first recall the following energy inequalities which were presented in \cite{FS21} for the case $\lambda=0$. The general cases for $\lambda>0$ follows similarly and we omit the details here.

\begin{proposition}\label{boundedness on theta with kappa>0 prop}
Assume that $S\in L^\infty$, and let $\theta^{(\lambda)}$ be a smooth solution to \eqref{abstract active scalar eqn}. For $\lambda\ge0$, we have
\begin{align}\label{energy inequality kappa>0}
\|\theta^{(\lambda)}(\cdot,t)\|^2_{L^2}+&\lambda\int_0^t\|\theta^{(\lambda)}(\cdot,\tau)\|^2_{H^\frac{1}{2}}d\tau+\kappa\int_0^t\|\Lambda^\frac{\gamma}{2}\theta^{(\lambda)}(\cdot,\tau)\|^2_{L^2}d\tau\notag\\
&\le\|\theta_0\|^2_{L^2}+\frac{t}{c_0\kappa}\|S\|^2_{L^2}, \qquad\forall t\ge0.
\end{align}
If we further assume that $\theta_0\in L^\infty$, for all $\kappa>0$, it gives
\begin{align}\label{infty bound on theta with kappa>0}
\|\theta^{(\lambda)}(\cdot,t)\|_{L^\infty}\le\|\theta_0\|_{L^\infty}e^{-c_o\kappa t}+\frac{\|S\|_{L^\infty}}{c_0\kappa}, \qquad\forall t\ge0,
\end{align}
where $c_0>0$ is a universal constant which depends only on the dimension $d$.
\end{proposition}

\begin{remark}
Instead of the decay estimate given in \eqref{infty bound on theta with kappa>0}, we also have the following bound on $\|\theta^{(\lambda)}(\cdot,t)\|_{L^\infty}$ for all $\kappa\ge0$ provided that $\theta_0\in L^\infty$:
\begin{align}\label{infty bound on theta with kappa=0 or >0}
\|\theta^{(\lambda)}(\cdot,t)\|_{L^\infty}\le \|\theta_0\|_{L^\infty}+\|S\|_{L^\infty},\qquad\forall t\ge0.
\end{align} 
\end{remark}

Next we state and prove the following lemma which gives local-in-time existence results for the equation \eqref{abstract active scalar eqn}. 

\begin{lemma}\label{local-in-time existence lemma}
Let $\lambda>0$, $\kappa>0$, $\gamma\in(0,2]$ and fix $s \ge 0$. Assume that $\theta_0\in H^s$ and $S\in H^s\cap L^\infty$, then there exists $T_*=T_*(\theta_0,S)>0$ and a unique solution $\theta^{(\lambda)}$ of \eqref{abstract active scalar eqn} such that
\begin{align*}
\theta^{(\lambda)}\in C([0,T_*);H^s)\cap L^2([0,T_*);H^{s+\frac{\bar{\gamma}}{2}}),\qquad \bar\gamma=\max\{1,\gamma\}.
\end{align*}
When $\lambda=0$ and $s \ge 0$, if $\theta_0$, $S\in H^s\cap L^\infty$, then there exists $T_*=T_*(\theta_0,S)>0$ and a unique solution $\theta^{(0)}$ of \eqref{abstract active scalar eqn} such that $\theta^{(0)}\in C([0,T_*);H^s)\cap L^2([0,T_*);H^{s+\frac{\gamma}{2}})$.
\end{lemma}

\begin{proof}
We consider only $\lambda >0$ and the case for $\lambda=0$ was fully addressed in \cite{FS21}. Following the argument given in the proof of \cite[Theorem~3.1]{FRV12}, applying \eqref{two order smoothing for u when nu>0} and using standard energy method, there exists $T_*=T_*(\theta_0,S)>0$ and a unique solution $\theta^{(\lambda)}$ of \eqref{abstract active scalar eqn} such that
\begin{align*}
\theta^{(\lambda)}\in L^\infty([0,T_*);H^s)\cap L^2([0,T_*);H^{s+\frac{\gamma}{2}}),
\end{align*}
and in particular, the following inequality holds for $t\in[0,T_*)$:
\begin{align}\label{H^s bound on theta with S and higher norm}
\|\theta^{(\lambda)}(\cdot,t)\|^2_{H^s}+\lambda\int_0^t\|\theta^{(\lambda)}(\cdot,\tau)\|^2_{H^{s+\frac{1}{2}}}d\tau+\kappa\int_0^t\|\theta^{(\lambda)}(\cdot,\tau)\|^2_{H^{s+\frac{\gamma}{2}}}d\tau\notag\\
\le \|\theta_0\|^2_{H^s}+Ct\|S\|^2_{H^s},
\end{align}
where $C$ is a positive constant which is independent of $t$. The continuity of $\theta^{(\lambda)}$ in time follows by the same argument given in \cite{FS21} for the case $\lambda=0$ and we omit the details here.
\end{proof}

The lemma below gives the desired bound on $\|\theta^{(\lambda)}(\cdot,t)\|_{L^\infty}$ for the case when the time $t$ is small, which will be used for proving Theorem~\ref{global-in-time wellposedness in Sobolev}. The idea follows from the work given in \cite{CV10}, \cite{CZV16} and \cite{FS21}.
\begin{lemma}[From $L^2$ to $L^\infty$]\label{De Giorgi iteration lemma}
Fix $\lambda\ge0$, $\kappa>0$, $s\ge0$ and $\gamma\in(0,2]$. Let $\theta=\thetalambda(t)$ be the solution to \eqref{abstract active scalar eqn} with initial datum $\theta_0\in H^s$. Then for all $t\in(0,1]$, we have
\begin{align}\label{infty bound on theta in terms of L2 norm of initial theta}
\|\thetalambda(t)\|_{L^\infty}\le C\Big[\Big(\frac{2}{t}+1\Big)^{\frac{d+1-\gamma}{2\gamma}}\Big(\|\theta_0\|_{L^2}+\frac{\|S\|_{L^2}}{c_0^\frac{1}{2}\kappa^\frac{1}{2}}\Big)+\|S\|_{L^\infty}\Big],
\end{align}
where $C=C(d)>0$ is a constant which only depends on the dimension $d$ and is independent of $\lambda$.
\end{lemma}

\begin{remark}
By exploiting the bounds \eqref{infty bound on theta with kappa>0} and \eqref{infty bound on theta in terms of L2 norm of initial theta} on $\|\theta\|_{L^\infty}$, for $\theta_0\in H^s$ with $s\ge0$, we have the following bound for $t\ge1$ (see also \cite[Theorem~3.2]{CZV16}):
\begin{align}\label{infty bound on theta with kappa>0 and t>1}
\|\thetalambda(t)\|_{L^\infty}\le\frac{C}{\kappa}\left[\|\theta_0\|_{L^2}+\frac{\|S\|_{L^2}}{\kappa^\frac{1}{2}}\right]e^{-c_0\kappa t}+\frac{1}{c_0\kappa}\|S\|_{L^\infty},
\end{align}
for some constant $C>0$ independent of $\lambda$. 
\end{remark}
\begin{remark}
The bound \eqref{infty bound on theta with kappa>0 and t>1} can be regarded as an alternative $L^\infty$ bound for $\theta$ compared with the bound \eqref{bound on theta by damping} given in Proposition~\ref{uniform bound on theta kappa>0 prop}, in the sense that the right side of \eqref{infty bound on theta with kappa>0 and t>1} is independent of $\lambda$.
\end{remark}

\begin{proof}[Proof of Lemma~\ref{De Giorgi iteration lemma}]
Fix $\lambda\ge0$ and $\kappa>0$, we drop $(\lambda)$ from $\thetalambda$ and write $\theta=\thetalambda$ for simplicity. Suppose that $M\ge2\|S\|_{L^\infty}$, where $M>0$ to be fixed later, we define
\begin{align*}
 L_{n}:=M(1-2^{-n}),\qquad n\in\N\cup\{0\},
\end{align*}
and $\theta_{n}$ to be the truncated function $\theta_n=\max\{\theta(t)- L_{n},0\}$. Fix $t\in(0,1]$, we define the time cutoffs by
\begin{align*}
T_{n}:=t(1-2^{-n}),
\end{align*}
and we denote the level set of energy by
\begin{align*}
Q_{n}:=\sup_{T_{n}\le\tau\le1}\|\theta_n(\cdot,\tau)\|^2_{L^2}+2\kappa\int_{T_{n}}^1\|\Lambda^{\frac{\gamma}{2}}\theta_{n}(\cdot,\tau)\|^2_{L^2}d\tau.
\end{align*}
Using the point-wise inequality given in \cite[Proposition 2.3]{CC04}, for all $s\in(T_{n-1},T_{n})$, we have the following level set inequality:
\begin{align*}
&\sup_{T_{n}\le\tau\le1}\|\theta_{n}(\cdot,\tau)\|^2_{L^2}+2\lambda\int_{T_{n}}^1\|\Lambda^{\frac{1}{2}}\theta_{n}(\cdot,\tau)\|^2_{L^2}d\tau\\
&\qquad+2\kappa\int_{T_{n}}^1\|\Lambda^{\frac{\gamma}{2}}\theta_{n}(\cdot,\tau)\|^2_{L^2}d\tau\le\|\theta_{n}(\cdot,s)\|^2_{L^2}+2\|S\|_{L^\infty}\int_{T_{n-1}}^1\|\theta_{n}(\cdot,\tau)\|_{L^1}d\tau.
\end{align*}
Since $\lambda\ge0$, we take the mean value in $s$ on $(T_{n-1},T_{n})$ and multiply by $\frac{2^n}{t}$ to obtain
\begin{align}\label{bound on Qn}
Q_{n}\le \frac{2^{n}}{t}\int_{T_{n-1}}^1\|\theta_{n}(\cdot,\tau)\|^2_{L^2}d\tau+2\|S\|_{L^\infty}\int_{T_{n-1}}^1\|\theta_{n}(\cdot,\tau)\|_{L^1}d\tau
\end{align}
and using \eqref{energy inequality kappa>0}, we also have
\begin{align}\label{bound on Q0}
Q_{0}\le\|\theta_0\|^2_{L^2}+\frac{1}{c_0\kappa}\|S\|^2_{L^2}.
\end{align}
We aim at bounding the right side of \eqref{bound on Qn} by a power of $Q_{n-1}$. Using H\"{o}lder inequality and Sobolev embedding, there exists $C_d>0$ which depends on $d$ such that
\begin{align}\label{bound on Qn-1 using theta n-1}
Q_{n-1}\ge C_{d}\|\theta_{n-1}\|^2_{L^\frac{2(d+1)}{d+1-\gamma}(\Omega\times[T_{n-1},1])}\qquad\forall n\in\N.
\end{align}
Since $\theta_{n-1}\ge 2^{-n}M$ on the set $\{(x,\tau):\theta_{n}(x,\tau)>0\}$, together with \eqref{bound on Qn-1 using theta n-1}, we have
\begin{align*}
\frac{2^{n}}{t}\int_{T_{n}}^1\|\theta_{n}(\cdot,\tau)\|^2_{L^2}d\tau\le 2C_{d}\Big(\frac{2^{n(\frac{d+1+\gamma}{d+1-\gamma})}}{tM^{\frac{2\gamma}{d+1-\gamma}}}\Big)Q_{n-1}^\frac{d+1}{d+1-\gamma}.
\end{align*}
Similarly, we have
\begin{align*}
2\|S\|_{L^\infty}\int_{T_{n-1}}^1\|\theta_{n}(\cdot,\tau)\|_{L^1}d\tau\le 2C_{d}\|S\|_{L^\infty}\Big(\frac{2^{n(\frac{d+1+\gamma}{d+1-\gamma})}}{M^{\frac{2\gamma}{d+1-\gamma}+1}}\Big)Q_{n-1}^\frac{d+1}{d+1-\gamma}.
\end{align*}
Since $M\ge2\|S\|_{L^\infty}$, we further obtain
\begin{align*}
Q_{n}\le \Big(\frac{2C_{d}}{t}+C_{d}\Big)\Big(\frac{2^{n(\frac{d+1+\gamma}{d+1-\gamma})}}{M^{\frac{2\gamma}{d+1-\gamma}}}\Big)Q_{n-1}^\frac{d+1}{d+1-\gamma}.
\end{align*}
If we assume that
\begin{align}\label{bounding condition on M}
M\ge\Big(\frac{2C_d}{t}+C_d\Big)^\frac{d+1-\gamma}{2\gamma}Q_{0}^\frac{1}{2},
\end{align} 
then $Q_{n}\to0$ as $n\to\infty$. Hence using \eqref{bound on Q0}, if we choose $M>0$ such that
\begin{align*}
M\ge\Big(\frac{2C_d}{t}+C_d\Big)^\frac{d+1-\gamma}{2\gamma}\Big(\|\theta_0\|_{L^2}+\frac{\|S\|_{L^2}}{c_0^\frac{1}{2}\kappa^\frac{1}{2}}\Big)+2\|S\|_{L^\infty},
\end{align*}
then it implies that $\theta$ is bounded above by $M$. Applying the same argument to $-\theta$, we conclude that the bound \eqref{infty bound on theta in terms of L2 norm of initial theta} holds for all $t\in(0,1]$.
\end{proof}
By exploiting the bound \eqref{infty bound on theta in terms of L2 norm of initial theta} on $\|\thetalambda(\cdot,t)\|_{L^\infty}$, we give the following proof for Theorem~\ref{global-in-time wellposedness in Sobolev}.

\begin{proof}[Proof of Theorem~\ref{global-in-time wellposedness in Sobolev}]
It is enough to establish an {\it a priori} estimate on $\theta^{(\lambda)}$ with initial data $\theta_0\in H^s$ for $s\ge0$. We multiply \eqref{abstract active scalar eqn} by $\Lambda^{2s}\thetalambda$ and integrate over $\Omega$ to obtain
\begin{align}\label{differential inequality on theta}
&\frac{1}{2}\frac{d}{dt}\|\Lambda^s\theta^{(\lambda)}\|^2_{L^2}+\lambda\Big(\|\Lambda^{s}\theta^{(\lambda)}\|^2_{L^2}+\|\Lambda^{s+\frac{1}{2}}\theta^{(\lambda)}\|^2_{L^2}\Big)+\kappa\|\Lambda^{s+\frac{\gamma}{2}}\theta^{(\lambda)}\|^2_{L^2}\notag\\
&\qquad\le\Big|\int_{\Omega}S\Lambda^{2s}\theta^{(\lambda)}\Big|+\Big|\int_{\Omega}u^{(\lambda)}\cdot\nabla\theta^{(\lambda)}\Lambda^{2s}\theta^{(\lambda)}\Big|,
\end{align}
where $u^{(\lambda)}=u[\thetalambda]$. The term $\dis\Big|\int_{\Omega}S\Lambda^{2s}\theta^{(\lambda)}\Big|$ is readily bounded by $\dis\|\Lambda^sS\|_{L^2}\|\Lambda^s\theta^{(\lambda)}\|_{L^2}$, and for $\dis\Big|\int_{\Omega}u^{(\lambda)}\cdot\nabla\theta^{(\lambda)}\Lambda^{2s}\theta^{(\lambda)}\Big|$, it can be estimated as follows.
\begin{align}\label{estimate on convective term}
&\Big|\int_{\Omega}u^{(\lambda)}\cdot\nabla\theta^{(\lambda)}\Lambda^{2s}\theta^{(\lambda)}\Big|\notag\\
&=\Big|\int_{\Omega}(u^{(\lambda)}\theta^{(\lambda)})\cdot\nabla\Lambda^{2s}\theta^{(\lambda)}\Big|\notag\\
&\le\Big|\int_{\Omega}[\Lambda^{s+1}(u^{(\lambda)}\theta^{(\lambda)})-u^{(\lambda)}\Lambda^{s+1}\theta^{(\lambda)}]\cdot\nabla\Lambda^{s-1}\theta^{(\lambda)}\Big|+\Big|\int_{\Omega}u^{(\lambda)}\Lambda^{s+1}\theta^{(\lambda)}\cdot\nabla\Lambda^{s-1}\theta^{(\lambda)}\Big|.
\end{align}
Since $\nabla\cdot u^{(\lambda)}=0$, the second term on the right side of \eqref{estimate on convective term} can be replaced by
\begin{align*}
\Big|\int_{\Omega}u^{(\lambda)}\Lambda^{s+1}\theta^{(\lambda)}\cdot\nabla\Lambda^{s-1}\theta^{(\lambda)}\Big|=\Big|\int_{\Omega}\Lambda^{s}\theta^{(\lambda)}(\Lambda(u^{(\lambda)}\cdot\Lambda^{s-1}\theta^{(\lambda)})-u^{(\lambda)}\cdot\Lambda^{s}\theta^{(\lambda)})\Big|.
\end{align*}
Hence using \eqref{commutator estimate}, we have
\begin{align*}
\Big|\int_{\Omega}u^{(\lambda)}\Lambda^{s+1}\theta^{(\lambda)}\cdot\nabla\Lambda^{s-1}\theta^{(\lambda)}\Big|=&\Big|\int_{\Omega}\Lambda^{s}\theta^{(\lambda)}(\Lambda(u^{(\lambda)}\cdot\Lambda^{s-1}\theta^{(\lambda)})-u^{(\lambda)}\cdot\Lambda^{s}\theta^{(\lambda)})\Big|\\
&\le C\|\Lambda^{s}\theta^{(\lambda)}\|_{L^2}\|\nabla u^{(\lambda)}\|_{L^\infty}\|\nabla\Lambda^{s-1}\theta^{(\lambda)}\|_{L^2}\\
&\le C\|\Lambda u^{(\lambda)}\|_{L^\infty}\|\Lambda^{s}\theta^{(\lambda)}\|^2_{L^2}.
\end{align*}
For the leading term on the right side of \eqref{estimate on convective term}, we apply \eqref{commutator estimate} to obtain
\begin{align*}
&\Big|\int_{\Omega}[\Lambda^{s+1}(u^{(\lambda)}\theta^{(\lambda)})-u^{(\lambda)}\Lambda^{s+1}\theta^{(\lambda)}]\cdot\nabla\Lambda^{s-1}\theta^{(\lambda)}\Big|\\
&\le \|\Lambda^{s+1}(u^{(\lambda)}\theta^{(\lambda)})-u^{(\lambda)}\Lambda^{s+1}\theta^{(\lambda)}\|_{L^2}\|\nabla\Lambda^{s-1}\theta^{(\lambda)}\|_{L^2}\\
&\le\Big(\|\Lambda u^{(\lambda)}\|_{L^\infty}\|\Lambda^{s}\theta^{(\lambda)}\|_{L^2}+\|\Lambda^{s+1}u^{(\lambda)}\|_{L^2}\|\theta^{(\lambda)}\|_{L^\infty}\Big)\|\nabla\Lambda^{s-1}\theta^{(\lambda)}\|_{L^2}.
\end{align*}
Hence we get from \eqref{estimate on convective term} that 
\begin{align}\label{estimate on convective term step 2}
&\Big|\int_{\Omega}u^{(\lambda)}\cdot\nabla\theta^{(\lambda)}\Lambda^{2s}\theta^{(\lambda)}\Big|\notag\\
&\le C\Big(\|\Lambda u^{(\lambda)}\|_{L^\infty}\|\Lambda^{s}\theta^{(\lambda)}\|_{L^2}+\|\Lambda^{s+1}u^{(\lambda)}\|_{L^2}\|\theta^{(\lambda)}\|_{L^\infty}\Big)\|\Lambda^{s}\theta^{(\lambda)}\|_{L^2}.
\end{align}
By \eqref{two order smoothing for u when nu>0}, the term $\|\Lambda^{s+1}u^{(\lambda)}\|_{L^2}$ can be bounded by $C\|\Lambda^{s}\theta^{(\lambda)}\|_{L^2}$, and for the term $\|\Lambda u^{(\lambda)}\|_{L^\infty}$, using \eqref{L infty bound 2} for $q=d+1$ and together with \eqref{two order smoothing for u when nu>0}, we have
\begin{align}\label{L infty bound on nabla u}
\|\Lambda u^{(\lambda)}\|_{L^\infty}&\le C\|\Lambda u^{(\lambda)}\|_{W^{1,d+1}}\notag\\
&\le C\|\theta^{(\lambda)}\|_{L^{d+1}}\le C\|\thetalambda\|_{L^\infty}.
\end{align}
Therefore, we obtain from \eqref{estimate on convective term step 2} that
\begin{align}\label{estimate on convective term step 3}
\Big|\int_{\Omega}u^{(\lambda)}\cdot\nabla\theta^{(\lambda)}\Lambda^{2s}\theta^{(\lambda)}\Big|\le C\|\theta^{(\lambda)}\|_{L^\infty}\|\Lambda^{s}\theta^{(\lambda)}\|^2_{L^2}.
\end{align}
We apply \eqref{estimate on convective term step 3} on \eqref{differential inequality on theta} and infer that
\begin{align}\label{differential inequality on theta final}
&\frac{d}{dt}\|\Lambda^s\theta^{(\lambda)}\|^2_{L^2}+2\lambda\Big(\|\Lambda^{s}\theta^{(\lambda)}\|^2_{L^2}+\|\Lambda^{s+\frac{1}{2}}\theta^{(\lambda)}\|^2_{L^2}\Big)+2\kappa\|\Lambda^{s+\frac{\gamma}{2}}\theta^{(\lambda)}\|^2_{L^2}\notag\\
&\qquad\le C\|\theta^{(\lambda)}\|_{L^\infty}\|\Lambda^{s}\theta^{(\lambda)}\|^2_{L^2}+2\|\Lambda^sS\|_{L^2}\|\Lambda^s\theta^{(\lambda)}\|_{L^2}.
\end{align}
Using the bounds \eqref{infty bound on theta in terms of L2 norm of initial theta} and \eqref{infty bound on theta with kappa>0 and t>1} on $\theta^{(\lambda)}$, for each $\tau>0$, there exists $C_\tau >0$ which depends on $\kappa$, $\tau$, $c_0$, $\|\theta_0\|_{L^2}$, $\|S\|_{L^\infty}$ but independent of $\lambda$ and $t$ such that
\begin{equation}\label{L infty bound on theta for t>tau}
\|\theta^{(\lambda)}(\cdot,t)\|_{L^\infty}\le C_\tau ,\qquad \forall t\ge\tau.
\end{equation}
Furthermore, by the continuity in time as proved in Lemma~\ref{local-in-time existence lemma}, we choose $\tau>0$ sufficiently small such that
\[\sup_{0\le t\le\tau}\|\Lambda^{s}\theta^{(\lambda)}(\cdot,t)\|_{L^2}\le 2\|\Lambda^{s}\theta_0\|_{L^2},\]
Hence with the help of Gr\"{o}nwall's inequality, we conclude from \eqref{differential inequality on theta final} that for $\kappa>0$,
\begin{align}\label{H^s bound for theta kappa>0 and nu>0}
\|\Lambda^{s}\theta^{(\lambda)}(\cdot,t)\|_{L^2}+2\lambda\int_{0}^{t}\Big(\|\Lambda^{s}\theta^{(\lambda)}\|^2_{L^2}+\|\Lambda^{s+\frac{1}{2}}\theta^{(\lambda)}\|^2_{L^2}\Big)+2\kappa\int_{0}^{t}\|\Lambda^{s+\frac{\gamma}{2}}\theta^{(\lambda)}\|^2_{L^2}\notag\\
\le (2\|\Lambda^{s}\theta_0\|_{L^2}+C\|\Lambda^{s}S\|_{L^2})(e^{CC_\tau t}+1),\qquad \forall t\ge0.
\end{align}
Therefore, the above inequality implies that $\|\Lambda^{s}\theta^{(\lambda)}(\cdot,t)\|_{L^2}$ remains finite for all positive time when $\lambda\ge0$ and $\kappa>0$.
\end{proof}

\subsection{Convergence of $H^s$-solutions as $\lambda\to0$}\label{convergence of Hs solution subsec}

We are now ready to give the proof of Theorem~\ref{Hs convergence theorem} which shows the convergence of f $H^s$-solutions as claimed by \eqref{Hs convergence thm}. 

\begin{proof}[Proof of Theorem~\ref{Hs convergence theorem}]
We let $\theta_0$, $S\in C^\infty$. The proof is divided into three parts.

\medskip

\noindent{\bf Uniform $H^s$-bound:}
For fixed $\lambda\ge0$, we first obtain a uniform (independent of $\lambda$) $H^{s}$-bound for all $s\ge0$ on $\theta^{(\lambda)}$. Apply the estimate \eqref{differential inequality on theta final} on $\theta^{(\lambda)}$, we have
\begin{align}\label{differential inequality on theta s+1}
&\frac{d}{dt}\|\Lambda^{s}\theta^{(\lambda)}\|^2_{L^2}+2\lambda\Big(\|\Lambda^{s}\theta^{(\lambda)}\|^2_{L^2}+\|\Lambda^{s+\frac{1}{2}}\theta^{(\lambda)}\|^2_{L^2}\Big)+2\kappa\|\Lambda^{s+\frac{\gamma}{2}}\theta^{(\lambda)}\|^2_{L^2}\notag\\
&\qquad\le C\|\theta^{(\lambda)}\|_{L^\infty}\|\Lambda^{s+1}\theta^{(\lambda)}\|^2_{L^2}+2\|\Lambda^{s}S\|_{L^2}\|\Lambda^{s}\theta^{(\lambda)}\|_{L^2}.
\end{align}
Since $\theta_0\in C^\infty\subset L^\infty$, we can apply the uniform bound \eqref{infty bound on theta with kappa=0 or >0} to deduce from \eqref{differential inequality on theta s+1} that 
\begin{align*}
&\frac{d}{dt}\|\Lambda^{s}\theta^{(\lambda)}\|^2_{L^2}+2\kappa\|\Lambda^{s+\frac{\gamma}{2}}\theta^{(\lambda)}\|^2_{L^2}\\
&\le C(\|\theta_0\|_{L^\infty}+\|S\|_{L^\infty})\|\Lambda^{s}\theta^{(\lambda)}\|^2_{L^2}+2\|\Lambda^{s}S\|_{L^2}\|\Lambda^{s}\theta^{(\lambda)}\|_{L^2}.
\end{align*}
Upon integrating the above inequality from $0$ to $t$, we obtain, for all $\kappa\ge0$ that
\begin{align}\label{uniform Hs+1 bound on theta for all kappa}
&\|\Lambda^{s}\theta^{(\lambda)}(\cdot,t)\|_{L^2}\notag\\
&\le C(\|\Lambda^{s}\theta_0\|_{L^2}+\|\Lambda^{s}S\|_{L^2})\exp\Big(Ct(\|\theta_0\|_{L^\infty}+\|S\|_{L^\infty})\Big),\qquad\forall t>0.
\end{align}

\noindent{\bf $L^2$-convergence:}
Fix $\kappa>0$ and $\gamma\in(0,2]$. We let $\theta^{(\lambda)}$, $\theta^{(0)}$ be the smooth solution to \eqref{abstract active scalar eqn} for $\lambda>0$ and $\lambda=0$ respectively. Define $\varphi=\theta^{(\lambda)}-\theta^{(0)}$, then $\varphi$ satisfies
\begin{align}\label{equation for varphi}
\partial_t\varphi+(u^{(\lambda)}-u^0)\cdot\nabla\theta^{(0)}+u^{(\lambda)}\cdot\nabla\varphi=-\kappa\Lambda^{\gamma}\varphi-\lambda\varphi-\lambda\Lambda^\frac{1}{2}\varphi,
\end{align}
where $u^{(\lambda)}$ and $u^0$ are given by
\begin{align*}
u_j^\lambda:=\partial_{x_i} T_{ij}[\theta^{(\lambda)}],\qquad u_j^0:=\partial_{x_i} T_{ij}[\theta^{(0)}]
\end{align*}
for $1\le i,j\le d$. Multiply \eqref{equation for varphi} by $\varphi$ and integrate,
\begin{align}\label{integral of varphi}
&\frac{1}{2}\frac{d}{dt}\|\varphi(\cdot,t)\|^2_{L^2}+\kappa\|\Lambda^{\frac{\gamma}{2}}\varphi\|^2_{L^2}\notag\\
&\qquad=-\int_{\Omega}(u^{(\lambda)}-u^0)\cdot\nabla\theta^{(0)}\cdot\varphi-\int_{\Omega}u^{(\lambda)}\cdot\nabla\varphi\cdot\varphi-\lambda\int_{\Omega}\theta^{(\lambda)}\cdot\varphi-\lambda\int_{\Omega}\Lambda\theta^{(\lambda)}\cdot\varphi.
\end{align}
Using the divergence-free assumption on $u^{(\lambda)}$, the term $\dis-\int_{\Omega}u^{(\lambda)}\cdot\nabla\varphi\cdot\varphi$ is zero. On the other hand, using \eqref{two order smoothing for u when nu>0}, we have
\begin{equation*}
\|(u^{(\lambda)}-u^0)(\cdot,t)\|_{L^2}\le C\|(\theta^{(\lambda)}-\theta^{(0)})(\cdot,t)\|_{L^2}.
\end{equation*}
Hence the first term on the right side of \eqref{integral of varphi} can be bounded as follows.
\begin{align*}
-\intox(u^{(\lambda)}-u^0)\cdot\nabla\theta^{(0)}\cdot\varphi&\le\|\nabla\theta^{(0)}(\cdot,t)\|_{L^\infty}\|(u^{(\lambda)}-u^0)(\cdot,t)\|_{L^2}\|\varphi(\cdot,t)\|_{L^2}\notag\\
&\le C\|\nabla\theta^{(0)}(\cdot,t)\|_{L^\infty}\|\varphi(\cdot,t)\|_{L^2}^2.
\end{align*}
For the third and fourth terms on the right side of \eqref{integral of varphi}, we can bound them by
\begin{align*}
-\lambda\int_{\Omega}\theta^{(\lambda)}\cdot\varphi-\lambda\int_{\Omega}\Lambda\theta^{(\lambda)}\cdot\varphi\le4\lambda\|\Lambda\theta^{(\lambda)}(\cdot,t)\|^2_{L^2}+4\lambda\|\theta^{(\lambda)}(\cdot,t)\|^2_{L^2}+\frac{\lambda}{2}\|\varphi(\cdot,t)\|^2_{L^2}.
\end{align*}
Applying the above estimates on \eqref{integral of varphi}, we obtain
\begin{align}\label{differential of varphi}
&\frac{1}{2}\frac{d}{dt}\|\varphi(\cdot,t)\|_{L^2}^2+\kappa\|\Lambda^{\frac{\gamma}{2}}\varphi\|^2_{L^2}\notag\\
&\le4\lambda\|\Lambda\theta^{(\lambda)}(\cdot,t)\|^2_{L^2}+4\lambda\|\theta^{(\lambda)}(\cdot,t)\|^2_{L^2}+\frac{\lambda}{2}\|\varphi(\cdot,t)\|^2_{L^2}+C\|\nabla\theta^{(0)}(\cdot,t)\|_{L^\infty}\|\varphi(\cdot,t)\|_{L^2}^2\notag\\
&\le\Big[C\|\nabla\theta^{(0)}(\cdot,t)\|_{L^\infty}+\frac{\lambda}{2}\Big]\|\varphi(\cdot,t)\|_{L^2}^2+4\lambda\|\Lambda\theta^{(\lambda)}(\cdot,t)\|^2_{L^2}+4\lambda\|\theta^{(\lambda)}(\cdot,t)\|^2_{L^2}.
\end{align}
Using the bound \eqref{uniform Hs+1 bound on theta for all kappa}, for sufficiently large $s>1$, we have
\begin{align*}
\|\nabla\theta^{(0)}(\cdot,t)\|_{L^\infty}&\le C\|\Lambda^{s}\theta^{(\lambda)}(\cdot,t)\|_{L^2}\\
&\le C(\|\Lambda^{s}\theta_0\|_{L^2}+\|\Lambda^{s}S\|_{L^2})\exp\Big(Ct(\|\theta_0\|_{L^\infty}+\|S\|_{L^\infty})\Big),
\end{align*}
as well as
\begin{align*}
\|\Lambda\theta^{(\lambda)}\|_{H^\frac{1}{2}}\le C(\|\Lambda^{s}\theta_0\|_{L^2}+\|\Lambda^{s}S\|_{L^2})\exp\Big(Ct(\|\theta_0\|_{L^\infty}+\|S\|_{L^\infty})\Big).
\end{align*}
Hence by integrating \eqref{differential of varphi} over $t$ and using Gr\"{o}nwall's inequality, for $t>0$ there exist positive functions $C_1(t)$, $C_2(t)$ depending on $t$, $\theta_0$ and $S$ such that
\begin{align*}
\|\varphi(\cdot,t)\|_{L^2}^2\le \lambda C_1(t)e^{C_2(t)},
\end{align*}
where we recall that $\varphi(\cdot,0)=0$. Therefore, as $\lambda\rightarrow0$, we have 
\begin{align}\label{L2 convergence}
\lim_{\lambda\rightarrow0}\int_{\Omega}|\theta^{(\lambda)}-\theta^{(0)}|^2(x,t)dx=\lim_{\lambda\rightarrow0}\|\varphi(\cdot,t)\|_{L^2}^2=0.
\end{align}

\noindent{\bf $H^s$-convergence:} By Sobolev inequality, for $s>0$, there exists $\sigma=\sigma(s)\in(0,1)$ such that for $t>0$,
\begin{align}\label{Hs estimate on varphi}
\|(\theta^{(\lambda)}-\theta^{(0)})(\cdot,t)\|_{H^s}\le\|(\theta^{(\lambda)}-\theta^{(0)})(\cdot,t)\|_{L^2}^\sigma\|(\theta^{(\lambda)}-\theta^{(0)})(\cdot,t)\|_{H^{s+1}}^{1-\sigma}.
\end{align}
Using the uniform bound \eqref{uniform Hs+1 bound on theta for all kappa} on $\|(\theta^{(\lambda)}-\theta^{(0)})(\cdot,t)\|_{H^{s+1}}^{1-\sigma}$ and the $L^2$-convergence result \eqref{L2 convergence}, by taking $\lambda\to0$ on \eqref{Hs estimate on varphi}, the result \eqref{Hs convergence thm} holds for all $s>0$ as well. 
\end{proof}
\begin{remark}[Convergence of solutions as $\kappa\to0$]
For each $\lambda>0$ and $\gamma\in(0,2]$, suppose that $\theta_0,S\in C^\infty$ are the initial datum and forcing term respectively which satisfy \eqref{zero mean assumption on data and forcing}. Following the similar argument as given above, if $\theta^{(\kappa)}$ and $\theta^{(0)}$ are smooth solutions to \eqref{abstract active scalar eqn} for $\kappa>0$ and $\kappa=0$ respectively, then we also have
\begin{align}\label{Hs convergence thm kappa to 0} 
\lim_{\kappa\rightarrow0}\|(\theta^{(\kappa)}-\theta^{(0)})(\cdot,t)\|_{H^s}=0,
\end{align}
for all $s\ge0$ and $t\ge0$. We note that the case for $\lambda=0$ was proved in \cite[Theorem~2.1]{FS21}.
\end{remark}

\section{Long time dynamics of solutions when $\kappa>0$}\label{long time behaviour and attractors section}

In this section, we study the long time behaviour for solutions to the active scalar equations \eqref{abstract active scalar eqn} when $\lambda\ge$ and $\kappa>0$ with assumptions A1, A2, A3' and A4 are being in force. Based on the global-in-time existence results established in Theorem~\ref{global-in-time wellposedness in Sobolev}, for fixed $\lambda\ge0$ and $\kappa>0$, we can define a solution operator $\pin(t)$ for the initial value problem \eqref{abstract active scalar eqn} via
\begin{align}\label{def of solution map nu>0 and kappa>0}
\pin(t): H^1\to H^1,\qquad \pin(t)\theta_0=\theta(\cdot,t),\qquad t\ge0.
\end{align}
We study the long-time dynamics of $\pin(t)$ on the phase space $H^1$. Specifically, we establish the existence of global attractors for $\pin(t)$ in $H^1$, which will be given in Subsection~\ref{existence of global attractors subsection}. Once we obtain the existence of global attractors, we further address some properties for the attractors which will be explained in Subsection~\ref{Further properties for the attractors subsection}. The results obtained in Subsection~\ref{existence of global attractors subsection} and Subsection~\ref{Further properties for the attractors subsection} will be sufficient for proving Theorem~\ref{existence of global attractor theorem}.


\subsection{Existence of global attractors in $H^1$-space}\label{existence of global attractors subsection}

The following theorem gives the main results for this subsection:

\begin{theorem}[Existence of $H^1$-global attractor]\label{Existence of H1 global attractor}
Let $S\in L^\infty\cap H^1$. For $\lambda\ge0$, $\kappa>0$ and $\gamma\in(0,2]$, the solution map $\pin(t):H^1\to H^1$ associated to \eqref{abstract active scalar eqn} possesses a unique global attractor $\Gg$ in $H^1$. Moreover, there exists $\tilde{M}_{\Gg}$ which depends on $\lambda$, $\kappa$, $\gamma$, $\|S\|_{L^\infty\cap H^1}$ and universal constants, such that if $\theta_0\in\Gg$, we have
\begin{align}\label{uniform bound on theta from the attractor} 
\|\theta(\cdot,t)\|_{H^{1+\frac{\gamma}{2}}}\le \tilde{M}_{\Gg},\qquad\forall t\ge0,
\end{align}
and
\begin{align}\label{uniform bound on time integral of theta from the attractor}
\frac{1}{T}\int_t^{t+T}\|\theta(\cdot,\tau)\|_{H^{1+\gamma}}d\tau\le M_{\Gg},\qquad \mbox{$\forall t\ge0$ and $T>0$,}
\end{align}
where $\theta(\cdot,t)=\pin(t)\theta_0$.
\end{theorem}


Theorem~\ref{Existence of H1 global attractor} will be proved in a sequence of lemmas. In the following lemma, we first show the existence of an $L^\infty$-absorbing set by using the bound \eqref{infty bound on theta with kappa>0 and t>1} on $\|\theta(t)\|_{L^\infty}$ for $t\ge1$.

\begin{lemma}[Existence of an $L^\infty$-absorbing set]\label{Existence of an L infty absorbing set lemma}
Let $c_0>0$ be defined in Proposition~\ref{boundedness on theta with kappa>0 prop}. Then the set 
\begin{align*}
B_{\infty}=\left\{\phi\in L^\infty\cap H^1:\|\phi\|_{L^\infty}\le\min\{\frac{2}{c_0\kappa},\frac{2}{c_0\lambda}\}\times\|S\|_{L^\infty}\right\}
\end{align*}
is an absorbing set for $\pin(t)$. Moreover, using \eqref{infty bound on theta with kappa>0} and \eqref{bound on theta by damping} for $p=\infty$, we have
\begin{align}\label{bound on solution map in L infty}
\sup_{t\ge0}\sup_{\theta_0\in B_{\infty}}\|\pin(t)\theta_0\|_{L^\infty}\le\min\{\frac{3}{c_0\kappa},\frac{3}{c_0\lambda}\}\times\|S\|_{L^\infty}.
\end{align}
\end{lemma}

\begin{proof}
The proof follows by the same argument given in \cite[Theorem~3.1]{CZV16}. For a fixed bounded set $B\subset H^1$, we let 
\[R=\sup_{\phi\in B}\|\phi\|_{H^1}.\]
Using the bound \eqref{infty bound on theta with kappa>0 and t>1} and the Poincar\'{e} inequality, we conclude that if $\theta_0\in B$, then
\begin{align*}
\|\pin(t)\theta_0\|_{L^\infty}\le\frac{C}{\kappa}\left[R+\frac{\|S\|_{L^\infty}}{\kappa^\frac{1}{2}}\right]e^{-c_0\kappa t}+\frac{\|S\|_{L^\infty}}{c_0\kappa},\qquad\forall t\ge1.
\end{align*}
Similar, using the bound \eqref{bound on theta by damping} for $p=\infty$ and the Poincar\'{e} inequality, we also have
\begin{align*}
\|\pin(t)\theta_0\|_{L^\infty}\le\frac{C}{\lambda}\left[R+\frac{\|S\|_{L^\infty}}{\lambda^\frac{1}{2}}\right]e^{-c_0\lambda t}+\frac{\|S\|_{L^\infty}}{c_0\lambda},\qquad\forall t\ge1.
\end{align*}
Choose $t_B=t_B(R,\|S\|_{L^2\cap L^\infty})\ge1$ such that
\begin{align*}
\|\pin(t)\theta_0\|_{L^\infty}\le\min\{\frac{2}{c_0\kappa},\frac{2}{c_0\lambda}\}\times\|S\|_{L^\infty},
\end{align*}
then we have $\pin(t)\theta_0\in B_{\infty}$ and hence $B_\infty$ is absorbing.
\end{proof}

Next we prove the following lemma which gives the necessary {\it a priori} estimate in $C^{\alpha}$-space with some appropriate exponent $\alpha\in(0,1)$. In view of Lemma~\ref{Existence of an L infty absorbing set lemma}, we can see that the solutions to \eqref{abstract active scalar eqn} emerging from data in a bounded subset of $H^1$ are absorbed in finite time by $B_{\infty}$. Hence we assume that $\theta_0\in L^\infty$ and derive {\it a priori} bounds in terms of $\|\theta_0\|_{L^\infty}$.

\begin{lemma}[Estimates in $C^\alpha$-space]\label{Estimates in C alpha-space lemma}
Assume that $\theta_0\in H^1\cap L^\infty$ and fix $\lambda\ge0$ and $\kappa>0$. There exists $\alpha=\alpha(\gamma)\in(0,\frac{\gamma}{3+\gamma}]$ which depends on $\|\theta_0\|_{L^\infty}$, $\|S\|_{L^\infty}$, $\kappa$, $\gamma$ such that
\begin{align}\label{bound on C alpha norm of theta with gamma}
\|\theta(\cdot,t)\|_{C\alpha}\le C(\K+\bK),\qquad\forall t\ge t_{\alpha}:=\frac{3}{2\gamma(1-\alpha)},
\end{align}
where $C>0$ is a positive constant, $\K$ and $\bK$ are given respectively by
{\footnotesize
\begin{align}\label{def of K and bar K}
\K:=\|\theta_0\|_{L^\infty}+\frac{1}{c_0\kappa}\|S\|_{L^\infty},\,\,\,\bK:=\kappa^{-\frac{1}{\gamma}}\K+\|S\|^\frac{2+\gamma}{2(1+\gamma)}_{L^\infty}\kappa^{-\frac{1}{2(1+\gamma)}}\K^{\frac{\gamma}{2(1+\gamma)}}+\kappa^{-\frac{1}{\gamma}}\K^{\frac{6+\gamma}{4}}.
\end{align}
}
For the case when $\lambda>0$ and $\gamma\le1$, there exists $\alpha\in(0,\frac{1}{4}]$ which depends on $\|\theta_0\|_{L^\infty}$, $\|S\|_{L^\infty}$, $\lambda$ but is independent of $\gamma$ such that 
\begin{align}\label{bound on C alpha norm of theta with gamma=1}
\|\theta(\cdot,t)\|_{C\alpha}\le C(K_{1,\infty}+\bar{K}_{1,\infty}),\qquad\forall t\ge \frac{3}{2(1-\alpha)},
\end{align}
where $C>0$ is a positive constant, $K_{1,\infty}$ and $\bar{K}_{1,\infty}$ are given respectively by
{\footnotesize
\begin{align}\label{def of K and bar K gamma=1}
K_{1,\infty}:=\|\theta_0\|_{L^\infty}+\frac{1}{c_0\lambda}\|S\|_{L^\infty},\,\,\,\bar{K}_{1,\infty}:=\lambda^{-1}K_{1,\infty}+\|S\|^\frac{3}{4}_{L^\infty}\lambda^{-\frac{1}{4}}K_{1,\infty}^{\frac{1}{4}}+\lambda^{-1}K_{1,\infty}^{\frac{7}{4}}.
\end{align}
}
\end{lemma}

\begin{remark}
Using the bound \eqref{infty bound on theta with kappa>0}, under the assumption that $\theta_0\in L^\infty$, for $\kappa>0$, we have
\begin{align}\label{infty bound on theta using K for kappa>0}
\|\theta(\cdot,t)\|_{L^\infty}\le \K,\qquad \forall t\ge0.
\end{align}
\end{remark}

\begin{proof}[Proof of Lemma~\ref{Estimates in C alpha-space lemma}]
The proof is reminiscent of the one given in proving \cite[Lemma~6.3]{FS21}. We introduce the following finite difference
\begin{align*}
\Dh\theta(x,t)=\theta(x+h)-\theta(x,t),
\end{align*}
where $x$, $h\in\Omega$. Then $\Dh\theta$ satisfies
\begin{align}\label{eqn for theta with operator Lh}
\Lh(\Dh\theta)^2+D_{\gamma}[\Dh\theta]=2(\delta_h S)(\delta_h \theta),
\end{align}
where $\Lh$ is the operator given by $$\Lh:=\dt+u\cdot\nabla+(\Dh u)\cdot\nabla_h+\lambda\mathcal{D}+\kappa\Lambda^\gamma,$$ and $D_{\gamma}[f](x)$ is the functional given by
\begin{align}\label{def of functional D gamma []}
D_{\gamma}[f](x)=P.V.\int_{\Omega}\frac{[f(x)-f(x+y)]^2}{|y|^{d+\gamma}}dy.
\end{align}
Let $\xi(t)=\xi:[0,\infty)\to[0,\infty)$ be a bounded decreasing differentiable function which will be defined later, and for $\alpha\in(0,\frac{\gamma}{3+\gamma}]$, we define $v$ by
\begin{align*}
v(x,t;h)=\frac{|\Dh\theta(x,t)|}{(\xi(t)^2+|h|^2)^\frac{\alpha}{2}}.
\end{align*}
Using \eqref{eqn for theta with operator Lh}, we can see that $v$ satisfies the following inequality
\begin{align}\label{ineq for v with operator Lh}
&\Lh v^2+\kappa\cdot\frac{D_{\gamma}[\Dh\theta(x,t)]}{(\xi^2+|h|^2)^\alpha}\notag\\
&\le\Lh v^2+\lambda\cdot\left(\frac{v^2}{\xi^2+|h|^2}+\frac{D_{1}[\Dh\theta(x,t)]}{(\xi^2+|h|^2)^\alpha}\right)+\kappa\cdot\frac{D_{\gamma}[\Dh\theta(x,t)]}{(\xi^2+|h|^2)^\alpha}\notag\\
&\le 2\alpha|\dot{\xi}|\frac{\xi}{\xi^2+|h|^2}v^2+2\alpha\frac{|h|}{\xi^2+|h|^2}|\Dh u|v^2+\frac{2\|S\|_{L^\infty}}{(\xi^2+|h|^2)^\frac{\alpha}{2}}v,
\end{align}
where $D_{1}=D_{\gamma}|_{\gamma=1}$. Following the details given in \cite{FS21}, the terms on the right side of \eqref{ineq for v with operator Lh} can be bounded by
\begin{align}\label{bound on the term with derivative on xi}
2\alpha|\dot{\xi}|\frac{\xi}{\xi^2+|h|^2}v^2\le \frac{\kappa|v|^{2+\gamma}}{8c_2\|\theta\|^\gamma_{L^\infty}(\xi^2+|h|^2)^\frac{\gamma(1-\alpha)(2+\gamma)}{6}}+C\kappa^{-\frac{2}{\gamma}}\|\theta\|^2_{L^\infty},
\end{align}
\begin{align}\label{bound on the term with nabla u}
2\alpha\frac{|h|}{\xi^2+|h|^2}|\Dh u|v^2\le \frac{\kappa|v|^{2+\gamma}}{8c_2\|\theta\|^\gamma_{L^\infty}(\xi^2+|h|^2)^\frac{\gamma(1-\alpha)(2+\gamma)}{6}}+C\|\nabla v\|^\frac{2+\gamma}{2}_{L^\infty}\kappa^{-\frac{2}{\gamma}}\|\theta\|^2_{L^\infty}
\end{align}
and
\begin{align}\label{bound on the term with S}
\frac{2\|S\|_{L^\infty}}{(\xi^2+|h|^2)^\frac{\alpha}{2}}v\le\frac{\kappa|v|^{2+\gamma}}{8c_2\|\theta\|^\gamma_{L^\infty}(\xi^2+|h|^2)^\frac{\gamma(1-\alpha)(2+\gamma)}{6}}+C\|S\|^\frac{2+\gamma}{1+\gamma}_{L^\infty}\kappa^{-\frac{1}{1+\gamma}}\|\theta\|^{\frac{\gamma}{1+\gamma}}_{L^\infty}.
\end{align}
The bounds \eqref{bound on the term with derivative on xi}-\eqref{bound on the term with S} hold upon choosing $\xi=\xi(t)$ that satisfies
\begin{equation} \label{def of xi 2}
			\xi(t)=\begin{cases}
				[1-t(1-a)]^\frac{1}{1-a},\qquad &t\in[0,t_{\alpha}] \\ 
				0,\qquad &t\in(t_{\alpha},\infty),
			\end{cases}
\end{equation}
with $t_{\alpha}$ being given by
\begin{align}\label{def of t alpha}
t_{\alpha}=\frac{3}{2\gamma(1-\alpha)}.
\end{align}
On the other hand, for $r\ge4|h|$, there exists some constant $C\ge1$ such that
\begin{align*}
D_{\gamma}[\Dh\theta](x)\ge\frac{1}{2r^\gamma}|\Dh\theta(x)|^2-C|\Dh\theta(x)|\|\theta\|_{L^\infty}\frac{|h|}{r^{1+\gamma}}.
\end{align*}
Choose $r\ge4(\xi^2+|h|^2)^\frac{1}{2}\ge 4|h|$ such that
\begin{align*}
r=\frac{4C\|\theta\|_{L^\infty}}{|\Dh\theta(x)|}(\xi^2+|h|^2)^{\frac{(1-\alpha)(2+\gamma)}{6}+\frac{\alpha}{2}},
\end{align*}
then we have
\begin{align}\label{lower bound for Dh alpha}
\frac{D_{\gamma}[\Dh\theta](x,t)}{(\xi^2+|h|^2)^\alpha}\ge\frac{|v(x,t;h)|^{2+\gamma}}{c_2\|\theta\|^\gamma_{L^\infty}(\xi^2+|h|^2)^\frac{\gamma(1-\alpha)(2+\gamma)}{6}},
\end{align}
for some positive constant $c_2$. Apply the bounds \eqref{bound on the term with derivative on xi}, \eqref{bound on the term with nabla u}, \eqref{bound on the term with S} and  \eqref{lower bound for Dh alpha} on \eqref{ineq for v with operator Lh}, we obtain
\begin{align*}
&\Lh v^2+\kappa\cdot\frac{D_{\gamma}[\Dh\theta(x,t)]}{(\xi^2+|h|^2)^\alpha}+\frac{\kappa|v|^{2+\gamma}}{8c_2\|\theta\|^\gamma_{L^\infty}(\xi^2+|h|^2)^\frac{\gamma(1-\alpha)(2+\gamma)}{6}}\\
&\le C\Big[\kappa^{-\frac{2}{\gamma}}\|\theta\|^2_{L^\infty}+\|S\|^\frac{2+\gamma}{1+\gamma}_{L^\infty}\kappa^{-\frac{1}{1+\gamma}}\|\theta\|^{\frac{\gamma}{1+\gamma}}_{L^\infty}+\|\nabla v\|^\frac{2+\gamma}{2}_{L^\infty}\kappa^{-\frac{2}{\gamma}}\|\theta\|^2_{L^\infty}\Big].
\end{align*}
Moreover, we apply the bound \eqref{infty bound on theta using K for kappa>0} on $\|\theta\|_{L^\infty}$ to obtain
\begin{align*}
\|\theta\|_{L^\infty}\le\K,
\end{align*}
and using \eqref{L infty bound on nabla u}, we can bound $\nabla u$ by 
\begin{align}\label{L infty bound on nabla u general}
\|\nabla u\|_{L^\infty}\le C\|\theta\|_{L^\infty}\le C\K.
\end{align}
Since $\|S\|_{L^\infty}\le c_0\kappa\K$, it gives
\begin{align*}
\Lh v^2+\frac{\kappa|v|^{2+\gamma}}{8c_2\|\theta\|^\gamma_{L^\infty}(\xi^2+|h|^2)^\frac{\gamma(1-\alpha)(2+\gamma)}{6}}\le C\bK^2,
\end{align*}
where $\bK$ is defined in \eqref{def of K and bar K}. Since $(\xi^2+|h|^2)^\frac{\gamma(1-\alpha)(2+\gamma)}{6}\le 2^\frac{\gamma(1-\alpha)(2+\gamma)}{6}\le 4$, this implies
\begin{align}\label{ineq for v with operator Lh step 2}
\Lh v^2+\frac{\kappa|v|^{2+\gamma}}{32c_2\K^\gamma}\le C\bK^2.
\end{align}
If we take $\psi(t)=\|v(t)\|^2_{L^\infty_{x,h}}$ in \eqref{ineq for v with operator Lh step 2}, then $\psi$ satisfies the following differential inequality
\begin{align}\label{diff ineq for psi}
\frac{d}{dt}\psi+\frac{\kappa\psi^{1+\frac{\gamma}{2}}}{32c_2\K^\gamma}\le C\bK^2,
\end{align}
with $\psi(0)$ satisfying the bound 
\begin{align*}
\psi(0)\le\frac{4\|\theta_0\|^2_{L^\infty}}{\xi(0)^{2\alpha}}\le 4\K^2.
\end{align*}
Hence it follows from \eqref{diff ineq for psi} that 
\begin{align}\label{bound on psi all time}
\psi(t)\le C(\K+\bK),\qquad\forall t\ge0,
\end{align}
for some $C>0$. We conclude from \eqref{bound on psi all time} that 
\begin{align}
[\theta(t)]^2_{C^\alpha}\le C(\K+\bK),\qquad \forall t\ge t_{\alpha},
\end{align}
where $t_{\alpha}$ is given by \eqref{def of t alpha}. Finally, together with the bound \eqref{infty bound on theta using K for kappa>0}, 
\begin{align}
\|\theta(t)\|_{C^\alpha}=\|\theta(t)\|_{L^\infty}+[\theta(t)]^2_{C^\alpha}\le C(\K+\bK),\qquad \forall t\ge t_{\alpha}.
\end{align}

For the case when $\lambda>0$ and $\gamma\le1$, using the damping term $\lambda\mathcal{D}\theta$, we can choose $\alpha\in(0,\frac{1}{4}]$ which is independent of $\gamma$ such that \eqref{bound on C alpha norm of theta with gamma}-\eqref{def of K and bar K} hold by replacing $\gamma$ with $1$. The key for obtaining \eqref{bound C alpha norm of theta for all time} is the following inequality that can be deduced from \eqref{ineq for v with operator Lh}:
\begin{align*}
&\Lh v^2+\lambda\cdot\left(\frac{v^2}{\xi^2+|h|^2}+\frac{D_{1}[\Dh\theta(x,t)]}{(\xi^2+|h|^2)^\alpha}\right)\notag\\
&\le\Lh v^2+\lambda\cdot\left(\frac{v^2}{\xi^2+|h|^2}+\frac{D_{1}[\Dh\theta(x,t)]}{(\xi^2+|h|^2)^\alpha}\right)+\kappa\cdot\frac{D_{\gamma}[\Dh\theta(x,t)]}{(\xi^2+|h|^2)^\alpha}\notag\\
&\le 2\alpha|\dot{\xi}|\frac{\xi}{\xi^2+|h|^2}v^2+2\alpha\frac{|h|}{\xi^2+|h|^2}|\Dh u|v^2+\frac{2\|S\|_{L^\infty}}{(\xi^2+|h|^2)^\frac{\alpha}{2}}v.
\end{align*}
By the same argument given in the proof of Lemma~\ref{Estimates in C alpha-space lemma}, we can obtain bounds \eqref{bound on the term with derivative on xi}, \eqref{bound on the term with nabla u}, \eqref{bound on the term with S} and  \eqref{lower bound for Dh alpha} as well as inequalities \eqref{ineq for v with operator Lh step 2} and \eqref{diff ineq for psi} for $\gamma=1$. Choosing $K_{1,\infty}$ and $\bar{K}_{1,\infty}$ that satisfy \eqref{def of K and bar K gamma=1}, the bounds \eqref{bound on C alpha norm of theta with gamma=1} follows.
\end{proof}

\begin{remark}
Given $\alpha\in(0,\frac{\gamma}{3+\gamma}]$, if we further assume that $\theta_0\in C^{\alpha}$, then following the similar argument given in the proof of Lemma~\ref{Estimates in C alpha-space lemma} as shown above (also see \cite{CTV14} for reference), we have
\begin{align}\label{bound C alpha norm of theta for all time}
\|\theta(t)\|_{C^\alpha}\le [\theta_0]_{C^\alpha}+C(\K+\bK),\qquad\forall t\ge0.
\end{align}
\end{remark}

With the help of Lemma~\ref{Estimates in C alpha-space lemma}, we obtain the following result which can be regarded as an improvement of the regularity of the absorbing set $B_{\infty}$ defined in Lemma~\ref{Existence of an L infty absorbing set lemma}. A proof can be found in \cite{FS21} and omit the details here.

\begin{lemma}[Existence of an $C^\alpha$-absorbing set]\label{Existence of an C alpha absorbing set lemma}
Given $\lambda\ge0$ and $\kappa>0$, there exists $\alpha\in(0,\frac{\gamma}{3+\gamma}]$ and a constant $C_{\alpha}=C_{\alpha}(\|S\|_{L^\infty},\alpha,\kappa,\gamma)\ge1$ such that the set
\begin{align}\label{def of B alpha set}
B_{\alpha}=\left\{\phi\in C^{\alpha}\cap H^1:\|\phi\|_{C^\alpha}\le C_{\alpha}\right\}
\end{align}
is an absorbing set for $\pin(t)$. Moreover, we have
\begin{align}\label{bound on solution map in C alpha}
\sup_{t\ge0}\sup_{\theta_0\in B_{\alpha}}\|\pin(t)\theta_0\|_{C^\alpha}\le 2C_{\alpha}.
\end{align}
Here the constant $C_\alpha$ is given by
{\footnotesize
\begin{align*}
C_\alpha:=\max\left\{1,(1+\kappa^{-\frac{1}{\gamma}})\Big(\frac{3\|S\|_{L^\infty}}{c_0\kappa}\Big)+\|S\|^\frac{2+\gamma}{2(1+\gamma)}_{L^\infty}\kappa^{-\frac{1}{2(1+\gamma)}}\Big(\frac{3\|S\|_{L^\infty}}{c_0\kappa}\Big)^{\frac{\gamma}{2(1+\gamma)}}+\kappa^{-\frac{1}{\gamma}}\Big(\frac{3\|S\|_{L^\infty}}{c_0\kappa}\Big)^{\frac{6+\gamma}{4}}\right\}.
\end{align*}
}
\end{lemma}

\begin{remark}\label{Existence of an C alpha absorbing set remark}
Similar as before, for the case when $\lambda>0$ and $\gamma\le1$, there exists $\alpha\in(0,\frac{1}{4}]$ such that the constant $\bar{C}_\alpha=\bar{C}_{\alpha}(\|S\|_{L^\infty},\alpha,\lambda)\ge1$ can be chosen to be independent of $\gamma$ as 
\begin{align*}
\bar{C}_\alpha:=\max\left\{1,(1+\lambda^{-1})\Big(\frac{3\|S\|_{L^\infty}}{c_0\lambda}\Big)+\|S\|^\frac{3}{4}_{L^\infty}\lambda^{-\frac{1}{4}}\Big(\frac{3\|S\|_{L^\infty}}{c_0\lambda}\Big)^{\frac{1}{4}}+\lambda^{-1}\Big(\frac{3\|S\|_{L^\infty}}{c_0\lambda}\Big)^{\frac{7}{4}}\right\}.
\end{align*}
By replacing $C_\alpha$ with $\bar{C}_\alpha$, the set $B_\alpha$ as given in \eqref{def of B alpha set} is an absorbing set for $\pin(t)$ and the bound \eqref{bound on solution map in C alpha} holds. 
\end{remark}

We now proceed to prove the existence of a bounded absorbing set in $H^1$:

\begin{lemma}[Existence of an $H^1$-absorbing set]\label{Existence of an H1 absorbing set lemma}
For $\lambda\ge0$ and $\kappa>0$, there exists $\alpha\in(0,\frac{\gamma}{3+\gamma}]$ and a constant $R_1=R_1(\|S\|_{L^\infty\cap H^1},\alpha,\kappa,\gamma)\ge1$ such that the set
\begin{align*}
B_1=\{\phi\in C^{\alpha}\cap H^1:\|\phi\|^2_{H^1}+\|\phi\|^2_{C^\alpha}\le R^2_{1}\}
\end{align*}
is an absorbing set for $\pin(t)$. Moreover, we have
\begin{align}\label{bound on time integral on theta in H 1+gamma/2}
\sup_{t\ge0}\sup_{\theta_0\in B_1}\left[\|\pin(t)\theta_0\|^2_{H^1}+\|\pin(t)\theta_0\|^2_{C^\alpha}+\int_t^{t+1}\|\pin(\tau)\theta_0\|^2_{H^{1+\frac{\gamma}{2}}}d\tau\right]\le 2R_1^2.
\end{align}
For $\lambda>0$ and $\gamma\le1$, there exists $\alpha\in(0,\frac{1}{4}]$ such that the bound 
\begin{align}\label{bound on time integral on theta in H 1+1/2 gamma=1}
\sup_{t\ge0}\sup_{\theta_0\in B_1}\left[\|\pin(t)\theta_0\|^2_{H^1}+\|\pin(t)\theta_0\|^2_{C^\alpha}+\int_t^{t+1}\|\pin(\tau)\theta_0\|^2_{H^{\frac{3}{2}}}d\tau\right]\le 2\bar{R}_1^2
\end{align}
holds for some constant $\bar{R}_1=\bar{R}_1(\|S\|_{L^\infty\cap H^1},\alpha,\lambda)\ge1$ independent of $\gamma$, and the set $B_1$ is also an absorbing set for $\pin(t)$ by replacing $R_1$ with $\bar{R}_1$.
\end{lemma}

\begin{proof}
It suffices to establish an {\it a priori} estimate for initial data in $H^1\cap C^\alpha$ for some $\alpha\in(0,\frac{\gamma}{3+\gamma}]$. Suppose that $\theta_0\in H^1\cap C^{\alpha}$. We apply $\nabla$ to \eqref{abstract active scalar eqn}$_1$ and take the inner product with $\nabla\theta$ to obtain
\begin{align}\label{eqn for nabla u with D[nabla theta]}
&(\dt+u\cdot\nabla+\lambda\Lambda+\kappa\Lambda^\gamma)|\nabla\theta|^2+\lambda|\nabla\theta|^2+\lambda D_{1}[\nabla\theta]+\kappa D_{\gamma}[\nabla\theta]\notag\\
&=-2\partial_{x_j} u_l\partial_{x_l}\theta\partial_{x_j}\theta+2\nabla S\cdot\nabla\theta,
\end{align}
where $D_{\gamma}[\cdot]$ is given by \eqref{def of functional D gamma []}. Applying \cite[Theorem~2.2]{CV12}, we obtain
\begin{align}\label{lower bound for D[nabla theta]}
D_{\gamma}[\nabla\theta]\ge\frac{|\nabla\theta|^{2+\frac{\gamma}{1-\alpha}}}{C[\theta]^\frac{\gamma}{1-\alpha}_{C^\alpha}}\ge\frac{|\nabla\theta|^{2+\frac{\gamma}{1-\alpha}}}{C(2C_\alpha)^\frac{\gamma}{1-\alpha}}
\end{align}
and
\begin{align}\label{lower bound for D[nabla theta] gamma=1}
D_{1}[\nabla\theta]\ge\frac{|\nabla\theta|^{2+\frac{1}{1-\alpha}}}{C(2C_\alpha)^\frac{1}{1-\alpha}}.
\end{align}
Also, by applying Young's inequality, we have
\begin{align}\label{bound on terms involving nabla u}
|-2\partial_{x_l} u_j\partial_{x_j}\theta\partial_{x_j}\theta|\le \frac{\kappa}{4}\frac{|\nabla\theta|^{2+\frac{\gamma}{1-\alpha}}}{C(2C_\alpha)^\frac{\gamma}{1-\alpha}}+\Big(\frac{4C}{\kappa}\Big)^\frac{2-2\alpha}{\gamma}(2C_\alpha)^2|\nabla u|^{1+\frac{2(1-\alpha)}{\gamma}},
\end{align}
where $C_\alpha$ is given in Lemma~\ref{Existence of an C alpha absorbing set lemma}. We apply \eqref{lower bound for D[nabla theta]}, \eqref{lower bound for D[nabla theta] gamma=1} and \eqref{bound on terms involving nabla u} on \eqref{eqn for nabla u with D[nabla theta]} to get
\begin{align}\label{ineq for nabla u with D[nabla theta]}
&(\dt+u\cdot\nabla+\lambda\Lambda+\kappa\Lambda^\gamma)|\nabla\theta|^2+\lambda|\nabla\theta|^2+\lambda D_{1}[\nabla\theta]+\frac{\kappa}{2}D_{\gamma}[\nabla\theta]\notag\\
&\le\Big(\frac{4C}{\kappa}\Big)^\frac{2-2\alpha}{\gamma}(2C_\alpha)^2|\nabla u|^{1+\frac{2(1-\alpha)}{\gamma}}+2|\nabla S||\nabla \theta|.
\end{align}
Integrate \eqref{ineq for nabla u with D[nabla theta]} over $\Omega$ and dropping the terms involving $\lambda$,
\begin{align}\label{ineq for nabla u with D[nabla theta] final}
\frac{d}{dt}\|\theta\|^2_{H^1}+\frac{\kappa}{2}\|\theta\|^2_{H^{1+\frac{\gamma}{2}}}
&\le\Big(\frac{4C}{\kappa}\Big)^\frac{2-2\alpha}{\gamma}(2C_\alpha)^2\|\nabla u\|^{1+\frac{2(1-\alpha)}{\gamma}}_{L^\infty}+2\|S\|_{H^1}\|\theta\|_{H^1}\notag\\
&\le\Big(\frac{4C}{\kappa}\Big)^\frac{2-2\alpha}{\gamma}(2C_\alpha)^2\K^{1+\frac{2(1-\alpha)}{\gamma}}+\frac{8}{c_{\gamma,d}\kappa}\|S\|^2_{H^1}+\frac{\kappa}{4}\|\theta\|^2_{H^{1+\frac{\gamma}{2}}},
\end{align}
where the last inequality follows by the bound \eqref{L infty bound on nabla u general} and Young's inequality, and $c_{\gamma,d}>0$ is the dimensional constant for which it satisfies the bound
\begin{align}\label{def of c gamma d}
\|\theta\|^2_{H^1}\le(c_{d,\gamma})^{-1}\|\theta\|^2_{H^{1+\frac{\gamma}{2}}}.
\end{align}
If we choose $K_1\ge1$ such that
\begin{align*}
K_1=\frac{8}{c_{\gamma,d}\kappa}\Big[\Big(\frac{4C}{\kappa}\Big)^\frac{2-2\alpha}{\gamma}(2C_\alpha)^2\K^{1+\frac{2(1-\alpha)}{\gamma}}+\frac{8}{c_{\gamma,d}\kappa}\|S\|^2_{H^1}\Big],
\end{align*}
then by applying Gr\"{o}nwall's inequality on \eqref{ineq for nabla u with D[nabla theta] final}, we conclude that
\begin{align}\label{bound on H^1 in terms of initial H^1 data}
\|\theta(t)\|^2_{H^1}\le\|\theta_0\|^2_{H^1}e^{-\frac{c_{\gamma,d}}{8}\kappa t}+K_1.
\end{align}
Now we define $R_1=2K_1$, then it is straight forward to see that the set $B_1$ is an absorbing set for $\pin(t)$. Moreover, upon integrating \eqref{ineq for nabla u with D[nabla theta] final} on the time interval $(t,t+1)$ and applying \eqref{bound on H^1 in terms of initial H^1 data}, we further obtain \eqref{bound on time integral on theta in H 1+gamma/2}.

For the case when $\lambda>0$ and $\gamma\le1$, using Young's inequality, we can rewrite the inequality given in \eqref{bound on terms involving nabla u} as
\begin{align*}
|-2\partial_{x_l} u_j\partial_{x_j}\theta\partial_{x_j}\theta|\le \frac{\lambda}{4}\frac{|\nabla\theta|^{2+\frac{1}{1-\alpha}}}{C(2\bar{C}_\alpha)^\frac{1}{1-\alpha}}+\Big(\frac{4C}{\lambda}\Big)^{2-2\alpha}(2\bar{C}_\alpha)^2|\nabla u|^{1+2(1-\alpha)},
\end{align*}
for some $\alpha\in(0,\frac{1}{4}]$ and $\bar{C}_\alpha$ is chosen in Remark~\ref{Existence of an C alpha absorbing set remark} that is independent of $\gamma$. The inequality \eqref{ineq for nabla u with D[nabla theta]} can then be rewritten as
\begin{align*}
&(\dt+u\cdot\nabla+\lambda\Lambda+\kappa\Lambda^\gamma)|\nabla\theta|^2+\lambda|\nabla\theta|^2+\frac{\lambda}{2} D_{1}[\nabla\theta]+\kappa D_{\gamma}[\nabla\theta]\notag\\
&\le\Big(\frac{4C}{\lambda}\Big)^{2-2\alpha}(2\bar{C}_\alpha)^2|\nabla u|^{1+2(1-\alpha)}+2|\nabla S||\nabla \theta|,
\end{align*}
which further implies
\begin{align*}
\frac{d}{dt}\|\theta\|^2_{H^1}+\frac{\lambda}{2}\|\theta\|^2_{H^{\frac{3}{2}}}
\le\Big(\frac{4C}{\lambda}\Big)^{2-2\alpha}(2\bar{C}_\alpha)^2K_{1,\infty}^{1+2(1-\alpha)}+\frac{8}{c_{\gamma,d}\lambda}\|S\|^2_{H^1}+\frac{\lambda}{4}\|\theta\|^2_{H^\frac{3}{2}}.
\end{align*}
Following the previous steps, the bound \eqref{bound on time integral on theta in H 1+1/2 gamma=1} holds for some $\bar{R}_1\ge1$ independent of $\gamma$. The proof of Lemma~\ref{Existence of an H1 absorbing set lemma} is complete.
\end{proof}

With the help of Lemma~\ref{Existence of an H1 absorbing set lemma}, we can improve the regularity of the absorbing set $B_1$, which is illustrated in the next lemma. 

\begin{lemma}[Existence of an $H^{1+\frac{\gamma}{2}}$-absorbing set]\label{Existence of an H 1+gamma/2 absorbing set lemma}
For $\lambda\ge0$ and $\kappa>0$, there exists a constant $R_{1+\frac{\gamma}{2}}\ge1$ which depends on $\|S\|_{L\infty\cap H^1}$, $\kappa$, $\gamma$ such that the set
\begin{align*}
B_{1+\frac{\gamma}{2}}=\left\{\phi\in H^{1+\frac{\gamma}{2}}:\|\phi\|_{H^{1+\frac{\gamma}{2}}}\le R_{1+\frac{\gamma}{2}}\right\}
\end{align*}
is an absorbing set for $\pin(t)$. Moreover
\begin{align}\label{bound on theta in H 1+gamma/2}
\sup_{t\ge0}\sup_{\theta_0\in B_{1+\frac{\gamma}{2}}}\|\pin(t)\theta_0\|_{H^{1+\frac{\gamma}{2}}}\le 2R_{1+\frac{\gamma}{2}}.
\end{align}
For $\lambda>0$ and $\gamma\le1$, there exists a constant $\bar{R}_{\frac{3}{2}}\ge1$ which depends on $\|S\|_{L\infty\cap H^1}$, $\lambda$ but is independent of $\gamma$ such that the set
\begin{align*}
\bar{B}_{\frac{3}{2}}=\left\{\phi\in H^{1+\frac{\gamma}{2}}:\|\phi\|_{H^{1+\frac{\gamma}{2}}}\le \bar{R}_{\frac{3}{2}}\right\}
\end{align*}
is an absorbing set for $\pin(t)$. Moreover, the bound 
\begin{align}\label{bound on theta in H 1+gamma/2 gamma=1}
\sup_{t\ge0}\sup_{\theta_0\in \bar{B}_{\frac{3}{2}}}\|\pin(t)\theta_0\|_{H^{\frac{3}{2}}}\le 2\bar{R}_{\frac{3}{2}}
\end{align}
holds.
\end{lemma}

\begin{proof}
It is sufficient to show that $B_{1+\frac{\gamma}{2}}$ absorbs the $H^1$-absorbing set $B_1$ obtained in Lemma~\ref{Existence of an H1 absorbing set lemma}. Suppose $\theta_0\in B_{1}$, then from \eqref{bound on time integral on theta in H 1+gamma/2}, we have
\begin{align}\label{bound on time integral for theta in H 1+gamma/2 from t to t+1}
\sup_{t\ge0}\int_{t}^{t+1}\|\pin(\tau)\theta\|^2_{H^{1+\frac{\gamma}{2}}}d\tau\le 2R_1^2.
\end{align}
By the same argument given in the proof of Theorem~\ref{global-in-time wellposedness in Sobolev}, for $t\ge0$, 
\begin{align}\label{differential inequality on theta 1+gamma}
&\frac{d}{dt}\|\Lambda^{1+\frac{\gamma}{2}}\thetakappa\|^2_{L^2}+\kappa\|\thetakappa\|^2_{H^{1+\gamma}}\notag\\
&\le C\K\|\Lambda^{1+\frac{\gamma}{2}}\thetakappa\|^2_{L^2}+\frac{4}{c_{\gamma,d}\kappa}\|S\|^2_{H^{1+\frac{\gamma}{2}}}+\frac{\kappa}{2}\|\theta\|^2_{H^{1+\gamma}},
\end{align}
where we have used the bound \eqref{infty bound on theta using K for kappa>0} on $\|\thetakappa\|_{L^\infty}$ and $c_{\gamma,d}$ is defined in \eqref{def of c gamma d}. By applying the local integrability \eqref{bound on time integral for theta in H 1+gamma/2 from t to t+1} and the uniform Gr\"{o}nwall's lemma,
\begin{align}\label{bound on theta in H 1+gamma/2 for t>1}
\|\pin(t)\theta_0\|^2_{H^{1+\frac{\gamma}{2}}}\le\Big[2C\K R_1^2+\frac{4}{c_{\gamma,d}\kappa}\|S\|^2_{H^{1+\frac{\gamma}{2}}}\Big]e^{C\K R_1^2},\qquad \forall t\ge1.
\end{align}
Choosing 
\begin{align*}
R_{1+\frac{\gamma}{2}}:=\Big[2C\K R_1^2+\frac{4}{c_{\gamma,d}\kappa}\|S\|^2_{H^{1+\frac{\gamma}{2}}}\Big]^\frac{1}{2}e^{\frac{C\K R_1^2}{2}},
\end{align*}
we conclude that $\pin(t) B_1\subset B_{1+\frac{\gamma}{2}}$ holds for all $t\ge1$. Since $B_{1+\frac{\gamma}{2}}$ absorbs $B_1$, the bound \eqref{bound on theta in H 1+gamma/2} then follows  by \eqref{bound on time integral on theta in H 1+gamma/2} and \eqref{bound on theta in H 1+gamma/2 for t>1}. For the case when $\lambda>0$ and $\gamma\le1$, we apply \eqref{bound on time integral on theta in H 1+1/2 gamma=1} to show that \eqref{bound on time integral for theta in H 1+gamma/2 from t to t+1} holds by replacing $\gamma$ and $R_1$ by $1$ and $\bar{R}_1$ respectively, and choosing
\begin{align*}
\bar{R}_{\frac{3}{2}}:=\Big[2CK_{1\infty} \bar{R}_1^2+\frac{4}{c_{\gamma,d}\lambda}\|S\|^2_{H^{\frac{3}{2}}}\Big]^\frac{1}{2}e^{\frac{CK_{1,\infty} \bar{R}_1^2}{2}}.
\end{align*}
The remaining steps are just similar and we omit the details here.
\end{proof}

\begin{remark}
By combining \eqref{differential inequality on theta 1+gamma} with \eqref{bound on theta in H 1+gamma/2 for t>1}, if $\theta_0\in B_1$, then for $t\ge1$ and $T>0$, we further obtain
\begin{align}\label{bound on time integral on theta in H 1+gamma}
\frac{1}{T}\int_{t}^{t+T}\|\pin(\tau)\theta_0\|^2_{H^{1+\gamma}}d\tau\le R_{1+\gamma}^2,
\end{align}
where
\begin{align*}
R_{1+\gamma}:=\Big[C\K R_{1+\frac{\gamma}{2}}^2+\frac{4}{c_{\gamma,d}\kappa}\|S\|^2_{H^{1+\frac{\gamma}{2}}}\Big]^\frac{1}{2},
\end{align*}
while for the case when $\lambda>0$ and $\gamma\le1$, we choose
\begin{align*}
\bar{R}_{2}:=\Big[CK_{1,\infty} \bar{R}_{\frac{3}{2}}^2+\frac{4}{c_{\gamma,d}\lambda}\|S\|^2_{H^{\frac{3}{2}}}\Big]^\frac{1}{2}
\end{align*}
such that the following bound holds for $t\ge1$ and $T>0$ that
\begin{align}\label{bound on time integral on theta in H 1+gamma gamma=1}
\frac{1}{T}\int_{t}^{t+T}\|\pin(\tau)\theta_0\|^2_{H^{2}}d\tau\le \bar{R}_{2}^2.
\end{align}
\end{remark}

The existence and regularity of the global attractor claimed by Theorem~\ref{Existence of H1 global attractor} now follows from Lemma~\ref{Existence of an H 1+gamma/2 absorbing set lemma} by applying the argument given in \cite[Proposition~8]{CCP12}, and the bounds \eqref{uniform bound on theta from the attractor}-\eqref{uniform bound on time integral of theta from the attractor} are immediate consequences of \eqref{bound on theta in H 1+gamma/2} and \eqref{bound on time integral on theta in H 1+gamma} by taking 
\begin{align*}
\tilde{M}_{\Gg}:=\max\{M_{\Gg},\bar{M}_{\Gg}\},
\end{align*}
where $M_{\Gg}$ and $\tilde M_{\Gg}$ are defined by
\begin{align}\label{def of constant M on global attractor}
M_{\Gg}:=\max\{2R_{1+\frac{\gamma}{2}},R_{1+\gamma}\},\qquad\bar{M}_{\Gg}:=\max\{2\bar{R}_{\frac{3}{2}},\bar{R}_{2}\}.
\end{align}
We summarise the properties of the global attractor $\Gg$ as claimed by Theorem~\ref{Existence of H1 global attractor}:
\begin{corollary}\label{Unique attractor general corollary}
The solution map $\pin(t):H^1\to H^1$ associated to \eqref{abstract active scalar eqn} possesses a unique global attractor $\Gg$ with the following properties:
\begin{itemize}
\item $\Gg\subset H^{1+\frac{\gamma}{2}}$ and is the $\omega$-limit set of $B_{1+\frac{\gamma}{2}}$, namely
\begin{align*}
\Gg=\omega(B_{1+\frac{\gamma}{2}})=\bigcap_{t\ge0}\overline{\bigcup_{\tau\ge t}\pin(\tau) B_{1+\frac{\gamma}{2}}}.
\end{align*}
\item For every bounded set $B\subset H^1$, we have
\begin{align*}
\lim_{t\to\infty}\dist(\pin(t)B,\Gg)=0,
\end{align*}
where $\dist$ stands for the usual Hausdorff semi-distance between sets given by the $H^1$-norm.
\item $\Gg$ is minimal in the class of $H^1$-closed attracting set.
\end{itemize}
\end{corollary}
\begin{remark}
When $\lambda>0$ and $\gamma\le1$, the global attractor $\Gg$ can be chosen independent of $\gamma$, in the sense that $\Gg\subset H^{\frac{3}{2}}$ and is the $\omega$-limit set of $\bar{B}_{\frac{3}{2}}$ such that 
\begin{align*}
\Gg=\omega(\bar{B}_{\frac{3}{2}})=\bigcap_{t\ge0}\overline{\bigcup_{\tau\ge t}\pin(\tau) \bar{B}_{\frac{3}{2}}},
\end{align*}
where $\bar{B}_{\frac{3}{2}}$ is given in Lemma~\ref{Existence of an H 1+gamma/2 absorbing set lemma}.
\end{remark}
\begin{remark}[global attractors in $H^2$-space when $\gamma=2$]\label{H2 global attractor rem}
For the case $\gamma=2$, the regularity of the global attractor given in Theorem~\ref{Existence of H1 global attractor} can be further improved to $H^2$. More precisely, when $\gamma=2$, we restrict $\tilde{\pi}^\lambda(t)=\pin(t)\Big|_{H^2}$ and consider solution map $\tilde{\pi}^\lambda(t):H^2\to H^2$ associated to \eqref{abstract active scalar eqn}. Following the same argument for proving Theorem~\ref{Existence of H1 global attractor}, it shows that $\tilde{\pi}^\lambda$ possesses a unique global attractor $\tilde{\mathcal{G}}^\lambda$ with the following refined properties:
\begin{itemize}
\item $\tilde{\mathcal{G}}^\lambda\subset H^2$ and is the $\omega$-limit set of $B_{2}$, namely
\begin{align*}
\tilde{\mathcal{G}}^\lambda=\omega(B_{2})=\bigcap_{t\ge0}\overline{\bigcup_{\tau\ge t}\tilde{\pi}^\lambda(\tau) B_{2}}.
\end{align*}
\item For every bounded set $B\subset H^2$, we have
\begin{align*}
\lim_{t\to\infty}\dist(\pin(t)B,\tilde{\mathcal{G}}^\lambda)=0,
\end{align*}
where $\dist$ stands for the usual Hausdorff semi-distance between sets given by the $H^2$-norm.
\item $\tilde{\mathcal{G}}^\lambda$ is minimal in the class of $H^2$-closed attracting set.
\end{itemize}
\end{remark}

\subsection{Further properties for the global attractors}\label{Further properties for the attractors subsection}

In this subsection, we prove some additional properties for the global attractors obtained in Theorem~\ref{Existence of H1 global attractor}. Recall that we have $d\in\{2,3\}$, and throughout this subsection, assume that $\lambda$ and $\gamma$ satisfy either one of the following conditions:
\begin{align}\label{conditions on gamma}
\mbox{$\lambda\ge0$ and $\gamma\in[1,2]$};
\end{align}
or
\begin{align}\label{conditions on gamma lambda>0}
\mbox{$\lambda>0$ and $\gamma\in(0,2]$}.
\end{align}
Our goal is to prove the following theorem: 
\begin{theorem}\label{Existence of H1 global attractor further properties thm}
Let $S\in L^\infty\cap H^1$. For $\kappa>0$, assume that $\lambda$ and $\gamma$ satisfy either \eqref{conditions on gamma} or \eqref{conditions on gamma lambda>0}. Then the global attractor $\Gg$ of $\pin(t)$ further enjoys the following properties:
\begin{itemize}
\item $\Gg$ is fully invariant, namely
\begin{align*}
\pin(t)\Gg=\Gg,\qquad \forall t\ge0.
\end{align*}
\item $\Gg$ is maximal in the class of $H^1$-bounded invariant sets.
\item $\Gg$ has finite fractal dimension.
\end{itemize}
\end{theorem}

\begin{remark}
For the case when $\lambda=0$ and $\gamma\in(0,1)$, it is unknown whether the global attractor $\Gg$ satisfies those properties given in Theorem~\ref{Existence of H1 global attractor further properties thm} or not. 
\end{remark}

We recall some auxiliary estimates on $u$ and $\theta$ under the assumption \eqref{conditions on gamma}, which are summarised in Proposition~\ref{auxiliary estimates on u with gamma>1}. The proof can be found in \cite{FS21}.

\begin{proposition}\label{auxiliary estimates on u with gamma>1}
Assume that $\gamma\in[1,2]$, then if $f\in H^{1+\frac{\gamma}{2}}$ and $g$, $\theta\in H^1$, we have 
\begin{align}\label{bound on Lambda theta for gamma>1}
\|\Lambda^{2-\frac{\gamma}{2}}f\|_{L^2}\le C\|\Lambda^{1+\frac{\gamma}{2}}f\|_{L^2},
\end{align}
\begin{align}\label{bound on L4 using H1}
\|g\|_{L^4}\le C\|\Lambda g\|_{L^2},
\end{align}
\begin{align}\label{bound on u in terms of nabla theta}
\|u\|_{L^\infty}\le C\|\Lambda^\frac{\gamma}{2}\theta\|_{L^2}\le C\|\Lambda\theta\|_{L^2},
\end{align}
\begin{align}\label{bound on nabla u in terms of nabla theta}
\|\nabla u\|_{L^\infty}\le C\|\Lambda^{1+\frac{\gamma}{2}}\theta\|_{L^2},
\end{align}
where $u=u[\theta]$, $C$ is a positive constant which depends on $d$ only.
\end{proposition}

We now show that the solution map $\pin(t)$ is indeed continuous in the $H^1$-topology. More precisely, we have the following lemma:

\begin{lemma}[Continuity of $\pin(t)$]\label{continuity of solution map lemma}
Let $\kappa>0$ be given and fixed.
\begin{itemize}
\item Assume that $\lambda$ and $\gamma$ satisfy \eqref{conditions on gamma}. For every $t>0$, the solution map $\pin(t):B_{1+\frac{\gamma}{2}}\to H^1$ is continuous in the topology of $H^1$.
\item If $\lambda$ and $\gamma$ satisfy \eqref{conditions on gamma lambda>0}, then for every $t>0$, the solution map $\pin(t):\bar{B}_{\frac{3}{2}}\to H^1$ is continuous in the topology of $H^1$. 
\end{itemize} 
Here $B_{1+\frac{\gamma}{2}}$ and $\bar{B}_{\frac{3}{2}}$ are the absorbing set for $\pin$ defined in Lemma~\ref{Existence of an H 1+gamma/2 absorbing set lemma}.
\end{lemma}

\begin{proof}
Assume that $\lambda$ and $\gamma$ satisfy \eqref{conditions on gamma}. We fix $t>0$ and let $\theta_0$, $\tilde\theta_0\in B_{1+\frac{\gamma}{2}}$ be arbitrary such that $\theta=\pin(t)\theta_0$ and $\tilde\theta=\pin(t)\tilde\theta_0$ with
\begin{align}\label{bound on theta and tilde theta}
\|\theta(t)\|_{H^{1+\frac{\gamma}{2}}}\le 2R_{1+\frac{\gamma}{2}}, \qquad\|\tilde\theta(t)\|_{H^{1+\frac{\gamma}{2}}}\le 2R_{1+\frac{\gamma}{2}},\qquad\forall t\ge0.
\end{align}
Denote the difference by $\bar{\theta}=\theta-\tilde\theta$ with $\bar{\theta}_0=\theta_0-\tilde\theta_0$, then $\bar{\theta}$ satisfies the following equation
\begin{align}\label{differential eqn for bar theta}
\dt\bar{\theta}+\lambda\mathcal{D}\bar{\theta}+\kappa\Lambda^\gamma\bar{\theta}-\bar{u}\cdot\nabla\bar{\theta}+\tilde u\cdot\nabla\bar{\theta}+\bar{u}\cdot\nabla\tilde\theta=0,
\end{align}
where $\tilde u=u[\tilde\theta]$ and $\bar{u}=u[\bar{\theta}]$. Multiply \eqref{differential eqn for bar theta} by $-\Delta\bar{\theta}$ and integrate, we obtain
\begin{align}\label{differential ineq for bar theta step 1}
&\frac{1}{2}\frac{d}{dt}\|\bar\theta\|^2_{H^1}+\lambda\|\bar\theta\|^2_{H^{\frac{3}{2}}}+\kappa\|\bar\theta\|^2_{H^{1+\frac{\gamma}{2}}}\notag\\
&\le\Big|\intox(\nabla\bar{u}\cdot\nabla\bar{\theta})\cdot\nabla\bar{\theta}\Big|+\Big|\intox(\nabla\tilde u\cdot\nabla\bar{\theta})\cdot\nabla\bar{\theta}\Big|+\Big|\intox(\bar{u}\cdot\nabla\tilde\theta)(-\Delta\bar{\theta})\Big|\notag\\
&\le(\|\nabla\bar{u}\|_{L^\infty}+\|\nabla\tilde{u}\|_{L^\infty})\|\nabla\bar{\theta}\|^2_{L^2}+\Big|\intox(\bar{u}\cdot\nabla\tilde\theta)(-\Delta\bar{\theta})\Big|.
\end{align}
Using \eqref{bound on nabla u in terms of nabla theta} and \eqref{bound on theta and tilde theta}, the terms $\|\nabla\bar{u}\|_{L^\infty}$ and $\|\nabla\tilde{u}\|_{L^\infty}$ can be bounded by
\begin{align*}
(\|\nabla\bar{u}\|_{L^\infty}+\|\nabla\tilde{u}\|_{L^\infty})\le C\Big(\|\Lambda^{1+\frac{\gamma}{2}}\theta\|_{L^2}+\|\Lambda^{1+\frac{\gamma}{2}}\tilde\theta\|_{L^2}\Big)\le CR_{1+\frac{\gamma}{2}},
\end{align*}
For the last term on the right side of \eqref{differential ineq for bar theta step 1}, we integrate by parts and use the product estimate \eqref{product estimate} to obtain
\begin{align}\label{bound on the cross term bar theta}
\Big|\intox(\bar{u}\cdot\nabla\tilde\theta)(-\Delta\bar{\theta})\Big|\le C\Big(\|\bar{u}\|_{L^\infty}\|\Lambda^{2-\frac{\gamma}{2}}\tilde\theta\|_{L^2}+\|\tilde\theta\|_{L^4}\|\Lambda^{2-\frac{\gamma}{2}}\bar{u}\|_{L^4}\Big)\|\Lambda^{1+\frac{\gamma}{2}}\bar{\theta}\|_{L^2}.
\end{align}
Using \eqref{two order smoothing for u when nu>0}, \eqref{bound on L4 using H1} and \eqref{bound on theta and tilde theta},
\begin{align*}
\|\tilde\theta\|_{L^4}\|\Lambda^{2-\frac{\gamma}{2}}\bar{u}\|_{L^4}\|\Lambda^{1+\frac{\gamma}{2}}\bar{\theta}\|_{L^2}\le CR_{1+\frac{\gamma}{2}}\|\Lambda\bar\theta\|_{L^2}\|\Lambda^{1+\frac{\gamma}{2}}\bar{\theta}\|_{L^2}.
\end{align*}
On the other hand, using \eqref{bound on Lambda theta for gamma>1}, \eqref{bound on u in terms of nabla theta} and \eqref{bound on theta and tilde theta}, we have the bounds
\begin{align*}
\|\Lambda^{2-\frac{\gamma}{2}}\tilde\theta\|_{L^2}\le C\|\Lambda^{1+\frac{\gamma}{2}}\tilde\theta\|_{L^2}\le CR_{1+\frac{\gamma}{2}},
\end{align*}
and
\begin{align*}
\|\bar u\|_{L^\infty}\le C\|\Lambda\bar\theta\|_{L^2},
\end{align*}
Hence we obtain from \eqref{bound on the cross term bar theta} that
\begin{align}\label{bound on the cross term bar theta final}
\Big|\intox(\bar{u}\cdot\nabla\tilde\theta)(-\Delta\bar{\theta})\Big|&\le CR_{1+\frac{\gamma}{2}}\|\Lambda\bar\theta\|_{L^2}\|\Lambda^{1+\frac{\gamma}{2}}\bar{\theta}\|_{L^2}
\end{align}
Apply \eqref{bound on the cross term bar theta final} on \eqref{differential ineq for bar theta step 1} and recall that $\lambda\ge0$, we obtain
\begin{align*}
\frac{d}{dt}\|\bar\theta\|^2_{H^1}+\kappa\|\bar\theta\|^2_{H^{1+\frac{\gamma}{2}}}\le C\Big[R_{1+\frac{\gamma}{2}}+\frac{2}{\kappa}R_{1+\frac{\gamma}{2}}^2\Big]\|\nabla\bar{\theta}\|^2_{L^2},
\end{align*}
which further implies
\begin{align*}
\|\bar\theta(t)\|^2_{H^1}\le K_t\|\bar\theta_0\|^2_{H^1},
\end{align*}
where $K_t=\exp\Big(Ct\Big(R_{1+\frac{\gamma}{2}}+\frac{2}{\kappa}R_{1+\frac{\gamma}{2}}^2\Big)\Big)$. This shows that $\pin(t)$ is continuous in the $H^1$-topology.

If $\lambda$ and $\gamma$ satisfy \eqref{conditions on gamma lambda>0}, then the above estimates hold by replacing $\gamma$ and $R_{1+\frac{\gamma}{2}}$ by $1$ and $\bar{R}_{\frac{3}{2}}$ respectively, which gives
\begin{align*}
\frac{d}{dt}\|\bar\theta\|^2_{H^1}+\lambda\|\bar\theta\|^2_{H^{\frac{3}{2}}}\le C\Big[\bar{R}_{\frac{3}{2}}+\frac{2}{\lambda}\bar{R}_{\frac{3}{2}}^2\Big]\|\nabla\bar{\theta}\|^2_{L^2}.
\end{align*}
Following the similar argument, it implies that $\pin(t)$ is continuous in the $H^1$-topology.
\end{proof}

Following the same argument given in \cite{FS21} and using the log-convexity method introduced by \cite{AN67}, we also prove that the solution map $\pin$ is injective on the absorbing sets $B_{1+\frac{\gamma}{2}}$ or $\bar{B}_{\frac{3}{2}}$. The results are summarised in the following proposition.

\begin{proposition}[Backwards uniqueness]\label{Backwards uniqueness lemma}
Fix $\kappa>0$. Let $\tsup[1]{\varphi}_0$, $\tsup[2]{\varphi}_0\in H^1$ be two initial data, and let
\begin{align*}
\tsup[1]{\varphi},\tsup[2]{\varphi}\in
\left\{ \begin{array}{l}
\mbox{$C([0, \infty); H^1)\cap L^2([0, \infty); H^{1+\frac{\gamma}{2}})$ if \eqref{conditions on gamma} is in force;} \\
\mbox{$C([0, \infty); H^1)\cap L^2([0, \infty); H^{\frac{3}{2}})$ if \eqref{conditions on gamma lambda>0} is in force}
\end{array}\right.
\end{align*}
be the corresponding solutions of the initial value problem \eqref{abstract active scalar eqn} for $\tsup[1]{\varphi}_0$ and $\tsup[2]{\varphi}_0$ respectively. If there exists $T > 0$ such that $\tsup[1]{\varphi}(\cdot,T) = \tsup[2]{\varphi}(\cdot,T)$, then $\tsup[1]{\varphi}_0= \tsup[2]{\varphi}_0$ holds.
\end{proposition}

\begin{proof}
It can be proved by the same method as given in \cite{FS21} and we omit the details here.
\end{proof}

\begin{remark}
In view of the results obtained from Proposition~\ref{Backwards uniqueness lemma}, the solution map $\pin(t)$ is injective on $\Gg$. When such solution map  is restricted to $\Gg$, the dynamics actually defines a dynamical system. Hence $\pin(t)\Big|_{\Gg}$ makes sense for ${\it all}$ $t\in\R$, not just for $t\ge0$.
\end{remark}

Using Lemma~\ref{continuity of solution map lemma} and Proposition~\ref{Backwards uniqueness lemma}, we obtain the invariance and the maximality of the attractor $\Gg$ stated in Theorem~\ref{Existence of H1 global attractor further properties thm}. The results are summarised in the following corollary.

\begin{corollary}
The global attractor $\Gg$ of $\pin(t)$ is fully invariant, namely
\begin{align*}
\pin(t)\Gg=\Gg,\qquad \forall t\ge0.
\end{align*}
In particular, $\Gg$ is maximal among the class of bounded invariant sets in $H^1$.
\end{corollary}

Finally, we address the fractal dimensions for the global attractors $\Gg$ obtained in Theorem~\ref{Existence of H1 global attractor further properties thm}. Given a compact set $X$, we give the following definition for fractal dimension $\dimf(X)$, which is based on counting the number of closed balls of a fixed radius $\varepsilon$ needed to cover $X$; see \cite{R13} for further explanation. 

\begin{definition}\label{fractal dimension def}
Given a compact set $X$, let $N(X,\varepsilon)$ be the minimum number of balls of radius $\varepsilon$ that cover $X$. The {\it fractal dimension} $\dimf (X)$ of $X$ is given by
\begin{align*}
\dimf(X) := \limsup_{\varepsilon\to0}\frac{\log N(X, \varepsilon)}{-\log\varepsilon}.
\end{align*}
\end{definition}

In order to prove that $\dimf (\Gg)$ is finite, we need to show that the solution map $\pin$ given in \eqref{def of solution map nu>0 and kappa>0} is uniform differentiable. More precisely, we have the following definition:

\begin{definition}\label{uniform differentiable def}
We say that $\pin(t)$ is {\it uniform differentiable} on $\Gg$ if for every $\theta_0\in\Gg$, there exists a linear operator $D\pin(t, \theta_0)$ such that
\begin{align}\label{condition 1 for uniform differentiable}
\mbox{$\dis\sup_{\theta_0,\varphi_0\in\Gg;0<\|\theta_0-\varphi_0\|_{H^1}\le\varepsilon}\frac{\|\pin(t)(\varphi_0)-\pin(t)(\theta_0)-D\pin(t)(\varphi_0-\theta_0)\|_{H^1}}{\|\theta_0-\varphi_0\|_{H^1}}\to0$ as $\varepsilon\to0$,}
\end{align}
and 
\begin{align}\label{condition 2 for uniform differentiable}
\mbox{$\dis\sup_{\theta_0\in\Gg}\sup_{\psi_0\in H^1}\frac{\|D\pin(t,\theta_0)(\psi_0)\|_{H^1}}{\|\psi_0\|_{H^1}}<\infty$,\qquad for all $t\ge0$.}
\end{align}
\end{definition}
Next, Lemma~\ref{uniform differentiable lemma} shows that $\pin(t)$ is indeed uniform differentiable with the associated linear operator $D\pin(t, \theta_0)$ being given by 
\begin{align*}
D\pin(t, \theta_0)[\psi_0]:=\psi(t),
\end{align*}
where $\psi$ is the solution of the following linearised problem
\begin{align}\label{linearised active scalar with gamma}
\left\{ \begin{array}{l}
\frac{d\psi}{dt}=A_{\theta}[\psi], \\
\psi(x,0)=\psi_0(x).
\end{array}\right.
\end{align}
Here $\theta=\pin(t)\theta_0$ and $A_{\theta}$ is the elliptic operator given by
\begin{equation}\label{def of operator A theta}
A_{\theta}[\psi]=A_{\theta_0}(t)[\psi]:=-\kappa\Lambda^\gamma\psi-u[\theta]\cdot\nabla\psi-u[\psi]\cdot\nabla\theta.
\end{equation}
\begin{lemma}\label{uniform differentiable lemma}
Let $\kappa>0$ be fixed. If $\lambda$ and $\gamma$ satisfy either \eqref{conditions on gamma} or \eqref{conditions on gamma lambda>0}, then the corresponding solution map $\pin(t)$ is uniform differentiable on $\Gg$. Furthermore, the linear operator $D\pin(t, \theta_0)$ is compact.
\end{lemma}

\begin{proof}
For $\theta_0$, $\varphi_0\in\Gg$, we let $\theta(t)=\pin(t)\theta_0$ and $\varphi(t)=\pin(t)\varphi_0$. We denote $\psi_0=\varphi_0-\theta_0$ and define $\psi(t)=D\pin(t, \theta_0)[\psi_0]$, where $\psi(t)$ satisfies \eqref{linearised active scalar with gamma}. We also define 
\begin{align*}
\eta(t)=\varphi(t)-\theta(t)-\psi(t)=\pin(t)\varphi_0-\pin(t)\theta_0-D\pin(t, \theta_0)[\psi_0].
\end{align*}
Then $\eta(t)$ satisfies
\begin{align}\label{eqn for eta}
\left\{ \begin{array}{l}
\dt\eta+\lambda\mathcal{D}\eta+\kappa\Lambda^\gamma\eta+u[\eta]\cdot\nabla\theta+u[\theta]\cdot\nabla\eta=-u[w]\cdot w, \\
\eta(x,0)=0,
\end{array}\right.
\end{align}
where $w(t)=\varphi(t)-\theta(t)=\pin(t)\psi_0$. We take $L^2$-inner product of \eqref{eqn for eta}$_1$ with $-\Delta\eta$ and obtain
\begin{align}\label{differential eqn for eta}
&\frac{1}{2}\frac{d}{dt}\|\eta\|^2_{H^1}+\lambda\|\eta\|^2_{H^{\frac{3}{2}}}+\kappa\|\eta\|^2_{H^{1+\frac{\gamma}{2}}}\notag\\
&=\intox u[\eta]\cdot\nabla\theta\Delta\eta-\intox\partial_{x_k}u[\theta]\cdot\nabla\eta\partial_{x_k}\eta+\intox u[w]\cdot\nabla w\Delta\eta.
\end{align}
Assume that $\lambda$ and $\gamma$ satisfy \eqref{conditions on gamma}. Using \eqref{bound on nabla u in terms of nabla theta} on $\nabla u$, the second integrand of \eqref{differential eqn for eta} can be bounded by 
\begin{align*}
\Big|\intox\partial_{x_k}u[\theta]\cdot\nabla\eta\partial_{x_k}\eta\Big|\le C\|\nabla u\|_{L^\infty}\|\nabla\eta\|^2_{L^2}\le C\|\Lambda^{1+\frac{\gamma}{2}}\theta\|_{L^2}\|\nabla\eta\|^2_{L^2}.
\end{align*}
To estimate the first integrand appeared in \eqref{differential eqn for eta}, we readily have
\begin{align*}
\Big|\intox u[\eta]\cdot\nabla\theta\Delta\eta\Big|=\Big|\intox\Lambda^{2-\frac{\gamma}{2}}(u[\eta]\theta)\cdot\Lambda^\frac{\gamma}{2}\nabla\eta\Big|\le \|\Lambda^{2-\frac{\gamma}{2}}(u[\eta]\theta)\|_{L^2}\|\Lambda^{1+\frac{\gamma}{2}}\eta\|_{L^2}.
\end{align*} 
Using \eqref{bound on Lambda theta for gamma>1}, we have
\begin{align*}
\|\Lambda^{2-\frac{\gamma}{2}}(u[\eta]\theta)\|_{L^2}\le C\|\Lambda^{1+\frac{\gamma}{2}}(u[\eta]\theta)\|_{L^2},
\end{align*}
and using the product estimate \eqref{product estimate}, 
\begin{align*}
\|\Lambda^{1+\frac{\gamma}{2}}(u[\eta]\theta)\|_{L^2}\le C\Big(\|\Lambda^{1+\frac{\gamma}{2}}u[\eta]\|_{L^4}\|\theta\|_{L^4}+\|u[\eta]\|_{L^\infty}\|\Lambda^{1+\frac{\gamma}{2}}\theta\|_{L^2}\Big).
\end{align*}
Using \eqref{two order smoothing for u when nu>0}, \eqref{bound on L4 using H1} and \eqref{bound on u in terms of nabla theta}, the terms $\|\Lambda^{1+\frac{\gamma}{2}}u[\eta]\|_{L^4}\|\theta\|_{L^4}$ and $\|u[\eta]\|_{L^\infty}\|\Lambda^{1+\frac{\gamma}{2}}\theta\|_{L^2}$ can be bounded respectively by 
\begin{align*}
\|\Lambda^{1+\frac{\gamma}{2}}u[\eta]\|_{L^4}\|\theta\|_{L^4}\le C\|\Lambda\eta\|_{L^2}\|\Lambda\theta\|_{L^2},
\end{align*}
and 
\begin{align*}
\|u[\eta]\|_{L^\infty}\|\Lambda^{1+\frac{\gamma}{2}}\theta\|_{L^2}\le C\|\Lambda \eta\|_{L^2}\|\Lambda^{1+\frac{\gamma}{2}}\theta\|_{L^2}.
\end{align*}
Therefore, we have
\begin{align*}
\Big|\intox u[\eta]\cdot\nabla\theta\Delta\eta\Big|\le C\Big(\|\Lambda\eta\|_{L^2}\|\Lambda\theta\|_{L^2}+\|\Lambda \eta\|_{L^2}\|\Lambda^{1+\frac{\gamma}{2}}\theta\|_{L^2}\Big)\|\Lambda^{1+\frac{\gamma}{2}}\eta\|_{L^2},
\end{align*} 
and by the similar argument,
\begin{align*}
\Big|\intox u[w]\cdot\nabla w\Delta\eta\Big|\le C\|\Lambda w\|_{L^2}\|\Lambda^{1+\frac{\gamma}{2}}w\|_{L^2}\|\Lambda^{1+\frac{\gamma}{2}}\eta\|_{L^2}.
\end{align*}
Dropping the non-negative term $\lambda\|\eta\|^2_{H^{\frac{3}{2}}}$, the identity \eqref{differential eqn for eta} then implies
\begin{align}\label{differential eqn for eta step 2}
&\frac{1}{2}\frac{d}{dt}\|\eta\|^2_{H^1}+\kappa\|\eta\|^2_{H^{1+\frac{\gamma}{2}}}\notag\\
&\le C\|\Lambda^{1+\frac{\gamma}{2}}\theta\|_{L^2}\|\nabla\eta\|^2_{L^2}+C\Big(\|\Lambda\theta\|_{L^2}+\|\Lambda^{1+\frac{\gamma}{2}}\theta\|_{L^2}\Big)\|\Lambda \eta\|_{L^2}\|\Lambda^{1+\frac{\gamma}{2}}\eta\|_{L^2}\notag\\
&\qquad+C\|\Lambda w\|_{L^2}\|\Lambda^{1+\frac{\gamma}{2}}w\|_{L^2}\|\Lambda^{1+\frac{\gamma}{2}}\eta\|_{L^2}.
\end{align}
On the other hand, we consider the function $w=w(t)$ satisfying
\begin{align}\label{eqn for w}
\left\{ \begin{array}{l}
\dt w+\lambda\mathcal{D}w+\kappa\Lambda^\gamma w+u[\varphi]\cdot\nabla w+u[w]\cdot\nabla\theta=0, \\
w(x,0)=\psi_0.
\end{array}\right.
\end{align}
Taking $L^2$-inner product of \eqref{eqn for w}$_1$ with $-\Delta w$,
\begin{align}\label{differential eqn for w}
\frac{1}{2}\frac{d}{dt}\|w\|^2_{H^1}+\lambda\|w\|^2_{H^{\frac{3}{2}}}+\kappa\|w\|^2_{H^{1+\frac{\gamma}{2}}}=\intox u[\varphi]\cdot\nabla w\Delta w+\intox u[w]\cdot\nabla\theta\Delta w.
\end{align}
Similar to the previous case for $\eta$, we have
\begin{align*}
\Big|\intox u[\varphi]\cdot\nabla w\Delta w\Big|\le C\|\Lambda^{1+\frac{\gamma}{2}}\varphi\|_{L^2}\|\nabla w\|^2_{L^2},
\end{align*}
and 
\begin{align*}
\Big|\intox u[w]\cdot\nabla\theta\Delta w\Big|\le C\Big(\|\Lambda w\|_{L^2}\|\Lambda\theta\|_{L^2}+\|\Lambda w\|_{L^2}\|\Lambda^{1+\frac{\gamma}{2}}\theta\|_{L^2}\Big)\|\Lambda^{1+\frac{\gamma}{2}}w\|_{L^2}.
\end{align*} 
Hence \eqref{differential eqn for w} implies
\begin{align}\label{differential eqn for w step 2}
&\frac{1}{2}\frac{d}{dt}\|w\|^2_{H^1}+\kappa\|w\|^2_{H^{1+\frac{\gamma}{2}}}\notag\\
&\le\frac{\kappa}{2}\|w\|^2_{H^{1+\frac{\gamma}{2}}}+C\Big[\|\Lambda^{1+\frac{\gamma}{2}}\varphi\|_{L^2}+\frac{1}{\kappa}(\|\Lambda\theta\|^2_{L^2}+\|\Lambda^{1+\frac{\gamma}{2}}\varphi\|^2_{L^2})\Big]\|w\|^2_{H^1}.
\end{align}
Since $\theta_0$, $\varphi_0\in\Gg$, $\theta$ and $\varphi$ both satisfy the bounds \eqref{uniform bound on theta from the attractor}-\eqref{uniform bound on time integral of theta from the attractor} with $M_{\Gg}$ being given by \eqref{def of constant M on global attractor}, namely,
\begin{align}\label{uniform bound on theta and varphi}
\|\theta\|_{H^{1+\frac{\gamma}{2}}}\le M_{\Gg},\qquad \|\varphi\|_{H^{1+\frac{\gamma}{2}}}\le M_{\Gg},
\end{align}
and
\begin{align}\label{uniform bound on time integral of theta and varphi}
\frac{1}{T}\int_0^{T}\|\theta(\cdot,\tau)\|_{H^{1+\gamma}}d\tau\le M_{\Gg},\qquad \frac{1}{T}\int_0^{T}\|\varphi(\cdot,\tau)\|_{H^{1+\gamma}}d\tau\le M_{\Gg},\qquad\forall T>0.
\end{align}
We apply the bounds \eqref{uniform bound on theta and varphi} on \eqref{differential eqn for w step 2} and use Gr\"{o}nwall's inequality to obtain
\begin{align}\label{H1 bound on w}
\|w(\cdot,t)\|^2_{H^1}\le\|\psi_0\|^2_{H^1}K(t,M_{\Gg}),\qquad\forall t\ge0,
\end{align}
as well as 
\begin{align}\label{time integral on H 1+gamma/2 norm of w}
\int_0^t\|w(\cdot,\tau)\|^2_{H^{1+\frac{\gamma}{2}}}d\tau\le\|\psi_0\|^2_{H^1}K(t,M_{\Gg}),\qquad\forall t\ge0.
\end{align}
where $K(t,M_{\Gg})$ is a positive function in $t$. We combine \eqref{differential eqn for eta step 2} with \eqref{uniform bound on theta and varphi} and \eqref{H1 bound on w} to get
\begin{align}\label{differential eqn for eta step 3}
\frac{d}{dt}\|\eta\|^2_{H^1}+\kappa\|\eta\|^2_{H^{1+\frac{\gamma}{2}}}\le CM_{\Gg}\|\eta\|^2_{H^1}+\frac{CM^2_{\Gg}}{\kappa}\|\eta\|^2_{H^1}+\frac{C}{\kappa}K(t,M_{\Gg})\|\psi_0\|^2_{H^1}\|w\|^2_{H^{1+\frac{\gamma}{2}}}.
\end{align}
Using Gr\"{o}nwall's inequality on \eqref{differential eqn for eta step 3} and recalling the fact that $\eta(0)=0$, we conclude that
\begin{align*}
\|\eta(\cdot,t)\|^2_{H^1}\le \tilde K(t,M_{\Gg})\|\psi_0\|^4_{H^1},
\end{align*}
where $\tilde K(t,M_{\Gg}):=\exp\Big((CM_{\Gg}+\frac{CM^2_{\Gg}}{\kappa})t\Big)\frac{C}{\kappa}K^2(t,M_{\Gg})$. Hence we prove that
\begin{align*}
\lim_{\varepsilon\to0}\left(\sup_{\theta_0,\varphi_0\in\Gg;0<\|\psi_0\|_{H^1}\le\varepsilon}\frac{\|\eta(\cdot,t)\|_{H^1}}{\|\psi_0\|_{H^1}}\right)\le \lim_{\varepsilon\to0}\tilde K(t,M_{\Gg})\varepsilon^2=0,
\end{align*}
and \eqref{condition 1 for uniform differentiable} follows. The bound \eqref{condition 2 for uniform differentiable} can be proved by performing similar analysis on $\psi$.

For any $t>0$ and $\theta_0\in\Gg$, in order to show that the linear operator $D\pin(t, \theta_0)$ is compact, it suffices to show that if $U_1$ is the unit ball in $H^1$, then $D\pin(t, \theta_0)U_1\subset H^{1+\frac{\gamma}{2}}$. Based on previous estimates, we readily obtain
\begin{align*}
\int_0^t\|\psi(\cdot,\tau)\|^2_{H^{1+\frac{\gamma}{2}}}d\tau\le K(t,M_{\Gg}),\qquad\forall t\ge0.
\end{align*}
By the mean value theorem, for $t>0$, there exists $\tau\in(0,t)$ such that
\begin{align}\label{bound on psi in H 1+gamma/2}
\|\psi(\cdot,\tau)\|^2_{H^{1+\frac{\gamma}{2}}}\le \frac{1}{t}K(t,M_{\Gg}).
\end{align}
We take the $L^2$-inner product of \eqref{linearised active scalar with gamma}$_1$ with $\Lambda^{2+\gamma}\psi$ and obtain
\begin{align*}
\frac{1}{2}\frac{d}{dt}\|\psi\|^2_{H^{1+\frac{\gamma}{2}}}+\lambda\|\psi\|^2_{H^2}+\kappa\|\psi\|^2_{H^{1+\gamma}}=\intox u[\theta]\cdot \nabla \psi\Lambda^{2+\gamma}\psi+\intox u[\psi]\cdot\nabla\theta\Lambda^{2+\gamma}\psi.
\end{align*}
Using the commutator estimate \eqref{commutator estimate} and the bound \eqref{bound on L4 using H1}, the term $\Big|\intox u[\theta]\cdot \nabla \psi\Lambda^{2+\gamma}\psi\Big|$ can be bounded by
\begin{align*}
\Big|\intox u[\theta]\cdot \nabla \psi\Lambda^{2+\gamma}\psi\Big|\le C\Big(\|\Lambda\theta\|_{L^2}\|\Lambda^{1+\frac{\gamma}{2}} \psi\|_{L^2}+\|\Lambda^\frac{\gamma}{2}\theta\|_{L^2}\|\Lambda^{1+\gamma} \psi\|_{L^2}\Big)\|\Lambda^{1+\frac{\gamma}{2}} \psi\|_{L^2},
\end{align*}
and similarly, we also have
\begin{align*}
\Big|\intox u[\psi]\cdot \nabla \theta\Lambda^{2+\gamma}\psi\Big|\le C\Big(\|\Lambda \psi\|_{L^2}\|\Lambda \theta\|_{L^2}+\|\Lambda \psi\|_{L^2}\|\Lambda^{1+\gamma} \theta\|_{L^2}\Big)\|\Lambda^{1+\gamma} \psi\|_{L^2}.
\end{align*}
Using Young's inequality, we obtain from \eqref{bound on psi in H 1+gamma/2} that
\begin{align}\label{differential eqn for psi higher order}
\frac{1}{2}\frac{d}{dt}\|\psi\|^2_{H^{1+\frac{\gamma}{2}}}+\frac{\kappa}{2}\|\psi\|^2_{H^{1+\gamma}}\le \frac{C}{\kappa}\Big(\|\Lambda\theta\|_{L^2}+\|\Lambda^\frac{\gamma}{2}\theta\|^2_{L^2}+\|\Lambda\theta\|^2_{L^2}+\|\Lambda^{1+\gamma}\theta\|^2_{L^2}\Big)\|\psi\|^2_{H^{1+\frac{\gamma}{2}}}
\end{align}
Integrating \eqref{differential eqn for psi higher order} from $\tau$ to $t$, applying Gr\"{o}nwall's inequality and using the bounds \eqref{uniform bound on time integral of theta and varphi} and \eqref{bound on psi in H 1+gamma/2}, we arrive at
\begin{align}\label{H 1+gamma/2 bound on psi}
\|\psi(\cdot,t)\|^2_{H^{1+\frac{\gamma}{2}}}&\le \|\psi(\cdot,\tau)\|^2_{H^{1+\frac{\gamma}{2}}}\exp\Big(\frac{C}{\kappa}\int_{\tau}^t\|\theta(\cdot,s)\|^2_{H^{1+\gamma}}ds\Big)\notag\\
&\le \frac{1}{t}K(t,M_{\Gg})\exp\Big(\frac{C}{\kappa}M_{\Gg}t\Big),
\end{align}
hence $\psi(t)\in H^{1+\frac{\gamma}{2}}$. The case when \eqref{conditions on gamma lambda>0} is in force can be handled similarly and we conclude the proof of lemma~\ref{uniform differentiable lemma}.
\end{proof}

Next we show that there is an $N$ such that volume elements which are carried by the flow of $\pin(t)\theta_0$, with $\theta_0\in\Gg$, decay exponentially for dimensions larger than $N$. We recall the following proposition for which the proof can be found in \cite{CF85}, \cite{CF88}.

\begin{proposition}\label{volume decay prop}
Consider $\theta_0\in\Gg$, and an initial orthogonal set of infinitesimal displacements $\{\psi_{1,0},\dots,\psi_{n,0}\}$ for some $n\ge1$. Suppose that $\psi_i$ obey the following equation:
\begin{align}
\dt\psi_i=A_{\theta_0}(t)[\psi_i],\qquad \psi_i(0)=\psi_{i,0},
\end{align}
for all $i\in\{1,\dots, n\}$ and $t\ge0$, where $A_{\theta_0}(t)$ is given by \eqref{def of operator A theta}. Then the volume elements
\begin{align*}
V_n(t)=\|\psi_1(t)\wedge\dots\wedge\psi_n(t)\|_{H^1}
\end{align*}
satsify
\begin{align*}
V_n(t)=V_n(0)\exp\left(\int_0^t \Tr(P_n(s)A_{\theta_0}(s))ds\right),
\end{align*}
where the orthogonal projection $P_n(s)$ is onto the linear span of $\{\psi_1(s),\dots,\psi_n(s)\}$ in the Hilbert space $H^1$, and $\Tr(P_n(s)A_{\theta})$ is defined by 
\[\Tr(P_n(s)A_{\theta}) = \sum_{j=1}^n\int_{\Omega}(-\Delta\varphi_j(s))A_{\theta}[\varphi_j(s)]dx\]
for $n\ge1$, with $\{\varphi_1(s),\dots,\varphi_n(s)\}$ an orthornormal set spanning the linear span of $\{\psi_1(s),\dots,\psi_n(s)\}$. If we define
\begin{align*}
\langle P_nA_{\theta_0}\rangle:=\limsup_{T\to\infty}\frac{1}{T}\int_0^T\Tr(P_n(t)A_{\theta_0}(t))dt,
\end{align*}
then we further obtain
\begin{align}\label{volume decay}
V_n(t)\le V_n(0)\left(t\sup_{\theta_0\in\Gg}\sup_{P_n(0)}\langle P_nA_{\theta_0}\rangle\right),\qquad \forall t\ge0,
\end{align}
where the supremum over $P_n(0)$ is a supremum over all choices of initial $n$ orthogonal set of infinitesimal displacements that we take around $\theta_0$.
\end{proposition}
Using Proposition~\ref{volume decay prop}, we show that the $n$-dimensional volume elements actually {\it decay exponentially in time} for $n$ is sufficiently large, which is based on the following lemma:

\begin{lemma}[Contractivity of large dimensional volume elements]\label{Contractivity of large dimensional volume elements lemma}
Let $\kappa>0$ be fixed, and assume that $\lambda$ and $\gamma$ satisfy \eqref{conditions on gamma}. There exists $N$ such that for any $\theta_0\in\Gg$ and any set of initial orthogonal displacements $\{\psi_{i,0}\}_{i=1}^n$, we have
\begin{align}\label{negativity of <Pn A theta>}
\langle P_nA_{\theta_0}\rangle<0,
\end{align}
whenever $n\ge N$. Here $N$ can be chosen explicitly from \eqref{choice of N} below. If the condition \eqref{conditions on gamma lambda>0} is in force instead, the inequality \eqref{negativity of <Pn A theta>} holds whenever $n\ge N$ with $N$ being chosen from \eqref{choice of N lambda>0} below.
\end{lemma}

\begin{proof}
Let $\xi\in H^1$ be arbitrary. For $\theta_0\in\Gg$, using the definition of $A_{\theta}$ in \eqref{def of operator A theta} and the fact that $\nabla\cdot u[\theta]=0$, we have
\begin{align*}
\intox\Lambda^2\xi A_{\theta}[\xi]dx\le-\kappa\|\xi\|^2_{H^{1+\frac{\gamma}{2}}}+\Big|\intox\partial_{x_k}u[\theta]\cdot\nabla\xi\partial_{x_k}\xi dx\Big|+\Big|\intox(u[\xi]\cdot\nabla\theta)\Lambda^2\xi dx\Big|. 
\end{align*}
Using the bound \eqref{bound on u in terms of nabla theta}, we readily have
\begin{align*}
\Big|\intox\partial_{x_k}u[\theta]\cdot\nabla\xi\partial_{x_k}\xi dx\Big|\le C\|\nabla u[\theta]\|_{L^\infty}\|\nabla \xi\|^2_{L^2}\le C\|\Lambda\theta\|_{L^2}\|\nabla \xi\|^2_{L^2}.
\end{align*}
And using the product estimate \eqref{product estimate} and the bounds \eqref{bound on L4 using H1}-\eqref{bound on u in terms of nabla theta}, we have
\begin{align*}
\Big|\intox(u[\xi]\cdot\nabla\theta)\Lambda^2\xi dx\Big|\le C\|\theta\|_{H^{1+\frac{\gamma}{2}}}\|\Lambda\xi\|_{L^2}\|\Lambda^{1+\frac{\gamma}{2}}\xi\|_{L^2}.
\end{align*}
Hence we deduce that
\begin{align}\label{bound on integrand in A theta}
\intox\Lambda^2\xi A_{\theta}[\xi]dx\le-\frac{\kappa}{2}\|\xi\|^2_{H^{1+\frac{\gamma}{2}}}+C\Big(\|\Lambda\theta\|_{L^2}+\frac{1}{\kappa}\|\theta\|^2_{H^{1+\frac{\gamma}{2}}}\Big)\|\xi\|^2_{H^1}.
\end{align}
On the other hand, for $T>0$, we have
\begin{align}\label{identity for integrand in A theta}
\frac{1}{T}\int_0^T\Tr(P_n(t)A_{\theta_0}(t))dt=\frac{1}{T}\int_0^T\sum_{j=1}^n\intox(\Lambda^2\varphi_j(t))A_{\theta}[\varphi_j(t)]dxdt.
\end{align}
We apply \eqref{bound on integrand in A theta} on \eqref{identity for integrand in A theta} and together with the bound \eqref{uniform bound on theta from the attractor} on $\theta$ to obtain
\begin{align*}
&\frac{1}{T}\int_0^T\Tr(P_n(t)A_{\theta_0}(t))dt\\
&\le -\frac{\kappa}{2T}\int_0^T\Tr(P_n(t)\Lambda^\gamma)+C(M_{\Gg}+M_{\Gg}^2)n\le -\frac{\kappa}{C}n^{1+\frac{\gamma}{d}}+C(M_{\Gg}+M_{\Gg}^2)n,
\end{align*}
where the last inequality follows from the fact that the eigenvalues $\{\lambda_{j}^{(\gamma)}\}$ of $\Lambda^\gamma$ obey the following estimate (see \cite[Theorem~1.1]{YY13}):
\begin{align*}
\sum_{j=1}^n\lambda_{j}^{(\gamma)}\ge\frac{1}{C}n^{1+\frac{\gamma}{d}},
\end{align*}
for some universal constant $C>0$ which depends only on $d$ and $\gamma$. We choose $N>0$ such that
\begin{align}\label{choice of N}
 -\frac{\kappa}{C}N^{1+\frac{\gamma}{d}}+C(M_{\Gg}+M_{\Gg}^2)N<0,
\end{align}
then \eqref{negativity of <Pn A theta>} holds whenever $n\ge N$. For the case when \eqref{conditions on gamma lambda>0} is in force, we make use of the estimate that
\begin{align*}
\intox\Lambda^2\xi A_{\theta}[\xi]dx\le-\lambda\|\xi\|^2_{H^{\frac{3}{2}}}+\Big|\intox\partial_{x_k}u[\theta]\cdot\nabla\xi\partial_{x_k}\xi dx\Big|+\Big|\intox(u[\xi]\cdot\nabla\theta)\Lambda^2\xi dx\Big|. 
\end{align*}
Then we choose $N>0$ such that
\begin{align}\label{choice of N lambda>0}
 -\frac{\lambda}{C}N^{1+\frac{1}{d}}+C(\bar{M}_{\Gg}+\bar{M}_{\Gg}^2)N<0,
\end{align}
then the same result holds whenever $n\ge N$.
\end{proof}

From the results obtained in Lemma~\ref{uniform differentiable lemma} and Lemma~\ref{Contractivity of large dimensional volume elements lemma}, together with Proposition~\ref{volume decay prop}, we obtain:
\begin{itemize}
\item the solution map $\pin(t)$ is uniform differentiable on $\Gg$;
\item the linearisation $D\pin(t, \theta_0)$ of $\pin(t)$ is compact;
\item the large-dimensional volume elements which are carried by the flow of $\pin(t)\theta_0$, with $\theta_0\in\Gg$, have exponential decay in time.
\end{itemize}
Therefore, following the lines of the argument in \cite[pp. 115--130, and Chapter 14]{CF88}, we can finally conclude that $\dimf(\Gg)$ is finite. The results can be summarised in the following corollary:

\begin{corollary}[Finite dimensionality of the attractor]\label{Finite dimensionality of the attractor corollary}
Let $N$ be as defined in Lemma~\ref{Contractivity of large dimensional volume elements lemma}. Then the fractal dimension of $\Gg$ is finite, and we have $\dimf(\Gg)\le N$.
\end{corollary}

\section{Applications to magneto-geostrophic equations}\label{Applications to magneto-geostrophic equations section}

\subsection{The MG equations in the class of drift-diffusion equations}\label{MG equations in the class of drift-diffusion equations}

We now apply our results claimed by Section~\ref{main results} to the magnetogeostrophic active scalar equation with the presence of damping operator $\lambda\mathcal{D}$. Specifically, we are interested in the following active scalar equation in the domain $\mathbb{T}^3\times(0,\infty)=[-\pi,\pi]^3\times(0,\infty)$ (with periodic boundary conditions):
\begin{align}
\label{MG active scalar} \left\{ \begin{array}{l}
\partial_t\theta+u\cdot\nabla\theta+\lambda\mathcal{D}\theta+\kappa\Lambda^{\gamma}\theta=S, \\
u=\mathcal{M}[\theta],\theta(x,0)=\theta_0(x)
\end{array}\right.
\end{align}
via a Fourier multiplier operator $\mathcal{M}$ which relates $u$ and $\theta$. More precisely,
\begin{align}\label{def of u by Fourier transform}
u_j=\mathcal{M}_j [\theta]=(\widehat{\mathcal{M}_j}\hat\theta)^\vee
\end{align}
for $j\in\{1,2,3\}$. The explicit expression for the components of $\widehat{\mathcal{M}}$ as functions of the Fourier variable $k=(k_1,k_2,k_3)\in\Z^3$ with $k_3\neq0$ are given by 
\begin{align}
\widehat{\mathcal{M}}_1(k)&=[k_2k_3|k|^2-k_1k_3(k_2^2+\nu|k|^4)]D(k)^{-1},\label{MG Fourier symbol_1 intro}\\
\widehat{\mathcal{M}}_2(k)&=[-k_1k_3|k|^2-k_2k_3(k_2^2+\nu|k|^4)]D(k)^{-1},\label{MG Fourier symbol_2 intro}\\
\widehat{\mathcal{M}}_3(k)&=[(k_1^2+k_2^2)(k_2^2+\nu|k|^4)]D(k)^{-1},\label{MG Fourier symbol_3 intro}
\end{align}
where
\begin{align*}
D(k)=|k|^2k_3^2+(k_2^2+\nu|k|^4)^2.
\end{align*}
Since for self-consistency of the model, we assume that $\theta$ and $u$ have zero vertical mean, and we take $\widehat{\mathcal{M}}_j(k) = 0$ on $\{k_3=0\}$ for all $j=1,2,3$. We write $\mathcal{M}_j=\partial_{x_i}T_{ij}$ for convenience and take the non-dimensional viscosity $\nu=1$ unless otherwise simplified. The system \eqref{MG active scalar}-\eqref{MG Fourier symbol_3 intro} can be rigorously obtained from the incompressible MHD via the postulates in \cite{ML94}, and the detailed derivations can be found in \cite{FV11a, FRV12}. We also refer to \eqref{MG active scalar} as the {\it fractional} MG equation when $\gamma\in(0,2]$.

To apply the results from Section~\ref{main results}, it suffices to show that the linear operator $T_{ij}$ satisfies the assumptions A1 to A4 given in Section~\ref{introduction}. 
\begin{proposition}\label{assumption checked}
We define $T_{ij}$ by $\mathcal{M}_j=\partial_{x_i}T_{ij}$. Then $T_{ij}$ satisfy the assumptions A1, A2, A4 in Section~\ref{introduction} and the assumption A3' in Section~\ref{Existence and convergence of Hs-solutions section}.
\end{proposition}
\begin{proof}
The details for the proof can be found in \cite[Lemma~5.1--5.2]{FS18} and from the discussion in \cite[Section 4]{FV11a} as well as \cite[Section 7]{FS21}. We omit the details here. 
\end{proof} 
In view of Proposition~\ref{assumption checked}, the abstract Theorem~\ref{absence of anomalous dissipation thm}, Theorem~\ref{Hs convergence theorem} and Theorem~\ref{existence of global attractor theorem} can then be applied to the MG equations \eqref{MG active scalar}. More precisely, we have

\begin{theorem}[Absence of anomalous dissipation of energy as $\kappa\to0$ for MG equations]\label{absence of anomalous dissipation MG}
Let $\theta_0,S\in L^\infty$, and assume that $\lambda>0$ and $\gamma\in(0,2]$ be fixed. For each $\kappa\ge0$, there exists a unique solution $\theta=\theta^{(\kappa)}(x,t)$ to \eqref{MG active scalar} satisfying
\begin{align*}
\lim_{\kappa\to0}\left(\limsup_{t\to\infty}\int_{0}^{t}\|\nabla\theta^{(\kappa)}(\cdot,s)\|_{L^2}^2ds\right)=0.
\end{align*}
\end{theorem}

\begin{theorem}[$H^s$-convergence as $\lambda\rightarrow0$ for MG equations]\label{Hs convergence MG}
Let $\kappa>0$ be given as in \eqref{MG active scalar}, and let $\theta_0,S\in C^\infty$ be the initial datum and forcing term respectively with zero mean. If $\theta^{(\lambda)}$ and $\theta^{(0)}$ are smooth solutions to \eqref{MG active scalar} for $\lambda>0$ and $\lambda=0$ respectively, then 
\begin{align*} 
\lim_{\lambda\rightarrow0}\|(\theta^{(\lambda)}-\theta^{(0)})(\cdot,t)\|_{H^s}=0,
\end{align*}
for all $s\ge0$ and $t\ge0$.
\end{theorem}

\begin{theorem}[Existence of global attractors for MG equations]\label{existence of global attractor MG theorem}
Let $S\in L^\infty\cap H^1$. For $\kappa>0$, let $\pin(t)$ be solution operator for the initial value problem \eqref{MG active scalar} via \eqref{def of solution map nu>0 and kappa>0}. Then the solution map $\pin(t):H^1\to H^1$ associated to \eqref{abstract active scalar eqn} possesses a unique global attractor $\Gg$ for all $\lambda\ge0$. In particular, if $\lambda$ and $\gamma$ satisfy either \eqref{conditions on gamma} or \eqref{conditions on gamma lambda>0}, then the global attractor $\Gg$ of $\pin(t)$ enjoys the following properties:
\begin{itemize}
\item $\Gg$ is fully invariant, namely
\begin{align*}
\pin(t)\Gg=\Gg,\qquad \forall t\ge0.
\end{align*}
\item $\Gg$ is maximal in the class of $H^1$-bounded invariant sets.
\item $\Gg$ has finite fractal dimension.
\end{itemize}
\end{theorem}

In the coming subsection, we will further address the limiting properties of $\Gg$ which are related to the MG equation when $\lambda=0$.

\subsection{Behaviour of global attractors for varying $\lambda\ge0$}\label{Behaviour of global attractors section}

In the work \cite{FS21}, when $\kappa>0$ and $\lambda=0$, the authors proved the existence of a compact global attractor $\A$ in $H^1(\mathbb{T}^3)$ for the MG equations \eqref{MG active scalar} with $S\in L^\infty\cap H^1$. More precisely, $\A$ is the global attractor generated by the solution map $\pio$ via
\begin{align}\label{def of solution map nu=0 and kappa>0}
\pio(t): H^1\to H^1,\qquad \pio(t)\theta_0=\theta(\cdot,t),\qquad t\ge0,
\end{align}
where $\theta$ is the solution to the MG equation with $\theta(\cdot,0)=\theta_0$. In this subsection, we obtain results when $\lambda$ is varying, which can be summarised in the following theorem:

\begin{theorem}[Upper semicontinuity of global attractors at $\lambda\ge0$]\label{varying lambda theorem}
Let $\kappa>0$ be fixed in \eqref{MG active scalar}. Let $\lambda_0\ge0$ be arbitrary, then the collection $\dis\{\Gg\}_{\lambda\ge0}$ is {\it upper semicontinuous} at $\lambda_0$ in the following sense: 
\begin{align}\label{upper semi-continuity at fixed lambda>0}
\mbox{$\dis\sup_{\phi\in\Gg}\inf_{\psi\in\Ggo}\|\phi-\psi\|_{H^1}\rightarrow0$ as $\lambda\rightarrow\lambda_0$.}
\end{align}
In particular, if $\Gg$ are the global attractors for the MG equations \eqref{MG active scalar} as obtained by Theorem~\ref{existence of global attractor MG theorem}, then $\Gg$ and $\A$ satisfy
\begin{align}\label{upper semi-continuity at lambda=0}
\mbox{$\dis\sup_{\phi\in\Gg}\inf_{\psi\in\A}\|\phi-\psi\|_{H^1}\rightarrow0$ as $\lambda\rightarrow0$.}
\end{align}
\end{theorem}

\begin{remark}
Here are some relevant remarks regarding Theorem~\ref{varying lambda theorem}:
\begin{itemize}
\item By the maximality of global attractors, we readily have $\Gg|_{\lambda=0}=\A$. Hence it suffices to show \eqref{upper semi-continuity at fixed lambda>0} as \eqref{upper semi-continuity at lambda=0} follows immediately.
\item Although $\Gg$ has finite fractal dimension for all $\lambda>0$ and $\gamma\in(0,2]$, it is unknown whether $\A$ has finite fractal dimension for the case when $\gamma\in(0,1)$.
\item It is also worth pointing out that the $H^1$-convergence result \eqref{upper semi-continuity at fixed lambda>0} is stronger than the one obtained in \cite[Theorem~7.5]{FS21} which showed the $L^2$-convergence for varying $\nu$ in the system \eqref{MG active scalar}-\eqref{MG Fourier symbol_3 intro}.
\end{itemize}
\end{remark}
To obtain the convergence result claimed by Theorem~\ref{varying lambda theorem}, we need to prove that
\begin{itemize}
\item[L1.] there is a compact subset $\U$ of $H^1$ such that $\Gg\subset\U$ for every $\lambda\ge0$; and
\item[L2.] for $t > 0$, $\pin(t)\theta_0$ is continuous in $\lambda\in[0,\infty)$, uniformly for $\theta_0$ in compact subsets of $H^1$.
\end{itemize}
Once the conditions L1 and L2 are fufilled, we can apply the result from \cite{HOR15} to conclude that \eqref{upper semi-continuity at fixed lambda>0} holds as well. 

To show that condition L1 holds, as claimed by Lemma~\ref{Existence of an H 1+gamma/2 absorbing set lemma}, there exists a constant $R_{1+\frac{\gamma}{2}}\ge1$ which depends on $\|S\|_{L\infty\cap H^1}$, $\kappa$, $\gamma$ such that the set $\U:=B_{1+\frac{\gamma}{2}}$ where
\begin{align*}
B_{1+\frac{\gamma}{2}}=\left\{\phi\in H^{1+\frac{\gamma}{2}}:\|\phi\|_{H^{1+\frac{\gamma}{2}}}\le R_{1+\frac{\gamma}{2}}\right\}
\end{align*}
enjoys the following properties:
\begin{itemize}
\item $B_{1+\frac{\gamma}{2}}$ is a compact set in $H^1$ which depends only on $\|S\|_{L\infty\cap H^1}$, $\kappa$, $\gamma$;
\item $\Gg\subset B_{1+\frac{\gamma}{2}}$ for all $\lambda\ge0$.
\end{itemize}

On the other hand, in order to prove that condition L2 holds as well, we need some higher order bounds on $\theta$ that are uniform in $\lambda$. The following lemma gives the necessary $H^1$-estimates on $\theta$ which are uniform in $\lambda$.

\begin{lemma}[Uniform $H^1$-bound on $\theta$]
We fix $\kappa>0$ and define $\Uc=\{\phi\in H^1:\|\phi\|^2_{H^1}\le R_{\Uc}\}$ where $R_{\Uc}>0$. For any $\gamma\in(0,2]$, $\theta_0\in\Uc$ and $\lambda\ge0$, if $\thetalambda(t)=\pin(t)\theta_0$, then $\thetalambda(t)$ satisfies
\begin{align}\label{H1 estimate in compact set of H1}
\sup_{0\le \tau\le t}\|\thetalambda(\cdot,\tau)\|^2_{H^1}+\kappa\int_0^t\|\thetalambda(\cdot,\tau)\|_{H^{1+\frac{\gamma}{2}}}^2d\tau\le M_*(t),\qquad\forall t>0,
\end{align}
where $M_*(t)$ is a positive function in $t$ which depends only on $t$, $\kappa$, $\|S\|_{H^1}$ and $R_{\Uc}$.
\end{lemma}

\begin{proof}
We take $\theta=\thetalambda$ and $u=u[\thetalambda]$ for simplicity. Recall \eqref{ineq for nabla u with D[nabla theta] final} from the proof of Lemma~\ref{Existence of an H1 absorbing set lemma} that
\begin{align}\label{ineq for nabla u with D[nabla theta] final uniform in lambda}
\frac{d}{dt}\|\theta\|^2_{H^1}+\frac{\kappa}{2}\|\theta\|^2_{H^{1+\frac{\gamma}{2}}}
\le\Big(\frac{4C}{\kappa}\Big)^\frac{2-2\alpha}{\gamma}(2C_\alpha)^2\K^{1+\frac{2(1-\alpha)}{\gamma}}+\frac{8}{c_{\gamma,3}\kappa}\|S\|^2_{H^1}+\frac{\kappa}{4}\|\theta\|^2_{H^{1+\frac{\gamma}{2}}},
\end{align}
where $C$, $K_\infty$, $C_\alpha$, $c_{\gamma,3}$ are all constant that are independent of $\lambda$. Upon applying Gr\"{o}nwall's inequality on \eqref{ineq for nabla u with D[nabla theta] final uniform in lambda}, the bound \eqref{H1 estimate in compact set of H1} holds for some $M_*(t)$ which depends on the terms appeared on the right side of \eqref{ineq for nabla u with D[nabla theta] final uniform in lambda} but is independent of $\lambda$. 
\end{proof}

\begin{remark}
Replacing $\gamma$ by $2-\gamma$, for any $\gamma\in(0,2]$, we also have the following uniform bound on $\thetalambda$:
\begin{align}\label{H1 estimate in compact set of H1 revised}
\sup_{0\le \tau\le t}\|\thetalambda(\cdot,\tau)\|^2_{H^1}+\kappa\int_0^t\|\thetalambda(\cdot,\tau)\|_{H^{2-\frac{\gamma}{2}}}^2d\tau\le M_*(t),\qquad\forall t>0.
\end{align}
The uniform-in-$\lambda$ bounds  \eqref{H1 estimate in compact set of H1} and \eqref{H1 estimate in compact set of H1 revised} are crucial for controlling the $H^1$-norm of $\pin(t)\theta_0$; see the proof of Lemma~\ref{continuity in nu lemma} for details.
\end{remark}

We are now ready to state and prove the following lemma which gives the continuity of $\pin$ in $\lambda\in[0,\infty)$.

\begin{lemma}\label{continuity in nu lemma}
For each $t > 0$, $\pin(t)\theta_0$ is continuous in $\lambda\in[0,\infty)$, uniformly for $\theta_0$ in compact subsets of $H^1$.
\end{lemma}

\begin{proof}
Given a compact set $\Uc$ in $H^1$, we choose $R_\Uc>0$ such that $\Uc\subset\{\phi\in H^1:\|\phi\|^2_{H^1}\le R_{\Uc}\}$. For each $\theta_0\in\Uc$ and $\lambda_1$, $\lambda_2\ge0$, we define 
$$\theta^{(\lambda_i)}(t)=\pi^{\lambda_i}(t)\theta_0,\qquad i\in\{1,2\}.$$
We write $\phi=\theta^{(\lambda_1)}-\theta^{(\lambda_2)}$, then $\phi$ satisfies $\phi(\cdot,0)=0$ and 
\begin{align}\label{differential eqn for the difference in nu}
\dt \phi+(u^{(\lambda_1)}-u^{(\lambda_2)})\cdot\nabla\theta^{\lambda_2}+u^{(\lambda_1)}\cdot\nabla \phi+\lambda_2\mathcal{D}\phi+(\lambda_1-\lambda_2)\mathcal{D}\theta^{\lambda_1}+\kappa\lambda^\gamma \phi=0,
\end{align}
where $u^{(\lambda_i)}=u[\theta^{(\lambda_i)}]$ for $i=1,2$. Multiply \eqref{differential eqn for the difference in nu} by $-\Delta\phi$ and integrate,
\begin{align}\label{integral eqn for the difference in lambda}
&\frac{1}{2}\frac{d}{dt}\|\phi\|^2_{H^1}+\frac{\lambda_2}{2}\|\phi\|^2_{H^\frac{3}{2}}+\frac{\kappa}{2}\|\phi\|^2_{H^{1+\frac{\gamma}{2}}}\notag\\
&\le\Big|\intox (u^{(\lambda_1)}-u^{(\lambda_2)})\cdot\nabla\theta^{\lambda_2}\Delta \phi\Big|+\Big|(\lambda_1-\lambda_2)\intox\mathcal{D}\theta^{\lambda_1}\cdot\Lambda^{1+\frac{\gamma}{2}}\phi\Big|+\Big|\intox u^{(\lambda_1)}\cdot\nabla \phi\Delta \phi\Big|.
\end{align}

Upon integrating by part and exploiting the fact that $\nabla\cdot u^{(\lambda_1)}=0$, the last integral of \eqref{integral eqn for the difference in lambda} vanishes. For the first term on the right side of \eqref{integral eqn for the difference in lambda}, by applying Fourier transform and Parseval's theorem, we readily obtain
\begin{align*}
\Big|\intox (u^{(\lambda_1)}-u^{(\lambda_2)})\cdot\nabla\theta^{\lambda_2}\Delta \phi\Big|=\Big|\intox \Lambda^{1-\frac{\gamma}{2}}[(u^{(\lambda_1)}-u^{(\lambda_2)})\cdot\nabla\theta^{\lambda_2}]\Lambda^{1+\frac{\gamma}{2}} \phi\Big|.
\end{align*}
Hence by H\"{o}lder's inequality, we have
\begin{align*}
\Big|\intox (u^{(\lambda_1)}-u^{(\lambda_2)})\cdot\nabla\theta^{\lambda_2}\Delta \phi\Big|\le\|\Lambda^{1-\frac{\gamma}{2}}[(u^{(\lambda_1)}-u^{(\lambda_2)})\cdot\nabla\theta^{\lambda_2}]\|_{L^2}\|\Lambda^{1+\frac{\gamma}{2}} \phi\|_{L^2}
\end{align*}
Using \eqref{commutator estimate}, we further obtain
\begin{align*}
&\|\Lambda^{1-\frac{\gamma}{2}}[(u^{(\lambda_1)}-u^{(\lambda_2)})\cdot\nabla\theta^{\lambda_2}]\|_{L^2}\\
&\le\|\Lambda^{1-\frac{\gamma}{2}}[(u^{(\lambda_1)}-u^{(\lambda_2)})\cdot\nabla\theta^{\lambda_2}]-(u^{(\lambda_1)}-u^{(\lambda_2)})\cdot\Lambda^{1-\frac{\gamma}{2}}\nabla\theta^{\lambda_2}\|_{L^2}\\
&\qquad\qquad\qquad+\|(u^{(\lambda_1)}-u^{(\lambda_2)})\cdot\Lambda^{1-\frac{\gamma}{2}}\nabla\theta^{\lambda_2}\|_{L^2}\\
&\le C\Big(\|u^{(\lambda_1)}-u^{(\lambda_2)}\|_{L^\infty}\|\Lambda^{2-\frac{\gamma}{2}}\theta^{\lambda_2}\|_{L^2}+\|\Lambda^{1-\frac{\gamma}{2}}(u^{(\lambda_1)}-u^{(\lambda_2)})\|_{L^\infty}\|\nabla\theta^{\lambda_2}\|_{L^2}\Big).
\end{align*}
To bound the term $\|\Lambda^{1-\frac{\gamma}{2}}(u^{(\lambda_1)}-u^{(\lambda_2)})\|_{L^\infty}$, using \eqref{two order smoothing for u when nu>0} and \eqref{L infty bound 2}-\eqref{Kondrachov embedding theorem}, for $q\in(3,6)$, we have
\begin{align*}
\|\Lambda^{1-\frac{\gamma}{2}}(u^{(\lambda_1)}-u^{(\lambda_2)})\|_{L^\infty}&\le C\|\Lambda^{1-\frac{\gamma}{2}}(u^{(\lambda_1)}-u^{(\lambda_2)})\|_{W^{1,q}}\\
&\le C\|\Lambda^{1-\frac{\gamma}{2}}(u^{(\lambda_1)}-u^{(\lambda_2)})\|_{W^{2,2}}\le C\|\phi\|_{H^1}.
\end{align*}
Similarly, we also have $\|u^{(\lambda_1)}-u^{(\lambda_2)}\|_{L^\infty}\le C\|\phi\|_{H^1}$ and hence we obtain
\begin{align*}
\|\Lambda^{1-\frac{\gamma}{2}}[(u^{(\lambda_1)}-u^{(\lambda_2)})\cdot\nabla\theta^{\lambda_2}]\|_{L^2}\le C\|\phi\|_{H^1}\|\theta^{\lambda_2}\|_{H^{2-\frac{\gamma}{2}}}.
\end{align*}
We deduce from \eqref{integral eqn for the difference in lambda} that
\begin{align}\label{integral eqn for the difference in lambda step 2}
\frac{d}{dt}\|\phi\|^2_{H^1}+\frac{\kappa}{2}\|\phi\|^2_{H^{1+\frac{\gamma}{2}}}\le C\|\phi\|_{H^1}^2\|\theta^{\lambda_2}\|_{H^{2-\frac{\gamma}{2}}}^2+C|\lambda_1-\lambda_2|^2\|\theta^{\lambda_1}\|^2_{H^\frac{3}{2}}.
\end{align}
By the bounds \eqref{H1 estimate in compact set of H1} and \eqref{H1 estimate in compact set of H1 revised}, the integrals $\dis\int_0^t \|\theta^{\lambda_2}(\cdot,\tau)\|_{H^{2-\frac{\gamma}{2}}}^2d\tau$ and \newline $\dis\int_0^t \|\theta^{\lambda_1}(\cdot,\tau)\|_{H^{\frac{3}{2}}}^2d\tau$ are bounded independent of $\lambda_1$ and $\lambda_2$. By applying Gr\"{o}nwall's inequality on \eqref{integral eqn for the difference in lambda step 2}, we conclude that there exists a positive function $C_*(t)$ which depends only on $t$, $\kappa$, $\|S\|_{H^1}$ and $R_{\Uc}$ such that
\begin{align}\label{continuity of pin in nu}
\|(\theta^{(\lambda_1)}-\theta^{(\lambda_2)})(\cdot,t)\|^2_{H^1}=\|\phi(t)\|^2_{H^1}\le C_*(t)|\lambda_1-\lambda_2|^2,\qquad\forall t>0,
\end{align}
and \eqref{continuity of pin in nu} implies $\pin(t)\theta_0$ is continuous in $\lambda\in[0,\infty)$ uniformly for $\theta_0$ in $\Uc$.
\end{proof}

\begin{proof}[Proof of Theorem~\ref{varying lambda theorem}]
By choosing $\U=B_{1+\frac{\gamma}{2}}$ and applying Lemma~\ref{continuity in nu lemma}, the conditions L1 and L2 as stated at the beginning of this subsection follow immediately and we conclude that \eqref{upper semi-continuity at fixed lambda>0} holds. We finish the proof of Theorem~\ref{varying lambda theorem}.
\end{proof}


\subsection*{Acknowledgment} 
A. Suen is supported by Hong Kong General Research Fund (GRF) grant project number 18300821.

\bibliographystyle{amsalpha}

\bibliography{References_for_MG_damping}

\end{document}